\documentclass[reqno,11pt]{amsart}
\DeclareMathOperator*{\esssup}{\mathrm{ess\,sup}}

\usepackage{amssymb,amsthm}
\usepackage{amsmath}
\usepackage{amsfonts}
\usepackage{dsfont}
\usepackage{hyperref}

\usepackage[a4paper,  margin=2cm]{geometry}

\usepackage[active]{srcltx}
\makeatletter\@addtoreset{equation}{section}\makeatother

\newtheorem{theorem}{Theorem}[section]
\newtheorem{corollary}[theorem]{Corollary}
\newtheorem{lemma}[theorem]{Lemma}

\newtheorem{proposition}[theorem]{Proposition}
\newtheorem{assumption}[theorem]{Assumption}
\newtheorem{definition}[theorem]{Definition}
\newtheorem{remark}[theorem]{Remark}
\numberwithin{equation}{section}

\title[A Markov process for a particle system]{A Markov process for an infinite interacting particle system in the
continuum}

\author{ Yuri  Kozitsky}

\address{Instytut Matematyki, Uniwersytet Marii Curie-Sk{\l}odowskiej, 20-031 Lublin, Poland}
\email{jkozi@hektor.umcs.lublin.pl}

\author{Michael R\"ockner}

\address{Fakult\"at f\"ur Mathematik, Universit\"at Bielefeld, Bielefeld, Germany}

\email{roeckner@math.uni-bielefeld.de}

\begin{document}


\begin{abstract}

An infinite system of point particles placed in $\mathds{R}^d$ is
studied. Its constituents perform random jumps (walks) with mutual
repulsion described by a translation-invariant jump kernel and
interaction potential, respectively. The pure states of the system
are locally finite subsets of $\mathds{R}^d$, which can also be
interpreted as locally finite Radon measures. The set of all such
measures $\Gamma$ is equipped with the vague topology and the
corresponding Borel $\sigma$-field. For a special class
$\mathcal{P}_{\rm exp}$ of (sub-Poissonian) probability measures on
$\Gamma$, we prove the existence of a unique family $\{P_{t,\mu}:
t\geq 0, \ \mu \in \mathcal{P}_{\rm exp}\}$  of probability measures
on the space of cadlag paths with values in $\Gamma$ that solves a
restricted initial-value martingale problem for the mentioned
system. Thereby, a Markov process with cadlag paths is specified
which describes the stochastic dynamics of this particle system.

\end{abstract}

\maketitle

\tableofcontents

\section{Introduction}

As a challenging object of probability theory, measure-valued Markov
processes  attract considerable attention. They have also become
popular due to applications in mathematical physics, biology,
ecology, etc. Among such applications one can distinguish those
describing stochastic evolution of infinite systems of point
particles dwelling in a continuous habitat, e.g., $\mathds{R}^d$. In
this case, as the state space of the system is taken the set of all
locally finite configurations of particles, which can also be
interpreted as counting Radon measures. For finite particle systems,
the construction of the corresponding Markov processes is now quite
standard. For infinite systems, however, the list of results reduces
mostly to those describing free (noninteracting) systems \cite{KLR},
conservative diffusions with invariant Gibbs measures \cite{AKR}, or
birth-and-death dynamics with generators obeying essential
restrictions \cite{GK,holley,skor,penrose}. In this context, one can
 also mention models with interactions of Curie-Weiss (mean-field) type, e.g., \cite{Mel}, where one starts with a system of $N$ particles
interacting with a uniform strength proportional to $1/N$, and then
passes to the limit $N\to +\infty$.

In the present paper, we prove the existence and uniqueness of a
Markov process with cadlag paths for an infinite system of point
particles performing random jumps (walks) in $\mathds{R}^d$ with
mutual repulsion, which appears to be the first result of this kind
known in the literature. The starting point of our construction is
the configuration space $\Gamma$. As in \cite{Lenard}, by a
configuration $\gamma$ we mean a finite or countably infinite,
unordered system of points placed in $\mathds{R}^d$, in which
several points may have the same location. Configurations are
supposed to be \emph{locally finite}, which means that each compact
$\Lambda\subset \mathds{R}^d$ contains  a finite number of elements
of a given ${\gamma}\in {\Gamma}$.  The set $\Gamma$ is equipped
with the vague (weak-hash) topology -- the weakest topology that
makes continuous all the maps $\gamma \mapsto \sum_{x\in \gamma}
g(x)$, $g\in C_{\rm cs}(\mathds{R}^d)$, where $C_{\rm
cs}(\mathds{R}^d)$ denotes the set of all compactly supported
continuous functions $g:\mathds{R}^d\to \mathds{R}$. Here by writing
$\sum_{x\in \gamma}g(x)$ we understand $\sum_{i} g(x_i)$ for a
certain enumeration of the elements of $\gamma$. Clearly, such sums
are independent of the enumeration choice, see \cite{Lenard}. The
vague topology is separable and consistent with a complete metric,
i.e., is metrizable in such a way that the corresponding metric
space is complete. Then the states of the considered system are
probability measures on $\Gamma$, the set of which is denoted by
$\mathcal{P}(\Gamma)$. The point states $\gamma$ are associated to
the Dirac measures $\delta_\gamma$. The evolution of the system
which we consider is described by the (backward) Kolmogorov equation
\begin{equation}
  \label{I4}
 \frac{d}{dt} F_t  = L F_t,
\end{equation}
where $F_t:\Gamma\to \mathds{R}$, $t\geq 0$, are test functions and
\begin{eqnarray}
  \label{I5}
  (LF)(\gamma) & = & \sum_{x\in \gamma}\int_{\mathds{R}^d}   c(x,y;\gamma) \left[F(\gamma \setminus x \cup y) - F(\gamma) \right] d
  y, \\[.2cm] \nonumber
c(x,y; \gamma) & = &
  a(x-y) \exp\left(- \sum_{z\in \gamma\setminus x} \phi(z-y)
  \right).
\end{eqnarray}
Here --  and in sequel in similar expressions -- by writing
$\gamma\setminus x$, $x\in \gamma$, or $\gamma \cup x$, $x\in
\mathds{R}^d$, we mean $\gamma\setminus \{x\}$ and $\gamma\cup
\{x\}$, respectively,  i.e., $x$ is considered as the singleton
$\{x\}$.

The model specified by (\ref{I5}) presents an infinite collection of
point particles performing random walks (jumps) over $\mathds{R}^d$,
such that  the probability that the particle located at a given
$x\in \gamma$ changes instantly its position to $y\in \mathds{R}^d$
at time $t$ is $1 - \exp( - t c(x,y;\gamma))$. This probability is
asymmetric in $x$ and $y$ and is state dependent. This means that
the remaining particles prevent the one located at $x\in \gamma$
from jumping to $y$ -- by diminishing the jump kernel -- if the
target point is `close' to $\gamma\setminus x$. The diminishing
factor $\exp\left(- \sum_{z\in \gamma\setminus x} \phi(z-y)
  \right)$ is independent of $x$.
 Originally, models of this kind were introduced
and (heuristically) studied in physics \cite{Kawasaki}, where they
are known under a common name {\it Kawasaki model}. In the rigorous
setting, the stochastic dynamics of the model described by
(\ref{I4}), (\ref{I5}) were studied in \cite{asia} (see also
\cite{Berns} for preliminary results). In \cite{asia},  for a class
of states $\mathcal{P}_{\rm exp} \subset \mathcal{P}(\Gamma)$ --
defined by a certain analytic condition --  and each $\mu_0 \in
\mathcal{P}_{\rm exp}$, there was constructed a map $[0,+\infty) \ni
t \mapsto \mu_t \in \mathcal{P}_{\rm exp}$ that can be interpreted
as the evolution of states described by (\ref{I4}). In the present
work, we construct a Markov process with cadlag paths  such  that
the mentioned $\mu_t$ is its law at time $t$. Let us outline now
some of the aspects of this construction. As we show here, for a
sufficiently large set of functions $F:\Gamma\to \mathds{R}$, the
map $[0,+\infty) \ni t \mapsto \mu_t \in \mathcal{P}_{\rm exp}$
constructed in \cite{asia} is the unique (in the set of all
measures) solution of the Fokker-Planck equation
\begin{equation}
  \label{A1}
  \mu_t (F) = \mu_s (F) + \int_s^t \mu_u (LF) d
  u , \qquad \mu(F) := \int F d \mu,
\end{equation}
holding for all $0\leq s < t<\infty$, see \cite{FKP} for a general
theory of the equations of this kind. Unfortunately, the Dirac
measure $\delta_\gamma$ is not in $\mathcal{P}_{\rm exp}$ for any
$\gamma\in \Gamma$. Therefore, one cannot directly construct a
transition function (and hence the corresponding Markov process)
just by setting $\mu_0 = \delta_\gamma$. In view of this,  we take a
version of the martingale approach suggested in \cite{Stroock}, see
also \cite[Sect. 5.1]{Dawson}, \cite[Chapter 4]{EK}, and proceed as
follows. When dealing with measures $\mu\in \mathcal{P}_{\rm exp}$,
it is natural to use a subset $\Gamma_*\subset \Gamma$ such that
$\mu(\Gamma_*)=1$ for all $\mu \in \mathcal{P}_{\rm exp}$. We define
it by means of a positive continuous function $\psi:\mathds{R}^d \to
\mathds{R}$, chosen in such a way that $\Psi(\gamma):=\sum_{x\in
\gamma}\psi(x)$ be $\mu$-integrable for each $\mu\in
\mathcal{P}_{\rm exp}$. Thereby, we set $\Gamma_*=\{\gamma:
\Psi(\gamma)<\infty\}$, and equip it with the weakest topology that
makes continuous all the maps $\gamma \mapsto \sum_{x\in \gamma}
g(x) \psi(x)$, $g\in C_{\rm b}(\mathds{R}^d)$, where the latter is
the set of all bounded continuous functions. This topology makes
$\Gamma_*$ a Polish space, continuously embedded in $\Gamma$. Then
the measures of interest are redefined as measures on $\Gamma_*$.
 To construct
the process in question, we use spaces of cadlag maps
$[s,+\infty)\ni t \mapsto \gamma_t \in {\Gamma}_*$, $s\geq 0$,
denoted by $\mathfrak{D}_{[s,+\infty)}({\Gamma}_*)$, equipped with
the Skorohod metric, see \cite[page 118]{EK}, constructed with the
help of a complete metric of ${\Gamma}_*$. The principal result of
this work  (Theorem \ref{1tm}) can be characterized as follows. We
prove that there exists a family of probability measures,
$\{P_{s,\mu}: s \geq 0, \ \mu \in \mathcal{P}_{\rm exp}\}$, on
$\mathfrak{D}_{\mathds{R}_{+}}({\Gamma}_*)$ which is a unique
solution of the restricted initial-value martingale problem
corresponding to (\ref{I5}). For such measures, their
one-dimensional marginals belong to $\mathcal{P}_{\rm exp}$ and
satisfy the corresponding version of the Fokker-Planck equation
(\ref{A1}), i.e., they coincide with the measures $\mu_t$
constructed in \cite{asia}. By this we prove the existence of a
unique Markov process with cadlag paths taking values in
${\Gamma}_*$. Finally, we prove that with probability one the
constructed process takes values in the subset of $\Gamma_*$
consisting of simple configurations.

In \cite{asia1}, there was studied a model in which point particles
of two types perform random jumps over $\mathds{R}^d$. Their common
dynamics are described by the corresponding analog of the Kolmogorov
operator (\ref{I5}) in which particles of different types repel each
other, whereas those of the same type do not interact. This kind of
interaction is typical for the classical Widom-Rowlinson model (see
\cite{WR} and the literature quoted therein), for which the states
of thermal equilibrium can be multiple \cite{WR,KK}. The latter fact
ought to have an essential impact on the stochastic dynamics of such
models, cf. \cite{Kissel}, which further stimulates constructing
Markov processes here. The results of \cite{asia1} are pretty
analogous to those of \cite{asia}, which means that -- after proper
modification -- the approach developed in the present work can be
applied also to the model  of \cite{asia1}, which we will realize in
a subsequent paper.

The rest of the paper is organized as follows. In Sect. 2, we
introduce all necessary facts and notions, among which are
sub-Poissonian measures and the above-mentioned set $\Gamma_*\subset
\Gamma$. Here we also introduce and study two classes of functions
$F:\Gamma_*\to \mathds{R}$, which play a crucial role in defining
the Kolmogorov operator $L$ introduced in (\ref{I5}). In Sect. 3, we
impose standard assumptions on $a$ and $\phi$ and then make precise
the domain of $L$. Thereafter, in Theorem \ref{1tm} we formulate the
result, the main part of which is the statement that the restricted
initial value martingale problem for our model has precisely one
solution. Then we outline our strategy of proving this statement. In
Sect. 4, we present and employ the results of \cite{asia} where the
evolution of states $t\to \mu_t\in \mathcal{P}_{\rm exp}$ was
constructed. In Sect. 5, we prove that the restricted initial value
martingale problem for our model has at most one solution. This is
done by proving that the Fokker-Planck equation (\ref{A1}) has a
unique solution, which lies in the class of sub-Poissonian measures.
Since the one-dimensional marginals of the path measures in question
should solve (\ref{A1}), this yields a tool of proving the desired
uniqueness. In Sect. 6 and 7, we prove the existence of the path
measures by employing auxiliary models (Sect. 6) for which one can
construct the processes directly (by means of transition functions),
and then by proving (Sect. 7) that these models approximate the main
model. Their Markov property is then obtained similarly as in
\cite[Sect. 5.1, pages 78, 79]{Dawson}.

\section*{Notations and notions}

In view of the size of this work, for the reader convenience we
collect here essential notations and notions used throughout the
whole paper.
\subsection*{Sets and spaces}
\begin{itemize}
\item The habitat of the system which we study is the Euclidean
space $\mathds{R}^d$. By $\Lambda$ we always denote a compact subset
of it. Further related notations: $\mathds{R}_{+}=[0,+\infty)$;
$\mathds{N}=\{1,2,3 \dots\}$, $\mathds{N}_0 = \mathds{N}\cup \{0\}$;
$C_{\rm cs} (\mathds{R}^d)$ - the set of all compactly supported
continuous functions $g:\mathds{R}^d\to \mathds{R}$, $B_r(y) =
\{x\in \mathds{R}^d: |x-y|\leq r\}$,  $r>0$ and $y\in \mathds{R}^d$.
For a finite subset $\Delta \subset \mathds{R}$, by $|\Delta|$ we
denote its cardinality.
  \item By a Polish space we mean a separable topological space, the
  topology of which is consistent with a complete metric, see, e.g.,
  \cite[Chapt. 8]{Cohn}. Subsets of such spaces are usually denoted by $\mathbb{A}$,
  $\mathbb{B}$, whereas $A$, $B$ (with indices) are reserved for
  denoting operators.
  For a Polish space $E$, by $C_{\rm b}(E)$
  and $B_{\rm b}(E)$ we denote the sets of bounded continuous and
  bounded measurable functions $g:E\to \mathds{R}$, respectively; $\mathcal{B}(E)$ denotes the Borel $\sigma$-field
of subsets of $E$.   For a suitable set $\Delta$, by
  $\mathds{1}_\Delta$ we denote the indicator of $\Delta$.
\item By $\Gamma$, $\Gamma_0$, $\Gamma_*$  and $\breve{\Gamma}_*$ we denote
configurations spaces consisting of all configurations, finite
configurations (\ref{C22}), tempered configurations (\ref{N1}), and
tempered simple configurations, respectively, see (\ref{Psi10}).
These sets are equipped with the vague topology ($\Gamma$) and the
weak topologies ($\Gamma_0$, $\Gamma_*$, $\breve{\Gamma}_*$), which
make them Polish spaces, see Lemma \ref{Lenardlm}. By
$\mathcal{P}(\Gamma)$, $\mathcal{P}(\Gamma_*)$ we denote the sets of
probability measures defined on these spaces. The set of
sub-Poissonian measures $\mathcal{P}_{\rm exp}$ is introduced in
Definition \ref{Pexpdf}. Its crucial property is established in
Lemma \ref{Lenalm}.
\item By $\mathfrak{D}_{[s,+\infty)}(\Gamma_*)$ we denote the space
of cadlag paths $\gamma:[0,+\infty)\to \Gamma_*$, and
$\mathfrak{D}_{\mathds{R}_{+}}(\Gamma_*):=\mathfrak{D}_{[0,+\infty)}(\Gamma_*)$.
Functions on such spaces are denoted by $\mathsf{F}$, $\mathsf{G}$,
etc. By $\varpi_t$ we denote the evaluation map, i.e.,
$\varpi_t(\gamma) = \gamma_t\in \Gamma_*$.  Related $\sigma$-fields
of measurable subsets are defined in (\ref{Frak}).
\end{itemize}
\subsection*{Functions, measures, operators}
\begin{itemize}
  \item Functions $f:\mathds{R}^d \to \mathds{R}$ are usually
  denoted  by small letters $f$, $g$, $\theta$, etc. By $\psi$ we
  denote the function by which we define tempered configurations, see (\ref{C3}) and (\ref{psi}).
For a positive integrable $\theta:\mathds{R}^d \to \mathds{R}_{+}$,
we write $\langle \theta \rangle = \int \theta(x) d x$.
  Functions
  $F:\Gamma_* \to \mathds{R}$ are denoted by capital letters, often $F$ with
  additional symbols. The key functions are defined in
  (\ref{T4}) and (\ref{TH1}). Functions defined on finite configurations $\Gamma_0$
 are mostly denoted by capital $G$ with exception  for correlation
 functions $k_\mu$, see (\ref{Lenard2}).
\item Measures on configuration spaces and their correlation measures are denoted by
$\mu$ and $\chi_\mu$, respectively. Measures on $\mathds{R}^d$ are
usually denoted by $\nu$. By $\lambda$ we denote the
Lebesgue-Poisson measure, see (\ref{C22c}). Measures on path spaces
are denoted by capital $P$. For a tempered configuration $\gamma\in
\Gamma_*$, by $\nu_\gamma$ we denote the measure $\sum_{x\in \gamma}
\psi(x) \delta_x$, see (\ref{Psi32a}). The complete metric on
$\Gamma_*$ used to obtain Chentsov-like estimates is defined in
(\ref{Psi4}).
\item By $L$ we denote the Kolmogorov operator (\ref{I5}),
(\ref{KO}), whereas $L^\alpha$ stands for the approximating operator
(\ref{Lalpha}). Then $L^\Delta$ and $\widehat{L}$ are the
counterparts of $L$ acting in the spaces of functions of $\eta\in
\Gamma_0$, see (\ref{A9a}), (\ref{K3}) and (\ref{M1}), (\ref{M2}).
By $K$ we define the operator defined in (\ref{A2}). Operators
$L^{\dagger, \alpha}$ act in the Banach space of signed measures
$\mathcal{M}_*$, see (\ref{V11}), (\ref{V12}).

 \end{itemize}

\section{Preliminaries}
\label{S2}

\subsection{The configuration spaces}

Each $\gamma\in \Gamma$ gives rise to a counting Radon measure
$\sum_{x\in \gamma} \delta_x$. Bearing this fact in mind, we shall
mostly keep using set notations, i.e., for a compact $\Lambda
\subset \mathds{R}^d$, the value of the mentioned measure on
$\Lambda$ is denoted by $|\gamma \cap \Lambda|$. The vague
(weak-hash) topology of $\Gamma$ is defined as the weakest topology
that makes continuous all the maps $\Gamma \ni \gamma \mapsto
\sum_{x\in \gamma} f(x)$ with $f\in C_{\rm cs} (\mathds{R}^d)$. The
corresponding Borel $\sigma$-field $\mathcal{B}(\Gamma)$ is the
smallest $\sigma$-field of subsets of $\Gamma$ that makes measurable
all the maps $\gamma \mapsto N_\Lambda (\gamma):= |\gamma \cap
\Lambda|$ with compact $\Lambda\subset \mathds{R}^d$. By
$\mathcal{P}(\Gamma)$ we denote the set of all probability measures
on $(\Gamma, \mathcal{B}(\Gamma))$.

As mentioned in Introduction,  configurations $\gamma\in \Gamma$ may
have multiple points. Let  $x_1, x_2, \dots $ be any enumeration of
the elements of a given $\gamma$ in which coinciding $x$ receive
distinct numbers. Then, for a suitable function $g$, by $\sum_{x\in
\gamma}g(x)$ we will mean $\sum_{i} g(x_i)$, which is independent of
the enumeration used herein.The same relates to the sums
\[
\sum_{x\in \gamma} \sum_{y\in \gamma\setminus x} \sum_{z\in
\gamma\setminus \{x,y\}} \cdots.
\]
Along with $\Gamma$, we also use
\begin{equation}
  \label{C22}
  \Gamma_0 = \bigcup_{n\in \mathds{N}_0} \Gamma^{(n)}, \qquad
  \Gamma^{(n)}= \{\gamma \in \Gamma: |\gamma| = n\}.
\end{equation}
Obviously, each $\Gamma^{(n)}$ -- and hence the set of finite
configurations $\Gamma_0$ -- belong to $\mathcal{B}(\Gamma)$. The
topology induced on $\Gamma_0$ by the vague topology of $\Gamma$
coincides with the weak topology determined with the help of $C_{\rm
b}(\mathds{R}^d)$. Then the corresponding Borel $\sigma$-field
$\mathcal{B}(\Gamma_0)$ is a sub-field of $\mathcal{B}(\Gamma)$. It
is possible to show that a function $G:\Gamma_0 \to \mathds{R}$ is
measurable if and only if there exists a family of symmetric Borel
functions $G^{(n)}: (\mathds{R}^d)^n \to \mathds{R}$, $n\in
\mathds{N}$ such that
\begin{equation}
  \label{C22b}
G(\{x_1 , \dots , x_n\}) =   G^{(n)}(x_1 , \dots , x_n).
\end{equation}
In this context, we also write $G^{(0)}= G(\varnothing)$.
\begin{definition}
  \label{Bbsdf}
A measurable function, $G:\Gamma_0\to \mathds{R}$, is said to have
bounded support if there exist $N\in \mathds{N}$ and a compact
$\Lambda$ such that: (a) $G^{(n)}\equiv 0$ for all $n>N$; (b)
$G(\eta)=0$ whenever $\eta$ is not a subset of $\Lambda$. By $B_{\rm
bs}$ we will denote the set of all bounded functions with bounded
support. For $G\in B_{\rm bs}$,  $N_G$ and $\Lambda_G$ will denote
the least $N$ and $\Lambda$ as in (a) and (b), respectively. We also
set $C_G=\sup_{\eta\in \Gamma_0} |G(\eta)|$.
\end{definition}
The Lebesgue-Poisson measure $\lambda$ is defined on $\Gamma_0$ by
the integrals
\begin{equation}
 \label{C22c}
\int_{\Gamma_0} G(\eta) \lambda(d \eta) = G(\varnothing ) +
\sum_{n=1}^\infty \frac{1}{n!} \int_{(\mathds{R}^d)^n} G^{(n)} (x_1
, \dots, x_n) d x_1 \cdots d x_n,
\end{equation}
holding for all $G\in B_{\rm bs}$. For $G\in B_{\rm bs}$, we set
\begin{equation}
  \label{A2}
  (KG)(\gamma) = \sum_{\eta \Subset \gamma} G(\eta), \qquad
  \gamma\in \Gamma,
\end{equation}
where $\eta \Subset \gamma$ means $\eta\in \Gamma_0$, i.e., the sum
in (\ref{A2}) runs over finite subsets of $\gamma$.
\begin{remark} \cite[Proposition 3.1]{Tobi}
  \label{Bbsrk}
For each $G\in B_{\rm bs}$, $KG$ is measurable and such that
$|(KG)(\gamma)| \leq C_G ( 1 + |\gamma \cap \Lambda_G|^{N_G})$ with
$C_G$, $\Lambda_G$ and $N_G$ as in Definition \ref{Bbsdf}.
\end{remark}

\subsection{Sub-Poissonian measures}

When dealing with infinite configurations, one might expect problems
(e.g., blowups) if the dynamics start from certain $\gamma\in
\Gamma$ or $\mu\in \mathcal{P}(\Gamma)$. Thus, it seems reasonable
to avoid considering such states by imposing appropriate
restrictions. Another reason to do this is gaining technical
advantages, which is especially important in view of the high
complexity of the problem. The main observation here is that, for
measures having \emph{finite correlations} \cite{Lenard},
integration over $\Gamma$ can be performed in the following way
\begin{gather}
  \label{Lenard}
  \int_\Gamma (KG)(\gamma) \mu ( d \gamma) =  G(\varnothing) +
  \sum_{n=1}^\infty \frac{1}{n!}\int_{(\mathds{R}^d)^n} G^{(n)} (x_1 , \dots , x_n)
  \chi^{(n)}_\mu(dx_1 , \dots , dx_n),
\end{gather}
where $\chi_\mu^{(n)}$ are the \emph{correlation measures} of $\mu$.
That is, for a compact $\Lambda \subset \mathds{R}^d$,
$\chi_\mu^{(n)}(\Lambda^n)/n!$ is the $\mu$-expected value of the
number of $n$-clusters of particles contained in $\Lambda$. Next,
one observes that the Kolmogorov operator (\ref{I5}) contains the
probability kernel $a(x-y) d y$, which is absolutely continuous with
respect to Lebesgue's measure on $\mathds{R}^d$. In view of this, we
shall demand that each $\chi_\mu^{(n)}$ satisfy
\begin{equation}
  \label{Lenard1}
\chi^{(n)}_\mu(dx_1 , \dots , dx_n) = k^{(n)}_\mu (x_1 , \dots ,
x_n) d x_1 \cdots  d x_n, \qquad k^{(n)}_\mu \in L^\infty
((\mathds{R}^d)^n), \quad n\in \mathds{N}.
\end{equation}
Thereby, the right-hand side of (\ref{Lenard}) can be rewritten in
the form, cf. (\ref{C22c}),
\begin{equation}
  \label{Lenard2}
  \int_\Gamma (KG)(\gamma) \mu ( d \gamma) = \int_{\Gamma_0} k_\mu
  (\eta)G(\eta) \lambda (d \eta) =:
\langle \!\langle k_\mu, G \rangle \!\rangle,
\end{equation}
where $k_\mu:\Gamma_0 \to \mathds{R}$ is defined as in (\ref{C22b}).
Then $k^{(n)}_\mu$ (resp. $k_\mu$) is called $n$-th order
correlation function (resp. correlation function) of $\mu$. Keeping
this in mind, we introduce the following class of measures. For
$\theta\in C_{\rm cs}(\mathds{R}^d)$ and $n\in \mathds{N}$, we write
$\theta^{\otimes n}(x_1 , \dots ,x_n)= \theta (x_1) \cdots
\theta(x_n)$.
\begin{definition}
 \label{Pexpdf}
By $\mathcal{P}_{\rm exp}$ we denote the set of all those $\mu\in
\mathcal{P}(\Gamma)$ that have finite correlations and their
correlation measures satisfy
\begin{equation}
  \label{Len}
  \chi^{(n)}_\mu (\theta^{\otimes n}) \leq \varkappa^n
  \|\theta\|_{L^ 1(\mathds{R}^d)}^n,
\end{equation}
holding for some $\mu$-specific $\varkappa>0$ and all $\theta \in
C_{\rm cs}(\mathds{R}^d)$ and $n\in \mathds{N}$.
\end{definition}
\begin{remark}
  \label{Len1rk}
It is clear from (\ref{Len}) that the map $C_{\rm cs}(\mathds{R}^d)
\ni \theta \mapsto \chi^{(n)}_\mu (\theta^{\otimes n}) \in
\mathds{R}$ can be continued to a homogeneous continuous monomial of
$\theta \in L^1 (\mathds{R}^d)$. One can show that $\mu\in
\mathcal{P}_{\rm exp}$ holds if and only if each $\chi^{(n)}_\mu$
satisfies (\ref{Lenard1}) with $k^{(n)}_\mu$ such that
\begin{equation}
  \label{I3}
  0 \leq k_{\mu}^{(n)} (x_1 , \dots , x_n) \leq \varkappa^n,
\end{equation}
for the  same $\varkappa$ as in (\ref{Len}). Moreover, if we set
\begin{equation}
  \label{I1}
F^\theta(\gamma) = \prod_{x\in \gamma} (1+\theta(x))=
\exp\left(\sum_{x\in \gamma} \log \left( 1 + \theta(x)
\right)\right), \qquad \theta \in C_{\rm cs}(\mathds{R}^d),
\end{equation}
then the map $C_{\rm cs}(\mathds{R}^d) \ni \theta \mapsto
\mu(F^\theta)\in \mathds{R}$ can be continued to a real exponential
entire function of normal type of $\theta\in L^1(\mathds{R}^d)$. The
least $\varkappa$ satisfying (\ref{I3}) will be called the type of
$\mu$.
\end{remark}
A Poisson measure, $\pi_\chi$, is characterized by its
\emph{intensity measure} $\chi$, see, e.g., \cite[page 45]{Dawson},
by the following formula
\begin{equation*}
\pi_\chi (F^\theta)= \exp\left( \chi(\theta)\right).
\end{equation*}
Then $\pi_\chi\in \mathcal{P}_{\rm exp}$ if
\begin{equation*}
  \chi(dx) = \varrho(x) d x, \qquad \varrho\in L^\infty
  (\mathds{R}^d).
\end{equation*}
In particular, this holds for the homogeneous Poisson measure
$\pi_\varkappa$, for which $\varrho(x) \equiv \varkappa>0$.
\begin{remark}
 \label{Iirk}
Let $G$ in (\ref{Lenard2}) be positive, i.e., such that $G(\eta)
\geq 0$ for all $\eta\in \Gamma_0$. Then by (\ref{I3}) it follows
that $\mu(KG) \leq \pi_\varkappa(KG)$, where $\varkappa$ is the type
of $\mu$. In view of this, the elements of $\mathcal{P}_{\rm exp}$
are called {\it sub-Poissonian} measures.  By taking in (\ref{Len})
$\theta = \mathds{1}_\Lambda$ one gets that the $\mu$-expected value
of the number of $n$-clusters contained in $\Lambda$ does not exceed
that of the homogeneous Poisson measure with density $\varkappa$,
i.e., clusters are not more probable than in the case of free
particles. Moreover, the states of thermal equilibrium of infinite
systems of physical particles interacting via super-stable
potentials belong to $\mathcal{P}_{\rm exp}$, see \cite{Ruelle}.
\end{remark}
Recall that $\mathds{1}_\Lambda$ denotes the indicator of $\Lambda$.
Then $N_\Lambda (\gamma) := |\gamma\cap \Lambda|=\sum_{x\in \gamma}
\mathds{1}_\Lambda(x)$, and thus
\begin{gather*}
N^n_\Lambda (\gamma) = \sum_{l=1}^n  S(n,l) \sum_{x_1\in \gamma}
\sum_{x_2\in \gamma\setminus x_1} \cdots \sum_{x_l\in
\gamma\setminus \{x_1,\dots , x_{l-1}\}}
 \mathds{1}_\Lambda(x_1)\cdots \mathds{1}_\Lambda(x_l)\\[.21cm] = \sum_{l=1}^n l! S(n,l) \sum_{\{x_1 , \dots ,
x_l\}\subset \gamma} \mathds{1}_\Lambda(x_1)\cdots
\mathds{1}_\Lambda(x_l), \qquad n\in \mathds{N},
\end{gather*}
where $S(n,l)$ is Stirling's number of second kind -- the number of
ways to divide $n$ labeled items into $l$ unlabeled groups. By
(\ref{Lenard2}) this yields
\begin{equation}
  \label{Boy1}
\pi_\varkappa (N_\Lambda^n) = \sum_{l=1}^n S(n,l) \left(\varkappa
|\Lambda| \right)^l = T_n \left(\varkappa |\Lambda| \right),  \qquad
n\in \mathds{N},
\end{equation}
where $|\Lambda|$ is the Lebesgue measure (volume) of $\Lambda$ and
$T_n$, $n\in \mathds{N}$, are Touchard's polynomials, attributed
also to J. A. Grunert, S. Ramanujan, and others, see \cite[page
6]{Boy}. For these polynomials, it is known that, see eq. (2.19)
\emph{ibid},
\begin{equation}
  \label{Boy}
  \exp\left( x (e^z-1) \right) = \sum_{n=0}^\infty T_n(x)
  \frac{z^n}{n!}.
\end{equation}
Then for $\mu\in \mathcal{P}_{\rm exp}$, by (\ref{Boy1}) we obtain,
cf. Remark \ref{Iirk},
\begin{equation}
  \label{C2c}
  \mu(N_\Lambda^n) \leq T_n \left(\varkappa |\Lambda| \right).
\end{equation}

\subsection{Tempered configurations}

When dealing with measures from $\mathcal{P}_{\rm exp}$, it might be
natural to distinguish a subset $\Gamma_* \subset \Gamma$ by the
condition that $\mu(\Gamma_*)=1$ for each $\mu\in \mathcal{P}_{\rm
exp}$. Obviously, the  choice of such $\Gamma_*$ should also be
consistent with the properties of $L$, in particular, with those of
the aforementioned probability kernel $a(x-y) d y$. Let $\psi\in
C_{\rm b}(\mathds{R}^d)$ be a strictly positive function that
vanishes at infinity. Denote
\begin{equation}
  \label{Psi}
  \Psi(\gamma) = \sum_{x\in \gamma} \psi(x), \qquad \Gamma^\psi =
  \{\gamma\in \Gamma: \Psi(\gamma) < \infty\}.
\end{equation}
Let $\{\psi_n\}_{n\in \mathds{N}}\subset C_{\rm cs}(\mathds{R}^d)$
be an increasing sequence such that $0<\psi_n(x) \to \psi(x)$, $n\to
+\infty$, for each $x$. Then the maps $\Gamma \ni \gamma \mapsto
\Psi_n(\gamma):=\sum_{x\in \gamma}\psi_n(x)$ are vaguely continuous;
hence $\{\gamma: \Psi_n (\gamma)\leq N\}$, $N\in \mathds{N}$ are
measurable, which by (\ref{Psi}) yields $\Gamma^\psi\in
\mathcal{B}(\Gamma)$. Moreover, if $\tilde{\psi}$ has similar
properties and satisfies $\tilde{\psi}(x) \leq \psi(x)$, $x\in
\mathds{R}^d$, then $\Gamma^\psi \subset \Gamma^{\tilde{\psi}}$.
Thus, the slower the decay of $\psi$ is, the more restrictive
condition is imposed on the configurations. Bearing this in mind, we
will choose $\psi$ satisfying
\begin{equation}
  \label{Psi1}
(i) \quad \mu(\Psi) < \infty; \quad \qquad (ii) \quad
\int_{\mathds{R}^d} \frac{a(x)}{\psi(x)} d x < \infty.
\end{equation}
By (\ref{A2}) and (\ref{Lenard2}) it follows that
\begin{equation*}
  \mu(\Psi) = \chi_\mu (\psi) = \int_{\mathds{R}^d} k_\mu^{(1)}(x)
  \psi(x) d x,
\end{equation*}
and thus condition (i) in (\ref{Psi1}) turns into
\begin{equation}
  \label{psi}
\langle \psi \rangle := \int_{\mathds{R}^d} \psi(x) d x < \infty.
\end{equation}
Our choice of $\psi$ in this work is
\begin{equation}
  \label{C3}
\psi (x) = \frac{1}{1 + |x|^{d+1}},
\end{equation}
which means that we prefer to be less restrictive in choosing the
jump kernel $a$ at the expense of  stronger restrictions imposed on
the configurations.

Similarly as in (\ref{C2c}), for all $n\in \mathds{N}$ and each
$\mu\in \mathcal{P}_{\rm exp}$, one obtains
\begin{eqnarray}
  \label{N7}
\mu(\Psi^n) \leq   \sum_{l=1}^nS(n,l) \left(\varkappa \langle \psi
\rangle\right)^l = T_n \left(\varkappa \langle \psi \rangle\right),
\end{eqnarray}
where we have taken into account that $\psi^n(x) \leq \psi(x)$ for
all $n\geq 1$ and $x$, $\varkappa$ is the type of $\mu$. By
(\ref{N7}) and (\ref{Boy})  it follows that
\begin{equation}
  \label{N701}
    \int_{\Gamma} \exp\left(\beta \Psi(\gamma) \right) \mu ( d
  \gamma)  \leq  \exp\left(\varkappa \langle \psi \rangle (e^\beta - 1)
  \right),
 \end{equation}
holding for all $\beta >0$. Next, we define
\begin{equation}
  \label{N1}
\Gamma_{*} = \Gamma^\psi,
\end{equation}
with $\psi$ as in (\ref{C3}).  By (\ref{N7}) it follows that
\begin{equation}
  \label{C4}
 \forall \mu \in \mathcal{P}_{\rm exp} \qquad  \mu(\Gamma_*)=1.
\end{equation}
This crucial property of the elements of $\mathcal{P}_{\rm exp}$
will allow us to consider only configurations belonging to
$\Gamma_*$. In particular, this means that we will use the following
sub-field of $\mathcal{B}(\Gamma)$:
\begin{equation}
\label{C4z} \mathcal{A}_* = \{\mathbb{A} \in \mathcal{B}(\Gamma):
\mathbb{A} \subset \Gamma_*\}.
\end{equation}
Now let us consider
\begin{equation*}
C_{\rm b}^L(\mathds{R}^d) =\{g \in C_{\rm b}(\mathds{R}^d): \|g\|_L
<\infty \}, \quad \|g\|_L := \sup_{x,y\in \mathds{R}^d, \ x\neq y}
\frac{|g(x) - g(y)|}{|x-y|},
\end{equation*}
and then define
\begin{equation*}
  \|g\|_{BL} = \|g\|_L + \sup_{x\in \mathds{R}^d}|g(x)|, \qquad g\in
  C_{\rm b}^L(\mathds{R}^d),
\end{equation*}
and also
\begin{equation*}
 \upsilon (\nu, \nu') = \sup_{g: \|g\|_{BL}\leq 1} \left\vert
 \nu(g) - \nu'(g)\right\vert, \qquad \nu, \nu' \in \mathcal{N},
\end{equation*}
where $\mathcal{N}$ is the set of all positive finite measures Borel
on $\mathds{R}^d$.
\begin{proposition}{\cite[Theorem 18]{Dudley}}
  \label{N1pn}
The following three types of the convergence of a sequence
$\{\nu_n\}\subset \mathcal{N}$ to a certain $\nu\in \mathcal{N}$ are
equivalent:
\begin{itemize}
  \item[(i)] $\nu_n(g) \to \nu(g)$ for all $g\in C_{\rm
  b}(\mathds{R}^d)$;
  \item[(ii)] $\nu_n(g) \to \nu(g)$ for all $g\in C_{\rm
  b}^L(\mathds{R}^d)$;
  \item[(iii)] $\upsilon (\nu_n, \nu) \to 0$.
\end{itemize}
\end{proposition}
By means of this statement we prove the following important facts.
For a configuration, $\gamma\in {\Gamma}_*$, by $\nu_\gamma\in
\mathcal{N}$ we mean the measure defined by
\begin{equation}
  \label{Psi32a}
\nu_\gamma (g) = \sum_{x\in \gamma} g(x) \psi(x), \qquad g\in C_{\rm
b}(\mathds{R}^d).
\end{equation}
Then we set
\begin{equation}
   \label{Psi4}
 \upsilon_* (\gamma, \gamma') = \upsilon (\nu_\gamma, \nu_{\gamma'}) = \sup_{g: \|g\|_{BL}\leq 1} \left\vert
 \sum_{x\in \gamma} g(x ) \psi(x)  - \sum_{x\in \gamma'} g(x ) \psi(x) \right\vert, \qquad  \gamma, \gamma'\in \breve{\Gamma}_*
 \end{equation}
In the next statement, by $\breve{\Gamma}_*$ we mean the subset of
$\Gamma_*$ consisting of single configurations. That is, $\gamma\in
\Gamma_*$ belongs to $\breve{\Gamma}_*$ if $B_\delta(x) \cap
\gamma=\{x\}$, holding for each $x\in \gamma$ and an $x$-specific
$\delta>0$.

\begin{lemma}
  \label{Lenardlm}
 The metric space $({\Gamma}_*, \upsilon_* )$
 is  complete and separable. $\breve{\Gamma}_*$ is a $G_\delta$ subset of
 $\breve{\Gamma}_*$, and thus is
 a Polish spaces.
\end{lemma}
\begin{proof}
First, we prove that ${\Gamma}_*$ has the properties in question.
Let $\{\gamma_n\}_{n \in \mathds{N}}\subset {\Gamma}_*$ be a
$\upsilon_*$-Cauchy sequence. Since the metric space
$(\mathcal{N},\upsilon)$ is complete, see \cite[Corollary 8.6.3,
Sect. 8.6]{VB}, the sequence $\{\nu_{\gamma_n}\}_{n \in \mathds{N}}$
converges to a certain $\nu \in \mathcal{N}$. As each $h\in C_{\rm
cs}(\mathds{R}^d)$ can be written in the form $h(x) = g(x) \psi(x)$,
$g\in C_{\rm cs}(\mathds{R}^d)$, this convergence implies the vague
convergence of $\{\gamma_n\}_{n\in \mathds{N}}$ to a certain
$\gamma\in {\Gamma}$. Let now $\{g_m\}_{m\in \mathds{N}}\subset
C_{\rm cs}(\mathds{R}^d)$ be such that $g_m(x)=1$ for $|x|\leq m$
and $g_m(x)=0$ for $|x|\geq m+1$, which is possible by Urysohn's
lemma. Then
\[
\lim_{n\to +\infty} \sum_{x\in \gamma_n} g_m(x) \psi(x) = \sum_{x\in
\gamma} g_m(x) \psi(x) \leq \nu(\mathds{R}^d),
\]
which by the dominated convergence theorem yields $\gamma \in
{\Gamma}_*$, and hence $\nu=\nu_\gamma$. Then $({\Gamma}_*,
\upsilon_*)$ is a complete metric space. Its separability follows by
the separability of $\mathds{R}^d$.

When dealing with a topological property of a subset of
${\Gamma}_*$, we may use any metric consistent with its weak
topology. As such one, we take Prohorov's metric, cf. \cite[page
96]{EK}, introduced as follows. For $\varepsilon>0$ and $\Delta
\subset \mathds{R}^d$, we set $\Delta^\varepsilon = \cup_{x\in
\Delta} B_\varepsilon (x)$ and also
\begin{equation}
  \label{Se16a}
 \upsilon_P (\gamma, \gamma') = \inf\{\varepsilon>0:
 \nu_{{\gamma}} (\Delta) \leq
 \nu_{{\gamma}'}(\Delta^\varepsilon) + \varepsilon,
 \ \& \ \nu_{{\gamma}'} (\Delta) \leq
 \nu_{{\gamma}}(\Delta^\varepsilon) + \varepsilon, \ \forall \Delta \
 - \ {\rm closed}
 \}.
\end{equation}
Let $\{R_k\}_{k\in \mathds{N}}$ be such that $0<R_1 < R_2 <\cdots <
R_k < \cdots$ and $\lim_{k\to +\infty}R_k =+\infty$. Set $D_k =\{
x\in \mathds{R}^d: |x| < R_k\}$ and $\gamma_k = \gamma\cap D_k$,
$\gamma\in {\Gamma}_*$, $k\in \mathds{N}$. By (\ref{C3}) we then
have
\begin{gather}
  \label{Psi5}
  \sup_{x\in D_k} 1/\psi(x) = 1+ R_k^{d+1} =: \alpha^{-1}_k, \\[.2cm] \nonumber
  |\psi(x) - \psi(y)| \leq (d+1)|x-y|,  \qquad
x,y\in D_k.
\end{gather}
Next,  we set
\begin{equation}
  \label{Psi6}
  \Gamma_{*,k} =\{ \gamma \in {\Gamma}_*: \gamma_k \in
  \breve{\Gamma}_*\}, \qquad k\in \mathds{N},
\end{equation}
i.e., $\gamma\in {\Gamma}_*$ belongs to $\Gamma_{*,k}$ if its part
in $D_k$ is a simple configuration. Our aim is to show that
$\Gamma_{*,k}$ is an open subset of the Polish space ${\Gamma}_*$.
To this end, we take any $\gamma\in \Gamma_{*,k}$ and look for $r>0$
such that
\begin{equation*}
\Upsilon_r(\gamma):= \{\gamma': \upsilon_P (\gamma, \gamma')<r\}
\subset \Gamma_{*,k}.
\end{equation*}
For the chosen $\gamma$, we pick $\ell>0$ satisfying $B_\ell (x)
\cap \gamma = \{x\}$ for all $x\in \gamma_k$. Now take positive
$\varepsilon$ and $\delta$ such that
\begin{equation}
  \label{Psi8}
  \varepsilon \leq \frac{1}{4} \min\{ \ell;\alpha_k\}, \qquad
  \delta  \leq \frac{1}{4} \min\{ \ell; \alpha_k /(d+1)\},
\end{equation}
and then assume that $\gamma' \in \Upsilon_r(\gamma)$ with $r<
\varepsilon$. For $x\in \gamma_k$,  the second estimate in
(\ref{Se16a}) for $\Delta =B_\delta (x)$ yields in this case
\begin{equation}
  \label{Psi9}
  \sum_{y\in \gamma'\cap B_\delta (x)} \psi(y) \leq \psi(x) +
  \varepsilon,
\end{equation}
where we have taken into account that $B_\delta^\varepsilon (x)
\subseteq B_\ell (x)$, see (\ref{Psi8}). For $y\in B_\delta (x)$, by
(\ref{Psi5}) we have $\psi(y) \geq \psi(x) - \delta (d+1)$. Thus,
\[
{\rm LHS}(\ref{Psi9}) \geq m(x) \psi(x) - m(x) \delta (d+1),
\]
where $m(x) =|\gamma'\cap B_\delta (x)|$. Then
\[
m(x) -1 \leq \frac{\varepsilon}{\psi(x)} + \frac{m(x) \delta(d+1)}
{\psi(x)} \leq (m(x) + 1)/4,
\]
which means that $m(x) =1$, holding for each $x\in \gamma_k$ and
$B_\delta(x)$. At the same time, for $\gamma' \in
\Upsilon_r(\gamma)$ with $r< \varepsilon$, it follows that $ \gamma'
\cap \left(D_k \setminus \cup_{x\in \gamma_k} B_\delta (x) \right) =
\varnothing$. For otherwise, the second estimate in (\ref{Se16a})
with $\Delta = B_\delta (y)$, $y$ lying in the mentioned
intersection, would yield $\psi(y)\leq \varepsilon$ which
contradicts (\ref{Psi8}). Thus, $\gamma'\in \Gamma_{*,k}$, and hence
the latter is an open subset of ${\Gamma}_*$. Therefore,
$\breve{\Gamma}_* = \cap_{k\in \mathds{N}} \Gamma_{*,k}$ is a
$G_\delta$-subset of ${\Gamma}_*$. In view of the first part of this
statement, $\breve{\Gamma}_*$ is a Polish space, see
\cite[Proposition 8.1.5, page 242]{Cohn}. This completes the proof.
\end{proof}
The following formulas summarize the relationships between the
configuration spaces we will deal with
\begin{equation}
  \label{Psi10}
  \breve{\Gamma}_* \subset {\Gamma}_* \subset {\Gamma}.
\end{equation}
Note that the embedding of the Polish space $\Gamma_*$ into the
Polish space $\Gamma$ is continuous, since the weak convergence
$\gamma_n \to \gamma$ implies also the corresponding vague
convergence. Let $\mathcal{B}(\Gamma_*)$,
$\mathcal{B}(\breve{\Gamma}_*)$ be the Borel $\sigma$-field of
subsets of $\Gamma_*$ and $\breve{\Gamma}_*$, respectively.  Recall
that we have another $\sigma$-field, $\mathcal{A}_*$, defined in
(\ref{C4z}).
\begin{corollary}
  \label{N1co}
It follows that $\mathcal{A}_* = \mathcal{B}(\Gamma_*)= \{\mathbb{A}
\in  \mathcal{B}({\Gamma}_*): \mathbb{A}\subset
\breve{\Gamma}_*\}=\mathcal{B}(\breve{\Gamma}_*)$.
\end{corollary}
\begin{proof}
The first equality follows follows  by the continuity of the
embedding and then by  Kuratowski's theorem, see \cite[Theorem 3.9,
page 21]{Part}. The second equality follows by the equality of the
weak topology of $\breve{\Gamma}_*$ with that induced by the weak
topology of ${\Gamma}_*$.
\end{proof}
\begin{remark}
  \label{Lenark}
The latter statement allows one to redefine each $\mu\in
\mathcal{P}(\Gamma)$ with the property $\mu(\Gamma_*)=1$ as a
measure on the measurable space $(\Gamma_*, \mathcal{B}(\Gamma_*))$.
And similarly, each measure on $({\Gamma}_*,
\mathcal{B}({\Gamma}_*))$ possessing the property
$\mu(\breve{\Gamma}_*)=1$ can be considered as a measure on
$(\breve{\Gamma}_*, \mathcal{B}(\breve{\Gamma}_*))$.
\end{remark}
Now we turn to proving the following statement.
\begin{lemma}
  \label{Lenalm}
For each $\mu \in \mathcal{P}_{\rm exp}$, see Definition
\ref{Pexpdf}, it follows that $\mu(\breve{\Gamma}_*)=1$. Hence, this
$\mu$ can be redefined as a measure on $(\breve{\Gamma}_*,
\mathcal{B}(\breve{\Gamma}_*))$, cf. Corollary \ref{N1co} and Remark
\ref{Lenark}.
\end{lemma}
\begin{proof}
By our assumption the
 correlation measures $\chi_\mu^{(n)}$ of the measure under consideration have the
 properties corresponding to (\ref{Lenard1}) and (\ref{I3}). For
 $N\in \mathds{N}$ and $\epsilon \in (0,1)$, we set
\begin{equation*}
  H_N (\gamma) = \sum_{x\in \gamma}\sum_{y\in \gamma\setminus x}
  h_N(x,y), \quad h_N(x,y)= \psi(x) \psi(y) \min\{N; |x-y|^{-d
  \epsilon}\}.
\end{equation*}
Note that $H_N (\gamma)<\infty$ for all $N\in \mathds{N}$ and
$\gamma\in {\Gamma}$. By (\ref{Lenard2}) we then have
\begin{gather*}
\mu(H_N) = \int_{(\mathds{R}^d)^2 } k^{(2)}_\mu ( x , y) h_N (x,y) d
x dy \leq \varkappa^2 \int_{\mathds{R}^d } \psi(x)
\left(\int_{\mathds{R}^d } \frac{\psi(y) dy}{|x-y|^{d \epsilon}}
\right) d x \\[.2cm] \nonumber \leq \varkappa^2 \int_{\mathds{R}^d } \psi(x)
\left( \int_{B_r(x)} \frac{dy}{|x-y|^{d\epsilon}} + \frac{\langle
\psi \rangle}{r^{d \epsilon}}\right) d x =: \varkappa^2 C,
\end{gather*}
for an appropriate $C>0$. Since $H_{N}\leq H_{N+1}$, we can apply
here the Beppo Levi (monotone convergence) theorem, which yields
that the point-wise limit
\begin{equation}
  \label{Hk}
  \lim_{N\to +\infty} H_N(\gamma) = H(\gamma):= \sum_{x\in
\gamma}\sum_{y\in \gamma\setminus x} \frac{\psi(x)
\psi(y)}{|x-y|^{d\epsilon}}
\end{equation}
is finite for $\mu$-almost all $\gamma$, i.e., for all ${\gamma}\in
{\Gamma}_{*,\mu}$ such that $\mu({\Gamma}_{*,\mu}) =1$. For $c>0$,
we set ${\Gamma}_c  =\{{\gamma}:H ({\gamma}) \leq c\}$. Then
$|x-y|\geq c^{-1/d\epsilon}$ for all pairs $x,y\in {\gamma}$ and
each ${\gamma}\in {\Gamma}_c$. That is, ${\gamma}$ is simple; hence,
${\Gamma}_{*,\mu}\subset \breve{\Gamma}_{*}$, which completes the
proof.
\end{proof}
\begin{remark}
  \label{Lena1rk}
By (\ref{C4}) it follows that the class of measures $\mu\in
\mathcal{P}({\Gamma}_*)$ with the property $\mu(\breve{\Gamma}_*)=1$
includes $\mathcal{P}_{\rm exp}$. Therefore, depending on the
context, we can and will consider such measures on either of these
spaces.
\end{remark}

\subsection{Functions and measures on  $\Gamma_*$}
The main aim of this part is to introduce suitable classes of
functions $F:\Gamma_* \to \mathds{R}$, for which we define $LF$ and
then use in (\ref{A1}). We begin by introducing suitable functions
$g:\mathds{R}^d\to \mathds{R}$. For $\psi$ defined in (\ref{C3}), we
set
\begin{eqnarray}
  \label{T1}
 \varTheta_\psi & = & \{ \theta (x) = g(x) \psi (x) :  g\in C_{\rm
 b}(\mathds{R}^d), \ \ \theta(x)\geq 0\}, \\[.2cm] \nonumber
 \varTheta^{+}_\psi & = & \{ \theta \in \varTheta_\psi: \theta(x) >0 \ \ \forall x\in
 \mathds{R}^d\}.
\end{eqnarray}
Clearly, each $\theta\in \varTheta_\psi$ is integrable. For such
$\theta$, we also define
\begin{equation}
  \label{C80}
  c_\theta = \sup_{x\in \mathds{R}^d} \frac{1}{\psi(x)}\log\left(1+{\theta(x)}
  \right),\qquad  \bar{c}_\theta:= e^{c_\theta} -1.
\end{equation}
Then
\begin{equation}
  \label{C801}
 0\leq  \theta (x) \leq \bar{c}_\theta
  \psi(x), \qquad \theta \in \varTheta_\psi.
\end{equation}
Now let us turn to $F^\theta$ defined in (\ref{I1}). By Remark
\ref{Iirk}, (\ref{psi}), Remark \ref{Lena1rk},  and  then by
(\ref{C801}), for $\mu \in \mathcal{P}_{\rm exp}$ of type
$\varkappa$ we have
\begin{gather*}
\int_{\breve{\Gamma}_*}F^\theta (\gamma) \mu(d \gamma) =
\int_{{\Gamma}_*}F^\theta (\gamma) \mu(d \gamma)\leq \pi_\varkappa
(F^\theta) \leq \exp\left( \varkappa \langle \psi \rangle
\bar{c}_\theta \right), \qquad \theta \in \varTheta_\psi.
\end{gather*}
\begin{remark}
  \label{N1rk}
In general, for $\theta\in\varTheta_\psi$ the map $\Gamma_* \ni
\gamma \mapsto \sum_{x\in \gamma} \theta(x)$ need not be vaguely
continuous. But it is weakly continuous for all such $\theta$, which
is also the case  for $\breve{\Gamma}_* \ni \gamma \mapsto
\sum_{x\in \gamma} \theta(x)$. In particular, the map ${\Gamma}_*
\ni \gamma \mapsto \Psi(\gamma)$ is weakly continuous, that is one
of the advantages of passing to tempered configurations. Since the
measurability and continuity of $F:\breve{\Gamma}_* \to \mathds{R}$
and $F:{\Gamma}_* \to \mathds{R}$ occur simultaneously, each such a
function can and will be considered as a map acting from either of
these spaces. In the sequel, when we speak of the properties of a
given $F:{\Gamma}_* \to \mathds{R}$, we tacitely assume that the
same also holds for its restriction to $\breve{\Gamma}_*$.
\end{remark}
For $\theta \in \varTheta_\psi$, we set, see (\ref{T1}) and
(\ref{C80}),
\begin{equation}
  \label{T3}
  v_{\tau}^{\theta} (x) = \tau - \frac{1}{\psi(x)} \log \left( 1 + \theta(x)\right),
  \qquad V = \left\{ v_{\tau}^\theta:\theta \in \varTheta_{\psi}, \ \tau > c_\theta  \right\}.
\end{equation}
Note that $V\subset C_{\rm b}(\mathds{R}^d)$ is closed with respect
to the pointwise addition and its elements are separated away from
zero. The former follows by the fact that $\theta + \theta' +\theta
\theta'$ belongs to $\varTheta_\psi$ for each $\theta, \theta'\in
\varTheta_\psi$. Next, define
\begin{equation}
  \label{T4}
\widetilde{F}_{\tau}^\theta (\gamma)  =  \prod_{x\in \gamma} \left(
1 + \theta(x)\right)e^{-\tau \psi(x)}= \exp\left( - \nu_\gamma
(v_\tau^\theta)\right).
\end{equation}
Recall here that $\tau >c_\theta$, see (\ref{T3}). We extend this to
$\tau = 0$ and $\theta(x) \equiv 0$ by setting  $\widetilde{F}_{0}^0
(\gamma)\equiv 1$ and include this function in the set
\begin{equation}
 \label{T4a}
\widetilde{\mathcal{F}}  :=  \{ \widetilde{F}_{\tau}^\theta : \theta
\in \varTheta_\psi, \  \tau > c_\theta \} \subset  C_{\rm
b}(\Gamma_*).
\end{equation}
 Similarly as in \cite[Sect. 3.2, page
41]{Dawson}, see also \cite[page 111]{EK}, we introduce the
following notion.
\begin{definition}
  \label{V1df}
A sequence of bounded measurable functions $F_n : \Gamma_* \to
\mathds{R}$, $n\in \mathds{N}$ is said to boundedly and pointwise
(bp-) converge to a given $F : \Gamma_* \to \mathds{R}$ if: (a) $F_n
(\gamma) \to F(\gamma)$ for all $\gamma\in \Gamma_*$; (b)
$\sup_{n\in \mathds{N}} \sup_{\gamma\in \Gamma_*} |F_n(\gamma)| <
\infty$. The bp-closure of a set $\mathcal{H}\subset B_{\rm
b}(\Gamma_*)$ is the smallest subset of $B_{\rm b}(\Gamma_*)$ that
contains $\mathcal{H}$ and is closed under the bp-convergence. In a
similar way, one defines also the bp-convergence of sequences of
functions $g:\mathds{R}^d\to \mathds{R}$.
\end{definition}
It is well-known that $C_{\rm b}(\mathds{R}^d)$ contains a countable
family of nonnegative functions, $\{g_i\}_{i\in \mathds{N}}$, which
is {\it convergence determining} and such that its linear span is
bp-dense in $B_{\rm b}(\mathds{R}^d)$, see \cite[Proposition 4.2,
page 111]{EK} and \cite[Lemma 3.2.1, page 41]{Dawson}. This means
that a sequence of finite positive measures $\{\nu_n\}\in
\mathcal{P}(\mathds{R}^d)$ weakly converges to a certain $\nu$ if
and only if $\nu_n(g_i) \to \nu (g_i)$, $n \to +\infty$ for all
$i\in \mathds{N}$. One may take such a family containing the
constant function $g(x)\equiv 1$ and closed with respect to the
pointwise addition. Moreover, one may assume that
\begin{equation}
  \label{S}
 \forall i\in \mathds{N} \qquad  \inf_{x\in \mathds{R}^d} g_i(x) =: \varsigma_i >0.
\end{equation}
If this is not the case for a given $g_i$, in place of it one may take
$\tilde{g}_i(x)= g_i(x) + \varsigma_i$ with some $\varsigma_i>0$. The new set,
$\{\tilde{g}_i\}$, has both mentioned properties and also satisfies
(\ref{S}). Then assuming the latter we conclude that
\begin{equation}
  \label{T3z}
 V_0 := \{g_i\}_{i\in \mathds{N}} \subset V.
\end{equation}
To see this, for a given $g_i$, take $\tau_i \geq \sup_{x} g_i(x)$
and then set
\begin{equation}
  \label{MG}
  \theta_i (x) = \exp\bigg{(} [\tau_i - g_i(x)]\psi(x) \bigg{)} -1.
\end{equation}
Clearly, $\theta_i(x) \geq 0$. Since $\psi^n (x) \leq \psi(x)$,
$n\in \mathds{N}$, we have that $\theta_i(x) \leq e^{\tau_i}
\psi(x)$, and hence $\{\theta_i\}_{i\in \mathds{N}}\subset
\varTheta_\psi$, see (\ref{T1}). At the same time,
$v^{\theta_i}_{\tau_i} =g_i$ and $c_{\theta_i} = \sup_{x} (\tau_i -
g_i(x))< \tau_i$ in view of (\ref{S}). By (\ref{MG}) and
(\ref{T3z}), for all $i\in \mathds{N}$, it follows that
\begin{equation*}
 F_i \in \widetilde{\mathcal{F}},  \qquad F_i(\gamma) := \exp\left( - \nu_\gamma ( g_i)\right).
\end{equation*}
\begin{proposition}
  \label{T1pn}
The set $\widetilde{\mathcal{F}}$ defined in (\ref{T4a}) is closed
with respect to the poinwise multiplication. Moreover, it has the
following properties:
\begin{itemize}
  \item[(i)] It is separating: $\mu_1 (F)=\mu_2 (F)$, holding for all
  $F\in \widetilde{\mathcal{F}}$, implies $\mu_1 = \mu_2$ for all $\mu_1,
  \mu_2 \in \mathcal{P}(\Gamma_*)$.
\item[(ii)] It is convergence determining: if a sequence $\{\mu_n\}_{n\in \mathds{N}}\subset \mathcal{P}(\Gamma_*)$ is such that $\mu_n(F) \to \mu(F)$, $n\to +\infty$ for all $F\in \widetilde{\mathcal{F}}$ and some $\mu\in  \mathcal{P}(\Gamma_*)$, then $\mu_n(F) \to \mu(F)$ for all $F\in C_{\rm b}(\Gamma_*)$.
\item[(iii)]   The set $B_{\rm b}(\Gamma_*)$ is the bp-closure of the linear span of $\widetilde{\mathcal{F}}$.
\end{itemize}
\end{proposition}
\begin{proof}
The closedness of $\widetilde{\mathcal{F}}$ under multiplication
follows directly by (\ref{T4}) and the fact that $\theta_1 +
\theta_2 + \theta_1\theta_2 \in \varTheta_\psi$ for each $\theta_1,
\theta_2\in \varTheta_\psi$. It is clear that
$\widetilde{\mathcal{F}}$ separates points of $\Gamma_*$, i.e., one
finds $F\in \widetilde{\mathcal{F}}$ such that $F(\gamma_1) \neq
F(\gamma_2)$ whenever $\gamma_1 \neq \gamma_2$, that holds for each
pair $\gamma_1 , \gamma_2\in \Gamma_*$. Then claim (i) follows by
\cite[claim (a) Theorem 4.5, page 113]{EK}. Claim (ii) follows by
the fact that $\{F_i\}_{i\in \mathds{N}} \subset
\widetilde{\mathcal{F}}$ has the property in question, which in turn
follows by \cite[Theorem 3.2.6, page 43]{Dawson}. Likewise, claim
(iii) follows by \cite[Lemma 3.2.5, page 43]{Dawson}.
\end{proof}
Note that each function as in (\ref{T4}) can be written in the form
\begin{equation}
  \label{TH}
\widetilde{F}^\theta_\tau (\gamma) = \exp\left( - \tau\Psi (\gamma)
\right) F^\theta (\gamma),
\end{equation}
where $F^\theta$ is as in (\ref{I1}), which is a
$\upsilon_*$-continuous function for each $\theta\in
\varTheta_{\psi}$.

For $m\in \mathds{N}$, $\theta_1, \dots , \theta_m\in
\varTheta^{+}_\psi$, see (\ref{T1}), we set
\begin{eqnarray}
  \label{TH1}
 & & \widehat{F}_\tau^{\theta_1, \dots , \theta_m} (\gamma)
\\[.2cm] \nonumber & &  = \sum_{x_1 \in \gamma} \theta_1(x_1) \sum_{x_2\in \gamma\setminus
x_1} \theta_2(x_2) \cdots \sum_{x_n \in \gamma\setminus \{x_1 ,
\dots , x_{m-1}\}} \theta_m (x_m)
\widetilde{F}_\tau^{0}(\gamma\setminus \{ x_1, \dots , x_m\})\\[.2cm] \nonumber & &  = \sum_{\{x_1, \dots , x_m\}\subset \gamma} \sum_{\sigma \in S_m} \theta_1(x_{\sigma (1)}) \cdots \theta_m(x_{\sigma (m)})\widetilde{F}_\tau^{0}(\gamma\setminus \{ x_1, \dots , x_m\}),
\end{eqnarray}
where $S_m$ is the symmetric group and $\widetilde{F}^0_\tau
(\gamma) = \exp\left( - \tau \Psi(\gamma) \right)$, see (\ref{TH}).
\begin{proposition}
  \label{TH1pn}
For each $\tau>0$, $m\in\mathds{N}$ and $\theta_1, \dots , \theta_m\in \varTheta_\psi^{+}$,  it follows that $\widehat{F}_\tau^{\theta_1, \dots , \theta_m}\in C_{\rm
 b}(\Gamma_*)$.
\end{proposition}
\begin{proof}
To prove the continuity of $\widehat{F}_\tau^{\theta_1, \dots ,\theta_m} $ we rewrite (\ref{TH1}) in the form
\begin{eqnarray}
  \label{TH1z}
\widehat{F}_\tau^{\theta_1, \dots , \theta_m} (\gamma)& = &
\exp\left(
-\tau \Psi_0(\gamma) \right)\\[.2cm] \nonumber & \times & \sum_{\{x_1, \dots , x_m\}\subset
\gamma} \sum_{\sigma \in S_m} \varphi_{\sigma(1)}(x_1) \cdots
\varphi_{\sigma(m)}(x_m),
\end{eqnarray}
with $\varphi_j(x) := \theta_j(x)e^{\tau \psi(x)}$, $j=1, \dots ,
m$. Clearly, all $\varphi_j$ belong to $\varTheta_\psi^{+}$. By an
inclusion-exclusion formula the right-hand side of (\ref{TH1z}) can
be written as a linear combination of the products of the following
terms
\[
\varPhi_{i_1, \dots , i_s} (\gamma)= \sum_{x\in \gamma}
\varphi_{i_1}(x) \cdots \varphi_{i_s} (x),
\]
multiplied by a continuous function, $\gamma \mapsto \exp\left(
-\tau \Psi(\gamma) \right)$. Since $\varTheta_\psi^{+}$ is closed
with respect to the pointwise multiplication, such terms are
continuous that yields the continuity of
$\widehat{F}_\tau^{\theta_1, \dots , \theta_m}$. To prove the
boundedness we estimate each  $\varphi_j(x) \leq \varphi(x):=
c\psi(x) e^{\tau \psi(x)}\leq c\psi(x) e^{\tau}$. Then
\begin{eqnarray*}
  \widehat{F}_\tau^{\theta_1, \dots , \theta_m} (\gamma)& \leq &
  \exp\left( - \tau \Psi(\gamma) \right) \sum_{x_1\in \gamma}
 \varphi(x_1) \sum_{x_2\in \gamma\setminus x_1}\varphi(x_2)\cdots \sum_{x_m\in \gamma\setminus \{x_1, \dots,
 x_{m-1}\}}\varphi(x_m) \\[.2cm] & \leq & \exp\left( - \tau \Psi(\gamma)
 \right)\left(\sum_{x\in \gamma}\varphi(x) \right)^m \leq c^m
 \Psi^m(\gamma)\exp\left( - \tau \left[\Psi(\gamma) - m
 \right]
 \right) \\[.2cm] &\leq & \left(\frac{c m }{\tau}\right)^m \exp\left(m(\tau-1)
 \right),
\end{eqnarray*}
which completes the proof.
\end{proof}
\section{The Result}

\subsection{The domain of $L$}

Here we recall that the model we study is specified by the
Kolmogorov operator , cf. (\ref{I5}),
\begin{equation}
  \label{KO}
(LF)(\gamma) = \sum_{x\in \gamma}\int_{\mathds{R}^d}
a(x-y)\exp\left(- \sum_{z\in \gamma\setminus x} \phi(z-y)
\right)\left[F(\gamma\setminus x\cup y) - F(\gamma)\right] d y.
\end{equation}
Let us make precise the conditions imposed on the model. The
positive measurable functions $a$ and $\phi$ are supposed to satisfy
the following:
\begin{gather}
  \label{C7}
  \sup_{x} a(x) = \bar{a} <\infty, \qquad \sup_{x} \phi(x) =
  \bar{\phi}<\infty, \\[.2cm] \nonumber
  \int_{\mathds{R}^d} \phi(x) d x =:\langle \phi\rangle <
  \infty,\qquad \int_{\mathds{R}^d}  a(x) dx  =1,
\end{gather}
and
\begin{equation}
  \label{C8}
  \int_{\mathds{R}^d} |x|^{l}a(x) d x =:m^a_l < \infty, \qquad
  {\rm for} \ \ l =1 , \dots ,  d+1.
\end{equation}
The conditions in (\ref{C7}) are the same as in \cite{asia}. We
impose them to be able to use the results of this work here. Note
that the assumed boubedness of $\phi$ excludes a hard-core
repulsion. The condition in (\ref{C8}) is the realization of item
(ii) of (\ref{Psi1}). It was not used in \cite{asia}.

As mentioned in Introduction,  we are going to construct the process
as a solution of a {\it restricted initial value martingale
problem}. In this case, the domain of the operator introduced in
(\ref{I5}) is crucial, cf. \cite[page 79]{Dawson}. Along with the
set introduced in (\ref{T4a}), we define
\begin{gather}
  \label{C800}
    \widehat{\mathcal{F}}=
  \{\widehat{F}^{\theta_1, \dots , \theta_m}_\tau: m\in \mathds{N},
  \ \theta_1, \dots , \theta_m\in \varTheta^{+}_\psi, \ \tau>0 \},
\end{gather}
where $\widehat{F}_\tau^{\theta_1, \dots , \theta_m}$ is as in
(\ref{TH1}).
\begin{definition}
  \label{THdf}
By $\mathcal{D}(L)$ we denote the linear span of the set
$\widetilde{\mathcal{F}} \cup \widehat{\mathcal{F}}$.
\end{definition}
By (\ref{T4a}) and Proposition \ref{TH1pn} one concludes that
$\mathcal{D}(L)\subset C_{\rm b}(\Gamma_*)$. Let us show that
$L\widehat{F}_\tau^{\theta_1, \dots , \theta_m}\in B_{\rm
b}(\Gamma_*)$. For $\gamma\in \Gamma_*$, $x\in \gamma$, $y\in
\mathds{R}^d$  and a suitable $F:\Gamma_* \to \mathds{R}$, define,
cf. (\ref{I5}),
\[
\nabla^{y,x} F (\gamma) = F(\gamma\setminus x\cup y) - F(\gamma).
\]
By (\ref{TH1}) we have
\[
\widehat{F}_\tau^{\theta_1, \dots , \theta_m} (\gamma) =
\sum_{x_1\in \gamma} \theta_1(x_1) \widehat{F}_\tau^{\theta_2, \dots
, \theta_m}(\gamma\setminus x_1).
\]
Then
\[
\nabla^{y,x} \widehat{F}_\tau^{\theta_1, \dots , \theta_m} (\gamma)
= [\theta_1 (y) - \theta_1(x) ]\widehat{F}_\tau^{\theta_2, \dots ,
\theta_m} (\gamma\setminus x) + \sum_{x_1\in \gamma\setminus x}
\theta_1(x_1) \nabla^{y,x}\widehat{F}_\tau^{\theta_2, \dots ,
\theta_m}(\gamma\setminus x_1).
\]
By iterating the latter we get
\begin{eqnarray}
  \label{TH2}
\nabla^{y,x} \widehat{F}_\tau^{\theta_1, \dots , \theta_m} (\gamma)
& = & \sum_{j=1}^m[\theta_j (y) - \theta_j(x)
]\widehat{F}_\tau^{\theta_1, \dots , \theta_{j-1}, \theta_{j+1} ,
\dots , \theta_m} (\gamma\setminus x) \\[.2cm] \nonumber & + &
\bigg{(}\exp\left(- \tau\psi(y) \right) - \exp\left(- \tau\psi(x)
\right)\bigg{)}\widehat{F}_\tau^{\theta_1, \dots , \theta_m}
(\gamma\setminus x).
\end{eqnarray}
For $\theta\in \varTheta_\psi$ and $a$ as in (\ref{C7}), we set
\begin{equation}
  \label{TH3}
(a*\theta)(x) = \int_{\mathds{R}^d} a(x-y) \theta (y) dy =
\int_{\mathds{R}^d} \theta(x-y) a (y) dy.
\end{equation}
Then $a*\theta \in C_{\rm b}(\mathds{R}^d)$, where the continuity
follows by the dominated convergence theorem and the latter equality
in (\ref{TH3}). Moreover, by (\ref{C801}) we have
\begin{eqnarray}
  \label{C81}
(a\ast\theta) (x) & \leq & \bar{c}_\theta \psi(x)\int_{\mathds{R}^d}
( 1 +
|x|^{d+1} )a(x-y) \psi(y) dy \\[.2cm] \nonumber & \leq &\bar{c}_\theta \psi(x)\left[ 1+ \int_{\mathds{R}^d}
(|x-y|+ |y|)^{d+1} a(x-y) \psi(y) dy \right] \\[.2cm] \nonumber & = & \bar{c}_\theta \psi(x)\left[
1+ \sum_{l=0}^{d+1} { d+1 \choose l} \int_{\mathds{R}^d}
|x-y|^{d+1-l} |y|^l \psi(y) a(x-y) dy \right]\\[.2cm] \nonumber &
\leq & \bar{c}_\theta \psi(x)\left[ 1+ \sum_{l=0}^{d+1} { d+1
\choose l} m_l^a \right],
\end{eqnarray}
where we have used (\ref{C7}), (\ref{C8}) and the fact that $|y|^l
\psi(y) \leq 1$ holding for all $y$ and $0\leq l\leq d+1$.
Therefore,
\begin{equation}
  \label{TH4}
\theta_j^1 (x) := (a *\theta_j)(x) + \theta_j(x) \leq c_a
\bar{c}_{\theta_j} \psi(x).
\end{equation}
Since $\theta_j\in \varTheta_\psi^{+}$, we then get by the latter
that  also $\theta_j^1\in \varTheta_\psi^{+}$, $j=1, \dots , m$.
Here
\begin{equation}
  \label{ca}
  c_a := 2 + \sum_{l=0}^{d+1} { d+1 \choose l} m_l^a.
\end{equation}
At the same time
\begin{eqnarray}
  \label{TH5}
|\exp\left(- \tau\psi(y) \right) - \exp\left(- \tau\psi(x) \right)
|\leq \tau \psi(y) \psi(x) ||x|^{d+1} - |y|^{d+1}|.
\end{eqnarray}
Then proceeding as in (\ref{C81}) we get
\begin{eqnarray}
  \label{TH6}
  \int_{\mathds{R}^d} a(x-y) |\exp\left(- \tau\psi(y) \right) - \exp\left(- \tau\psi(x) \right)
| dy \leq \tau c_a \psi(x).
\end{eqnarray}
Thereafter, by (\ref{TH2}), (\ref{TH3}), (\ref{TH4}), and (\ref{I5}) we
obtain
\begin{eqnarray}
  \label{TH7}
\left|L \widehat{F}_\tau^{\theta_1, \dots , \theta_m} (\gamma)
\right| & = & \left\vert\sum_{x\in \gamma} \int_{\mathds{R}^d}
a(x-y) \exp\left( - \sum_{z\in \gamma\setminus x}
\phi(z-y)\right)\nabla^{y,x}\widehat{F}_\tau^{\theta_1, \dots ,
\theta_m} (\gamma) d y  \right\vert \qquad   \\[.2cm] \nonumber &
\leq & \sum_{j=1}^m \widehat{F}_\tau^{\theta_1, \dots ,
\theta_{j-1}, \theta^1_j,\theta_{j+1} , \dots , \theta_m} (\gamma) +
\tau c_a \left(\prod_{j=1}^m \bar{c}_{\theta_j}\right)
\widehat{F}_\tau^{m+1} (\gamma).
\end{eqnarray}
where, cf. (\ref{TH1}),
\begin{eqnarray}
  \label{TH8}
\widehat{F}_\tau^{m} (\gamma) = \sum_{x_1\in \gamma } \psi(x_1)
\sum_{x_2\in \gamma\setminus x_1} \psi(x_2)\cdots \sum_{x_m \in
\gamma\setminus \{x_1, \dots , x_{m-1}\}} \psi(x_m)
\widetilde{F}^0_\tau (\gamma\setminus \{x_1 , \dots, x_m\}). \qquad
\end{eqnarray}
Then the boundedness of $L \widehat{F}_\tau^{\theta_1, \dots ,
\theta_m}$ follows by Proposition \ref{TH1pn}.

Now let us show that
$L \widetilde{F}_\tau^\theta \in B_{\rm b}(\Gamma_*)$ for all
$\theta \in \varTheta_{\psi}$ and $\tau>c_\theta$. Similarly as in
(\ref{TH2}) we get
\begin{eqnarray*}
 \nabla^{y,x} \widetilde{F}_\tau^\theta(\gamma) = \left[ e^{-\tau \psi(y)} -  e^{-\tau
 \psi(x)}\right] \widetilde{F}_\tau^\theta(\gamma\setminus x) +  \left[\theta(y) e^{-\tau \psi(y)} - \theta(x) e^{-\tau
 \psi(x)}\right] \widetilde{F}_\tau^\theta(\gamma\setminus x).
\end{eqnarray*}
Then by (\ref{TH4}), (\ref{ca}), (\ref{TH5}) and (\ref{TH6}) we
arrive at
\begin{eqnarray*}
\left|L \widetilde{F}_\tau^\theta (\gamma)\right| \leq e^\tau
(1+\tau) c_a \Psi(\gamma) \exp\left( - \tau_0 \Psi (\gamma)
\right)\prod_{x\in \gamma} \left(1+\theta(x)
\right)e^{-(\tau-\tau_0)\psi(x)},
\end{eqnarray*}
where $\tau_0>0$ is such that $\tau - \tau_0 > c_\theta $ which is
possible for each $\tau>c_\theta$. Then the boundedness in question
follows similarly as in Proposition \ref{TH1pn}. The next statement
summarizes the properties of $\mathcal{D}(L)$.
\begin{proposition}
  \label{T3pn}
The set of functions introduced in Definition \ref{THdf} has the
following properties:
\begin{itemize}
\item[(i)] $\mathcal{D}(L) \subset  C_{\rm b}(\Gamma_*)$ and $L: \mathcal{D}(L) \to B_{\rm b}(\Gamma_*)$.
  \item[(ii)] The set $B_{\rm b}(\Gamma_*)$ is the bp-closure of $\mathcal{D}(L)$.
  \item[(iii)] $\mathcal{D}(L)$ is separating. That is, if $\mu_1, \mu_2\in \mathcal{P}_{\rm exp}$ satisfy
  $\mu_1(F) = \mu_2(F)$ for all $F\in \mathcal{D}(L)$, then $\mu_1 =
  \mu_2$.
\item[(iv)] For each $F\in \widetilde{\mathcal{F}}$, see (\ref{C800}), and $\mu\in \mathcal{P}_{\rm exp}$, the measure $F \mu/ \mu(F)$ belongs to
$\mathcal{P}_{\rm exp}$.
\end{itemize}
\end{proposition}
\begin{proof}
Claim (i) has been just proved. Claims (ii) and (iii) follow by
Proposition \ref{T1pn} and the fact $\widetilde{\mathcal{F}}\subset
\mathcal{D}(L)$. It remains to check the validity of (\ref{Len}) for
$\mu_F := F \mu/\mu(F)$. For positive $\theta\in C_{\rm
cs}(\mathds{R}^d)$ and a bounded positive $F$, we have that
\begin{eqnarray*}
\chi_{\mu_F}(\theta^{\otimes n})& = & \frac{1}{\mu(F)}
\int_{\Gamma_*} F(\gamma)\left(\sum_{x_1\in \gamma} \sum_{x_2\in
\gamma\setminus x_1} \cdots \int_{\gamma\setminus \{x_1 . \dots ,
x_{n-1}\}} \theta (x_1) \theta(x_2) \cdots \theta(x_n) \right) \mu(d
\gamma) \\[.2cm] & \leq & \chi_\mu (\theta^{\otimes n})
\left(\sup_{\gamma}F(\gamma) / \mu(F)\right),
\end{eqnarray*}
which completes the proof.
\end{proof}

\subsection{Formulating the result}

As mentioned in Introduction, following \cite[Chapter 5]{Dawson} we
are going to obtain the process by solving a {\it restricted initial
value martingale problem}. Recall that
$\mathfrak{D}_{\mathds{R}_{+}} ({\Gamma}_*)$ stands for the space of
all cadlag maps $[0,+\infty)=:\mathds{R}_{+} \ni t \mapsto \gamma_t
\in \Gamma_*$, and the evaluation maps $\varpi_t$, $t\geq 0$, act as
follows: $\mathfrak{D}_{\mathds{R}_{+}} ({\Gamma}_*)\ni \gamma
\mapsto \varpi_t(\gamma) = \gamma_t \in {\Gamma}_*$. In a similar
way, one defines also the spaces $\mathfrak{D}_{[s,+\infty)}
({\Gamma}_*)$, $s>0$.
 For $s, t \geq 0$, $s < t$, by $\mathfrak{F}^0_{s,t}$ we
denote the $\sigma$-field of subsets of
$\mathfrak{D}_{\mathds{R}_{+}}({\Gamma}_*)$ generated by the family
$\{ \varpi_u: u\in [s,t]\}$. Then we set
\begin{equation}
  \label{Frak}
 \mathfrak{F}_{s,t}= \bigcap_{\varepsilon
>0}\mathfrak{F}^0_{s,t+\varepsilon}, \qquad \mathfrak{F}_{s,+\infty}
=\bigvee_{n\in \mathds{N}} \mathfrak{F}_{s,s+n}.
\end{equation}
 That is, $\mathfrak{F}_{s,+\infty}$ is the smallest $\sigma$-field
which contains all $\mathfrak{F}_{s,s+n}$. Given $s\geq 0$ and $\mu
\in \mathcal{P}_{\rm exp}$, in the definition below -- which is an
adaptation of the definition in \cite[Section 5.1, pages 78,
79]{Dawson}) -- we deal with probability measures $P_{s,\mu}$ on
$(\mathfrak{D}_{[s,+\infty)} (\Gamma_*), \mathfrak{F}_{s,+\infty})$.
\begin{definition}
 \label{A1df}
A family of probability measures $\{P_{s,\mu}: s\geq 0, \ \mu\in
\mathcal{P}_{\rm exp}\}$ is said to be a solution of the restricted
initial value martingale problem for our model if, for all $s\geq 0$
and $\mu \in \mathcal{P}_{\rm exp}$, the following holds: (a)
$P_{s,\mu}\circ \varpi_s^{-1} = \mu$; (b) $P_{s,\mu}\circ
\varpi_t^{-1} \in \mathcal{P}_{\rm exp}$ for all $t >s$; (c)  for
each $F \in \mathcal{D}(L)$ (Definition \ref{THdf}), $t_2 \geq t_1
\geq s$ and any bounded function ${\sf G}:\mathfrak{D}_{[s,+\infty)}
(\Gamma) \to \mathds{R}$ which is $\mathfrak{F}_{s,t_1}$-measurable,
the function
\begin{equation}
  \label{A120}
{\sf H}(\gamma) := \left[F(\varpi_{t_2} (\gamma)) - F(\varpi_{t_1}
(\gamma)) - \int_{t_1}^{t_2} (L F) (\varpi_u (\gamma)) d u \right]
{\sf G}(\gamma)
\end{equation}
is such that
\begin{gather}
 \label{A12}
\int_{\mathfrak{D}_{[s,+\infty)}}{\sf H}(\gamma) P_{s,\mu} (d
\gamma) = 0.
\end{gather}
The restricted initial value martingale problem is said to be {\it
well-posed} if, for each $s\geq 0$ and $\mu\in \mathcal{P}_{\rm
exp}$, there exists a unique $P_{s,\mu}$ satisfying all the
conditions mentioned above. Exactly in the same way one defines
\end{definition}
Here by saying ``for our model'' along with the Kolmogorov operator
$L$ given in (\ref{I5}) we mean also its domain $\mathcal{D}(L)$
(Definition \ref{THdf}) and the class $\mathcal{P}_{\rm exp}$
defined by the property (\ref{Len}). Note that ${\sf H}$ defined in
(\ref{A120}) is $P_{s,\mu}$-integrable, that follows by claim (i) of
Proposition \ref{T3pn}. Note also that the functions $\sf G$ in
(\ref{A120}) can be taken in the form
\begin{equation}
  \label{A1200}
  {\sf G} (\gamma) = F_1 (\varpi_{s_1} (\gamma)) \cdots F_m (\varpi_{s_m}
  (\gamma)),
\end{equation}
with all possible choices $m\in \mathds{N}$, $F_1, \dots, F_m \in
\widetilde{\mathcal{F}}$ (see Proposition \ref{T1pn}), and $s\leq
s_1 < s_2 <\cdots < s_m \leq t_1$, see \cite[eq. (3.4), page
174]{EK}.
\begin{definition}
\label{A01df} For a given $s\geq 0$, a map, $[s,+\infty)\ni t
\mapsto \mu_t\in \mathcal{P}(\Gamma_*)$, is said to be measurable if
the maps $[s,+\infty) \ni t \mapsto \mu_t(\mathbb{A}) \in
\mathds{R}$ are measurable for all $\mathbb{A}\in
\mathcal{B}(\Gamma_*)$. Such a map is said to be a solution of the
Fokker-Planck equation for our model if, for each $F\in
\mathcal{D}(L)$ and any $t_2 > t_1 \geq s$, the following holds
\begin{equation}
  \label{A12a}
  \mu_{t_2}(F) = \mu_{t_1}(F) + \int_{t_1}^{t_2} \mu_{u}(L F) d u.
\end{equation}
\end{definition}
\begin{remark}
  \label{A2rk}
In view of the integral form of (\ref{A12a}), its solutions are
often called \emph{weak}. We do not do this as the precise meaning
of this notion is clear from the definition above. By taking ${\sf
G}\equiv 1$ in (\ref{A120}) one comes to the following conclusion.
Let $\{P_{s,\mu}: s\geq 0, \mu \in \mathcal{P}_{\rm exp}\}$ be a
solution as in Definition \ref{A1df}. Then, for each $s$ and $\mu
\in \mathcal{P}_{\rm exp}$, the map $[s,+\infty) \ni t \mapsto
P_{s,\mu}\circ \varpi_{t}^{-1}$ solves (\ref{A12a}) for all
$t_2>t_1\geq s$.
\end{remark}
Below, by $\mathfrak{D}_{[s,+\infty)} (\breve{\Gamma}_*)$, $s\geq
0$, we mean the space of cadlag maps $[s,+\infty) \ni t \mapsto
\gamma_t \in \breve{\Gamma}_*$, where the latter is the space of
single configurations, see (\ref{Psi10}). Now we formulate our
principal result.
\begin{theorem}
  \label{1tm}
For the model defined in (\ref{I5}) satisfying (\ref{C7}) and
(\ref{C8}), the following is true:
\begin{itemize}
  \item[(a)] The restricted initial value martingale problem
is well-posed in the sense of Definition \ref{A1df}.
\item[(b)] The stochastic
process $X$ related to the family $$(\mathfrak{D}_{[s,+\infty)}
(\Gamma_*), \mathfrak{F}_{s,+\infty}, \{\mathfrak{F}_{s,t}: t\geq
s\}, \{P_{s,\mu}:\mu \in \mathcal{P}_{\rm exp}\})_{s\geq 0}$$ is
Markov. This  means that, for all $t>s$ and $\mathbb{B}\in
\mathfrak{F}_{t, +\infty}$, the following holds
\[
{\rm Prob}(X \in \mathbb{B})=P_{s,\mu}
(\mathbb{B}|\mathfrak{F}_{s,t}) = P_{s,\mu}
(\mathbb{B}|\mathfrak{F}_{t}), \qquad P_{s,\mu} \ - \ {\rm almost} \
{\rm surely}.
\]
Here  $\mathfrak{F}_{t}$ is the smallest $\sigma$-field of subsets
of $\mathfrak{D}_{[s,+\infty)}$ that contains all
$\varpi^{-1}_t(\mathbb{A})$, $\mathbb{A}\in \mathcal{B}(\Gamma)$.
\item[(c)] The aforementioned process has the property
\[
{\rm Prob}\left( X \in \mathfrak{D}_{\mathds{R}_{+}}
(\breve{\Gamma}_*)\right) =1.
\]
\end{itemize}
\end{theorem}
The proof of claim (a) of this statement is the main concern of the
rest of the paper. It will be done in the following two steps. First
we prove that the restricted initial value martingale problem as in
Definition \ref{A1df} has at most one solution. Thereafter, we
construct a solution by `superposing' (cf. \cite{Trev}) the
collection of measures constructed in \cite{asia}.

\subsection{Strategy of the proof and some comments} Our approach is
essentially based on the Fokker-Planck equation (\ref{A1}),
(\ref{A12a}) for which a solution, $t\to \mu_t\in \mathcal{P}_{\rm
exp}$, $\mu_0\in \mathcal{P}_{\rm exp}$ was constructed in
\cite{asia}. In Sect. 6, we introduce approximating models by
modifying the jump kernel in such the way that allows one to solve
the Fokker-Planck equation directly by constructing stochastic
semigroups in a Banach space of signed measures, with the
possibility to take Dirac measures $\delta_\gamma$, $\gamma\in
\Gamma_*$, as the initial conditions. This allows in turn for
introducing finite-dimensional marginals of the presumed law of the
processes corresponding to these approximating models by means of
the transition functions obtained in that way. Then we prove that
these marginals satisfy a Chentsov-like condition (see \cite[Theorem
3.8.8, page 139]{EK}) -- the same for all approximating models. Here
we employ the complete metric of $\Gamma_*$, see (\ref{Psi4}). This
yields the existence of cadlag versions of the approximating
processes and is used in Sect. 7 to prove that their distributions
have accumulating points -- possible distributions of the process in
question. Then we prove that such accumulation points solve the
martingale problem in the sense of Definition \ref{A1df}. To prove
uniqueness we again use the Fokker-Planck equation and the
construction made in \cite{asia}. At this stage -- realized in Sect.
5 -- we show that this equation has a unique solution, which implies
that the mentioned accumulation points have coinciding
one-dimensional marginals. A classical result (see \cite[claim (a)
of Theorem 4.4.2, page 184]{EK}) is that one would have uniqueness
if the one-dimensional marginals were equal for \emph{all} initial
$\mu\in \mathcal{P}(\Gamma_*)$. Since we have such an equality
\emph{only} for $\mu$ from a subset of $\mathcal{P}(\Gamma_*)$, we
turn to the restricted version of the martingale problem
\cite[Chapter 5]{Dawson}. A crucial element of this version is Lemma
\ref{W2lm} that states that a solution of the Fokker-Planck equation
with $\mu_0\in \mathcal{P}_{\rm \exp}$ is also in $\mathcal{P}_{\rm
\exp}$, and its type satisfies $\varkappa_t \leq \varkappa_T$ for
$t\leq T$, where $\varkappa_T$ depends on $T$ and $\varkappa_0$
only. The proof of Lemma \ref{W2lm} is the most technical element of
this part, based on a number of combinatorial results (see also
Appendix). By means of Lemma \ref{W2lm} we then prove (Theorem
\ref{2tm}) that (\ref{A1}) with $\mu_0\in \mathcal{P}_{\rm exp}$ has
a unique solution coinciding with the map $t\to \mu_t$ constructed
in \cite{asia}. This finally yields the uniqueness of the solution.

\section{The Evolution of States on $\Gamma_*$}

As mentioned above, in the proof of Theorem \ref{1tm} we essentially
use the construction of the family of measures $\{\mu_t\}_{t\geq
0}\subset \mathcal{P}_{\rm exp}$ performed in \cite{asia}. Notably,
in this construction, there was used the space of single
configurations $\breve{\Gamma}_*$, which for measures from
$\mathcal{P}_{\rm exp}$ makes no difference, see Remark
\ref{Lena1rk}. Thus, we begin by describing this family in a way
adapted to the present context.

\subsection{Spaces of functions on $\Gamma_0$}

By (\ref{N701}) it follows that each measurable $F$ satisfying
$|F(\gamma)| \leq C \exp(\beta \Psi(\gamma))$ for some positive
$\beta$ and $C$ is $\mu$-absolutely integrable for each $\mu\in
\mathcal{P}_{\rm exp}$. This obviously relates to $F=KG$ with $G\in
B_{\rm bs}$, see Remark \ref{Bbsrk}. For $a$ and $\phi$ as in
(\ref{C7}) and $G\in B_{\rm bs}$, let us consider
\begin{gather}
  \label{A9a}
  (\widehat{L} G)(\eta) = \sum_{\xi \subset \eta} \sum_{x\in
  \xi}  \int_{\mathds{R}^d} a (x-y) e(\tau_y;\xi) e(t_y; \eta
  \setminus \xi) \left[G ( \xi \setminus x\cup y) - G(\xi)\right] d
  y, \\[.2cm]
\nonumber \tau_y (x) := e^{- \phi (x-y)}, \qquad t_y(x) := \tau_y
(x) - 1, \quad x,y \in \mathds{R}^d.
\end{gather}
In (\ref{A9a}), the sums are finite and the integral is convergent
in view of the integrability of the jump kernel $a$. It turns out
that
\begin{equation*}
L K G =K \widehat{L} G,
\end{equation*}
holding for all $G\in B_{\rm bs}$, see \cite[Corollary 4.3 and eq.
(4.7)]{Fiknel}. By (\ref{Lenard}) this yields
\begin{equation}
  \label{A10A}
 \mu(LKG) = \langle \!\langle k_\mu, \widehat{L} G \rangle\!\rangle,
\end{equation}
which by (\ref{I3}) points to the possibility to extend
$\widehat{L}$ from $B_{\rm bs}$ to integrable functions. For a given
$\vartheta \in \mathds{R}$, let $\mathcal{G}_\vartheta$ stand for
the weighted $L^1$-space equipped with the norm
\begin{eqnarray}
  \label{A6}
  |G|_\vartheta & = & \int_{\Gamma_0} |G(\eta)| \exp\left( \vartheta |\eta|
  \right) \lambda (d \eta)\\[.2cm] \nonumber & = & |G(\varnothing)|
  +  \sum_{n=1}^\infty \frac{e^{\vartheta n}}{n!} \int_{(\mathds{R}^d)^n}
  |G^{(n)}(x_1 , \dots , x_n)| d x_1 \cdots d x_n.
\end{eqnarray}
In fact, we have a descending scale
$\{\mathcal{G}_\vartheta:\vartheta\in \mathds{R}\}$ such that
\begin{equation}
  \label{C21A}
\mathcal{G}_{\vartheta'} \hookrightarrow \mathcal{G}_{\vartheta},
\qquad \vartheta' > \vartheta,
\end{equation}
where by $\hookrightarrow$ we mean continuous embedding. For a given
$\vartheta\in \mathds{R}$ and $G\in B_{\rm bs}$, let us estimate
$|\widehat{L} G|_\vartheta$. By means of \cite[Lemma 2.3]{Fiknel},
see also \cite[Lemma 3.1]{asia}, by (\ref{C7}) we get
\begin{eqnarray*}
|\widehat{L} G|_\vartheta &\leq & \int_{\Gamma_0}
e^{\vartheta |\eta|} \bigg{(}\sum_{\xi \subset \eta} \sum_{x\in
\eta\setminus
  \xi}\int_{\mathds{R}^d} a(x-y) \left( |G(\eta\setminus \xi \setminus x \cup y)| \right. \\[.2cm] \nonumber & + & \left. |G(\eta\setminus \xi  )|
  \right) e(|t_y|; \xi) d y \bigg{)} \lambda (d \eta) \\[.2cm] \nonumber & = &
\int_{\Gamma_0}\bigg{(}\int_{\mathds{R}^d} e^{\vartheta |\eta|}
\sum_{x\in \eta} a(x-y) \left( |G(\eta\setminus x\cup y)| +
|G(\eta)| \right)
\\[.2cm]\nonumber & \times &  \bigg{(}\int_{\Gamma_0} e(e^\vartheta |t_y|; \xi) \lambda (d \xi)
\bigg{)} d y \bigg{)} \lambda ( d \eta) \\[.2cm]\nonumber & \leq & 2 \exp\left(e^\vartheta \langle \phi \rangle
\right) \int_{\Gamma_0} e^{\vartheta |\eta|} |\eta| |G(\eta)|
\lambda ( d \eta).
\end{eqnarray*}
To estimate the last line in the latter formula we use the
inequality $x e^{-\alpha x} \leq 1/e \alpha$, both $x, \alpha$
positive, and the fact that $B_{\rm bs}\subset
\mathcal{G}_{\vartheta'}$ for each $\vartheta' > \vartheta$.
Thereafter, we obtain
\begin{equation}
  \label{A11}
  |\widehat{L} G|_\vartheta \leq  \frac{2 }{ e
  (\vartheta' - \vartheta)}  \exp\left(e^\vartheta \langle \phi \rangle
\right) |G|_{\vartheta'}.
\end{equation}
Below by means of this estimate we extend $\widehat{L}$ to operators
acting in the scale $\{\mathcal{G}_\vartheta\}_{\vartheta\in
\mathds{R}}$, cf. (\ref{C21A}).

Along with $\mathcal{G}_\vartheta$ we introduce the following Banach
spaces. For symmetric $k^{(n)}\in L^\infty((\mathds{R}^d)^n)$, $n\in
\mathds{N}$, let $k$ be defined by $k^{(n)}$ as in (\ref{C22b}),
that includes also some constant $k(\varnothing)=k^{(0)}$. Such $k$
constitute a real linear space and can be considered as essentially
bounded functions $k:\Gamma_0 \to \mathds{R}$. Note that the
correlation functions $k_\mu$, cf. (\ref{I3}), are such functions.
Then for $\vartheta\in \mathds{R}$, we define
\begin{equation*}
  \|k\|_\vartheta = \sup_{n\geq 0} \bigg{(}\|k^{(n)}\|_{L^\infty((\mathds{R}^d)^n)}e^{-\vartheta n} \bigg{)}=
  \esssup_{\eta \in \Gamma_0}\bigg{(} |k(\eta)|\exp\left(-\vartheta |\eta|\right)\bigg{)}.
\end{equation*}
The linear space $\mathcal{K}_\vartheta$ equipped with this norm is
the Banach space in question. Clearly, cf. (\ref{C21A}),
\begin{equation}
  \label{C21b}
  \mathcal{K}_{\vartheta } \hookrightarrow \mathcal{K}_{\vartheta'
  }, \quad {\rm for} \ \vartheta < \vartheta'.
\end{equation}
Note  that $\mathcal{K}_\vartheta$ is the topological dual to
$\mathcal{G}_\vartheta$ as the value of $k$ on $G$ is given by the
formula
\[
\langle \! \langle k , G \rangle \! \rangle = \int_{\Gamma_0}
k(\eta) G (\eta) \lambda ( d \eta).
\]
Let us now define $L^\Delta$ by the condition, cf. (\ref{A10A}),
\begin{equation}
  \label{A13a}
\langle \! \langle L^\Delta k_\mu, G \rangle\!
  \rangle = \langle \! \langle k_\mu, \widehat{L} G \rangle\!
  \rangle.
\end{equation}
By (\ref{A9a}) it is obtained in the following form, see \cite[eqs.
(2.21), (2.22)]{asia},
\begin{eqnarray}
  \label{K3}
(L^{\Delta} k)(\eta) & = & \sum_{y\in \eta} \int_{\mathds{R}^d}
a(x-y) e(\tau_y; \eta \setminus y\cup x)( W_y k)(\eta \setminus y
\cup x) d x \\[.2cm]\nonumber  &-& \sum_{x\in \eta}
\int_{\mathds{R}^d} a (x-y) e(\tau_y; \eta)( W_y k)(\eta) d y,
\end{eqnarray}
where
\begin{equation}
  \label{K4}
( W_y k)(\eta) = \int_{\Gamma_0} k(\eta\cup \xi) e(t_y ; \xi)
\lambda ( d \xi).
\end{equation}
Proceeding similarly as in obtaining (\ref{A11}), for all
$\vartheta\in \mathds{R}$ and $\vartheta'>\vartheta$, we get
\begin{equation}
  \label{A14}
  \|L^\Delta k \|_{\vartheta'} \leq \frac{2}{e(\vartheta'-\vartheta)}
  \exp\left(e^{\vartheta}\langle \phi \rangle
  \right)\|k\|_{\vartheta},
\end{equation}
where we have taken into account that $\langle a \rangle =1$, see
(\ref{C7}).

\subsection{The evolution in spaces of functions on $\Gamma_0$}
\label{SS4.2}

 By combining (\ref{A10A}) with (\ref{A13a}) we
introduce the following versions of the Kolmogorov equation
(\ref{I4})
\begin{equation}
  \label{M1}
 \frac{d}{dt} G_t = \widehat{L} G_t, \qquad G_{t}|_{t=0} = G_0,
\end{equation}
\begin{equation}
  \label{M2}
 \frac{d}{dt} k_t = L^\Delta k_t, \qquad k_{t}|_{t=0} = k_0,
\end{equation}
which we will solve in the scales $\{\mathcal{G}_\vartheta:
\vartheta\in \mathds{R}\}$ and $\{\mathcal{K}_\vartheta:
\vartheta\in \mathds{R}\}$, respectively, see (\ref{C21A}) and
(\ref{C21b}).

Let us first consider (\ref{M2}). By (\ref{A14}) we see that
$L^\Delta$ maps each $\mathcal{K}_\vartheta$ in each
$\mathcal{K}_{\vartheta'}$, cf. (\ref{C21b}), and the corresponding
map is linear and bounded. Likewise, one can define the linear maps
$(L^\Delta)^n:\mathcal{K}_{\vartheta} \to \mathcal{K}_{\vartheta'}$,
$n\in \mathds{N}$ the norm of which can be estimated by means of the
inequality
\begin{equation}
 \label{A15}
\|(L^\Delta)^n k \|_{\vartheta'} \leq \frac{n^n }{\left(e
T(\vartheta', \vartheta) \right)^n} \|k\|_{\vartheta},
\end{equation}
where, cf. \cite[eq. (4.2)]{asia},
\begin{equation}
  \label{U1}
T(\vartheta_2, \vartheta_1) = \frac{\vartheta_2-\vartheta_1}{2}
\exp\left( - \langle \phi \rangle e^{\vartheta_2} \right), \qquad
\vartheta_2 > \vartheta_1.
\end{equation}
It is known, see \cite[eqs. (4.3), (4.4)]{asia} that
\begin{equation}
  \label{U2}
 \sup_{\vartheta'> \vartheta} T(\vartheta',\vartheta) =
\frac{\delta(\vartheta)}{2}\exp\left(-\frac{1}{\delta(\vartheta)}
\right)=:\tau (\vartheta),
\end{equation}
where $\delta (\vartheta)$ is a unique solution of $\delta e^\delta
= e^{-\vartheta}/\langle \phi \rangle$. The supremum in (\ref{U2})
is attained at $\vartheta' = \vartheta + \delta (\vartheta)$. Then
the expression in (\ref{K3}), (\ref{K4}) can be used to define: (a)
bounded linear operators $(L^\Delta)^n_{\vartheta'\vartheta}:
\mathcal{K}_{\vartheta} \to \mathcal{K}_{\vartheta'}$, $n\in
\mathds{N}$ the norm of which can be estimated by means of
(\ref{A15}); (b) unbounded linear operators $L^\Delta_{\vartheta'}$
with domains, cf. \cite[eq. 3.19)]{asia},
\begin{equation*}
{\rm Dom}L^\Delta_{\vartheta'} = \{k\in \mathcal{K}_{\vartheta'}:
L^\Delta k\in \mathcal{K}_{\vartheta'}\}..
\end{equation*}
Now we turn to (\ref{M1}). In a similar way, by means of (\ref{A11})
one defines: (a) bounded linear operators
$(\widehat{L})^n_{\vartheta\vartheta'}: \mathcal{G}_{\vartheta'} \to
\mathcal{G}_{\vartheta}$, $n\in \mathds{N}$, the norm of which
satisfies
\begin{equation}
 \label{A17}
\|(\widehat{L})^n_{\vartheta\vartheta'}\| =
\|(L^\Delta)^n_{\vartheta'\vartheta}\|, \qquad n \in \mathds{N};
\end{equation}
(b) unbounded operators $\widehat{L}_\vartheta$ with domains
\begin{equation*}
{\rm Dom} \widehat{L}_{\vartheta} = \{G\in \mathcal{G}_{\vartheta}:
\widehat{L} G\in \mathcal{G}_{\vartheta}\}..
\end{equation*}
It can be shown, see \cite[Lemma 3.1]{asia}, that, for each
$\vartheta \in \mathds{R}$ and $\vartheta'>\vartheta$, the following
is true
\begin{equation*}
\mathcal{K}_\vartheta \subset {\rm Dom}L^\Delta_{\vartheta'}, \qquad
\mathcal{G}_{\vartheta'} \subset {\rm Dom}\widehat{L}_{\vartheta},
\end{equation*}
by which one readily obtains that, for all $\vartheta, \vartheta'$,
$\vartheta'>\vartheta$, the following holds
\begin{equation}
 \label{A20}
\forall k\in \mathcal{K}_\vartheta \qquad \ \ L^\Delta_{\vartheta'\vartheta} k =
L^\Delta_{\vartheta'} k.
\end{equation}
Furthermore, up to the embedding (\ref{C21b}) we have that
\begin{equation*}
L^\Delta_{\vartheta'\vartheta} k = L^\Delta_{\vartheta''\vartheta} k,
\end{equation*}
holding for all $\vartheta''\in (\vartheta, \vartheta')$.
By (\ref{A15}) the series
\begin{equation}
 \label{A22}
Q_{\vartheta'\vartheta} (t) = 1 + \sum_{n=1}^{\infty} \frac{t^n}{n!}
(L^\Delta)^n_{\vartheta'\vartheta}
\end{equation}
converges in the operator norm topology -- uniformly on
compact subsets of $[0,T(\vartheta'\vartheta))$ -- to a bounded linear operator
$$Q_{\vartheta'\vartheta} (t) :\mathcal{K}_\vartheta \to
{\rm Dom}L^\Delta_{\vartheta'} \subset \mathcal{K}_{\vartheta'},$$
the norm of which satisfies
\begin{equation}
 \label{A23}
\| Q_{\vartheta'\vartheta} (t)\| \leq \frac{T(\vartheta',\vartheta)}{T(\vartheta',\vartheta)-t}.
\end{equation}
Moreover, the map $[0,T(\vartheta', \vartheta)) \ni t \mapsto Q_{\vartheta'\vartheta} (t)$
is differentiable and the following holds
\begin{equation}
 \label{A24}
\frac{d}{dt} Q_{\vartheta'\vartheta} (t)= L^\Delta_{\vartheta'}
Q_{\vartheta'\vartheta} (t) = L^{\Delta}_{\vartheta'\vartheta''}
Q_{\vartheta''\vartheta} (t) = Q_{\vartheta'\vartheta''} (t)
L^{\Delta}_{\vartheta''\vartheta},
\end{equation}
with an arbitrary $\vartheta''\in (\vartheta, \vartheta')$ provided
$t$ satisfies $t< T(\vartheta'', \vartheta)$ and $t< T(\vartheta',
\vartheta'')$ in the latter two terms, respectively, cf.
(\ref{A23}). By (\ref{A24}) one readily obtains that the Cauchy
problem in (\ref{M2}) with $k_0\in \mathcal{K}_\vartheta$ has a
unique classical solution in $\mathcal{K}_{\vartheta'}$, on the time
interval $[0, T(\vartheta',\vartheta))$, cf. \cite[Lemma 4.1]{asia}.
It is
\begin{equation}
 \label{A26}
k_t = Q_{\vartheta'\vartheta}(t) k_0.
\end{equation}
In a similar way, one shows that the Cauchy problem in (\ref{M1})
has a unique classical solution in $\mathcal{G}_{\vartheta}$, on the
time interval $[0, T(\vartheta',\vartheta))$, given by the formula
\begin{equation}
 \label{A28}
G_t = H_{\vartheta \vartheta'}(t) G_0 =: \left(1 + \sum_{n=1}^{\infty} \frac{t^n}{n!}
(\widehat{L})^n_{\vartheta\vartheta'} \right) G_0, \qquad G_0 \in \mathcal{G}_{\vartheta'}.
\end{equation}
By construction these solutions of (\ref{M2}) and (\ref{M1}) satisfy
\begin{equation}
 \label{A29}
\langle \! \langle k_t , G_0 \rangle\! \rangle = \langle \! \langle k_0 , G_t \rangle\! \rangle,
\qquad t< T(\vartheta', \vartheta).
\end{equation}

\subsection{The evolution of states}

\label{SS4.3}

A priori, the solution given in (\ref{A26}) need not be the
correlation function for any measure. Moreover, it may not even be
positive, cf. (\ref{I3}). To check whether a given $k:\Gamma_0\to
\mathds{R}_{+}$ is the correlation function of a certain $\mu\in
\mathcal{P}_{\rm exp}$ we introduce the following set
\begin{equation}
  \label{C30}
  B^\star_{\rm bs}=\{ G\in B_{\rm bs}: \sum_{\xi\subset \eta} G(\xi)
  \geq 0, \  {\rm for} \ {\rm all} \ \eta \in \Gamma_0\}.
\end{equation}
Note that some of its members can take also negative values. By
\cite[Theorems 6.1 and 6.2 and Remark 6.3]{Tobi} one proves the
following statement.
\begin{proposition}
  \label{Tobipn}
Let a measurable function, $k:\Gamma_0\to \mathds{R}$, have the
following properties:
\begin{itemize}
  \item[(a)] $\langle \! \langle k, G\rangle \! \rangle \geq 0$ for
  all $G\in B^\star_{\rm bs}$, see (\ref{C30});
  \item[(b)] $k(\varnothing) =1$;
  \item[(c)] $k(\eta) \leq \varkappa^{|\eta|}$ for some
  $\varkappa>0$, cf. (\ref{I3}).
\end{itemize}
Then $k$ is the correlation function for a unique $\mu \in
\mathcal{P}_{\rm exp}$.
\end{proposition}
Recall that the least $\varkappa$ as in item (c) above is the type
of $\mu$ of which $k$ is then the correlation function. Set
\begin{equation}
  \label{C30A}
  \mathcal{P}_{\rm exp}^\vartheta = \{\mu \in \mathcal{P}_{\rm exp}:
\mu \ {\rm is} \ {\rm of} \ {\rm type} \leq  e^\vartheta\}.
\end{equation}
Let $\mathcal{K}^\star$ be the set of all $k:\Gamma_0\to \mathds{R}$
that possess the properties listed in Proposition \ref{Tobipn}. In
\cite[Theorem 3.3]{asia}, it was shown that $k_t$ as in (\ref{A26})
belongs to $\mathcal{K}^\star$ whenever $k_0$ is the correlation
function of a certain $\mu\in \mathcal{P}_{\rm exp}$. In the context
of the present study, the relevant results of \cite{asia} can be
formulated as follows.
\begin{proposition}
  \label{3.3pn}
Given $\vartheta_0\in \mathds{R}$, let $\mu$ be an arbitrary element
of $\mathcal{P}_{\rm exp}^{\vartheta_0}$. For this $\vartheta_0$,
set $\vartheta_t = \vartheta_0 + t$, $t\geq 0$. Then there exists a
unique map, $[0,+\infty) \ni t \mapsto k_t \in \mathcal{K}^\star$,
such that $k_0 = k_\mu$ and the following holds:
\begin{itemize}
  \item[(a)] for each $t>0$,
  \[
0\leq k_t (\eta) \leq e^{\vartheta_t |\eta|} , \qquad \eta
\in \Gamma_0,
  \]
by which $k_t\in \mathcal{K}_{\vartheta_t}$.
 \item[(b)] For each $T>0$ and $t\in [0,T)$, the map $t\mapsto k_t
 \in \mathcal{K}_{\vartheta_t} \subset {\rm Dom}
 L^{\Delta}_{\vartheta_T}$ is continuous on $[0, T)$ and continuously differentiable on $(0,
T)$ in $\mathcal{K}_{\vartheta_T}$ and the following holds:
\begin{equation}
  \label{C31}
 \frac{d}{dt} k_t = L^{\Delta}_{\vartheta_T} k_t.
\end{equation}
\end{itemize}
\end{proposition}

\section{The Uniqueness}

In this section, we prove that the restricted initial value
martingale problem has at most one solution. To this end we use the
properties of $\mathcal{D}(L)$ stated in Proposition \ref{T3pn}. In
view of Remark \ref{A2rk}, see also Lemma \ref{A3pn} below, the
proof of the uniqueness in question amounts to proving that, for
each $\mu\in \mathcal{P}_{\rm exp}$, the Fokker-Planck equation
(\ref{A12a}) has at most one solution $\mu_t \in \mathcal{P}_{\rm
exp}$ satisfying $\mu_0=\mu$. The main tool for this is based on
controlling the type of $\mu_t$ by a method based on the use of the
concrete form of the elements of $\mathcal{D}(L)$, see Definition
\ref{THdf}.

\subsection{Solving the Fokker-Planck equation}

We begin by pointing out that in Definition \ref{A01df} we do not
assume that $\mu_t\in \mathcal{P}_{\rm exp}$ for $t>0$. Recall that
$(KG)(\gamma) = \sum_{\xi\Subset \gamma} G(\xi)$, see (\ref{A2}).
\begin{lemma}
  \label{W2lm}
Let $[0,+\infty) \ni t\mapsto \mu_t\in \mathcal{P}(\Gamma_*)$ be a
solution of (\ref{A12a}) with  all $F$ belonging to the linear span
of $\widehat{\mathcal{F}}$ and a given $\mu_0 \in \mathcal{P}_{\rm
exp}^{\vartheta_0}$. Then, for each $T>0$, there exists
$\vartheta_T\in \mathds{R}$ such that, for all $t\in [0,T]$,
$\mu_t\in \mathcal{P}_{\rm exp}^{\vartheta_t}$ with some
$\vartheta_t < \vartheta_T$.
\end{lemma}
Note that here we assume that {\it only} the initial state $\mu_0$
belongs to $\mathcal{P}_{\rm exp}$. Also, we assume that $\mu_t$
solves (\ref{A12a}) with $F$ belonging {\it only} to a subset of
$\mathcal{D}(L)$. It turns out that this is enough to solve it for
all $\mathcal{D}(L)$, and even more. Set
\begin{equation}
  \label{WA}
  \mathcal{F}= \left\{ F\in B_{\rm b}(\Gamma_* ): F = KG , \ G \in
  \bigcap_{\vartheta\in \mathds{R}} \mathcal{G}_\vartheta\right\},
\end{equation}
where $K$ is defined in (\ref{A2}) and $G$ is supposed to be such
that $|G|_\vartheta$ is finite for all $\vartheta$, see (\ref{A6}).
Let us show that $\mathcal{D}(L) \subset \mathcal{F}$. Since $K$ is linear, this will
follow from the fact that
\begin{equation}
  \label{WA1}
  \widetilde{\mathcal{F}} \cup \widehat{\mathcal{F}} \subset \mathcal{F}.
\end{equation}
 By (\ref{T4}) we have
\begin{eqnarray}
  \label{T4AA}
  \widetilde{F}_\tau^\theta (\gamma)& = & \prod_{x\in \gamma} (1+
  \theta(x)) e^{-\tau \psi(x)} = \sum_{\eta \subset
  \gamma} e(\theta_\tau ;\eta) =: (K \widetilde{G}_\tau^\theta)(\gamma) , \\[.2cm] \nonumber
  \theta_\tau (x) & := & \theta(x) e^{-\tau \psi(x)} +\psi_\tau (x), \qquad \psi_\tau(x) = -1 +e^{-\tau
  \psi(x)}.
\end{eqnarray}
Clearly, $\theta_\tau\in L^1 (\mathds{R}^d)$ for each $\tau \geq 0$
and $\theta\in \varTheta_\psi$, cf. Definition \ref{THdf}. Then
$\widetilde{G}_\tau^\theta = e(\theta_\tau;\cdot)\in
\mathcal{G}_\vartheta$ for any $\vartheta\in \mathds{R}$, which
yields $ \widetilde{\mathcal{F}} \subset \mathcal{F}$.

In the case of $F$ given in (\ref{TH1}), (\ref{TH1z}), we write
\begin{eqnarray}
  \label{T4AB}
 \widehat{F}_\tau^{\theta_1, \dots , \theta_m}(\gamma)& = & \sum_{\xi\subset \gamma} g_m(\xi) \prod_{x\in
 \gamma\setminus\xi} \left(1 + \psi_\tau (x)
 \right)\\[.2cm] \nonumber& = & \sum_{\eta \subset \gamma} \left( \sum_{\xi\subset \eta} g_m(\xi) e(\psi_\tau;\eta \setminus \xi)\right)=: \sum_{\eta \subset \gamma}\widehat{G}_\tau^{\theta_1, \dots ,
 \theta_m}(\eta),
\end{eqnarray}
where $\psi_\tau (x)$ is as in (\ref{T4AA}) and
\begin{equation*}
g_m (\xi) = \left\{ \begin{array}{ll} \sum_{\sigma\in S_m} \theta_1(x_{\sigma(1)}) \cdots \theta_m(x_{\sigma(m)}), \qquad &{\rm if} \ \xi =\{x_1 , \dots , x_m\}; \\[.3cm] 0 \qquad & {\rm otherwise}. \end{array} \right.
\end{equation*}
Let us estimate $\widehat{G}_\tau^{\theta_1, \dots ,
 \theta_m}$ with  $\theta_1, \dots , \theta_m \in \varTheta_\psi^{+}$. For $\tau\in (0,1]$, we have $|\psi_\tau (x)| \leq
 \psi(x)$, and hence
\begin{gather}
  \label{T4ABz}
 \left| \widehat{G}_\tau^{\theta_1, \dots ,
 \theta_m}(\eta)\right| \leq \sum_{\xi\subset \eta} g_m(\xi) e(\psi;
 \eta \setminus \xi).
\end{gather}
At the same time, for each $\eta\in \Gamma_0$, it follows that
\begin{equation}
  \label{T4ABy}
\widehat{G}_\tau^{\theta_1, \dots ,
 \theta_m}(\eta) \to g_m(\eta), \qquad \tau \to 0^{+}.
\end{equation}
By (\ref{T4ABz}) let us show that $\widehat{G}_\tau^{\theta_1, \dots
,
 \theta_m}$ belongs to
 $\mathcal{G}_\vartheta$, $\vartheta \in \mathds{R}$
 Indeed, by (\ref{A6}) we have
 \begin{eqnarray}
  \label{A6A}
 | \widehat{G}_\tau^{\theta_1, \dots ,
 \theta_m}|_\vartheta & = & \int_{\Gamma_0} |\widehat{G}_\tau^{\theta_1, \dots ,
\theta_m}(\eta)|e^{\vartheta |\eta|} \lambda (d\eta) \\[.2cm] \nonumber  & \leq & \int_{\Gamma_0}
\int_{\Gamma_0} e^{\vartheta|\xi|} g_m(\xi) e^{\vartheta|\eta|}
e(\psi; \eta )
 \lambda ( d \xi) \lambda ( d \eta)\\[.2cm] \nonumber & \leq & e^{m \vartheta} \langle \theta_1 \rangle \cdots
 \langle \theta_m \rangle \exp\left( e^\vartheta  \langle \psi \rangle  \right)\\[.2cm] \nonumber & =: &
 \delta_m^{\theta_1, \dots , \theta_m} (\vartheta),
 \end{eqnarray}
where $\langle \psi \rangle$, $\langle \theta_i \rangle$, $i=1,\dots
, m$ are the corresponding $L^1$-norms, cf. (\ref{psi}). This
completes the proof of (\ref{WA1}).

Lemma \ref{W2lm}  is proved below. Now assuming that its claim holds
true, we prove the next statement -- one of the two basic tools of
proving Theorem \ref{1tm}.
\begin{theorem}
 \label{2tm}
For each $\mu_0\in \mathcal{P}_{\rm exp}$, the solution of the
Fokker-Planck equation in the sense of Definition \ref{A01df} exists
and is unique.
\end{theorem}
\begin{proof}
We begin by showing that (\ref{A12a}) has a solution.  Take
$G=\widehat{G}_\tau^{\theta_1, \dots , \theta_m}$ and let $k_t$ be
as in Proposition \ref{3.3pn} with $k_0$ being the correlation
function of the initial state $\mu_0$. Since $k_t$ is in
$\mathcal{K}^\star$, by Proposition \ref{Tobipn} it determines a
unique $\mu_t\in \mathcal{P}_{\rm exp}^{\vartheta_t}$, see
(\ref{C30A}), for which
\begin{equation}
  \label{C101}
\mu_t (F) = \mu_t (KG )= \langle \! \langle k_t, G\rangle \!
\rangle, \qquad t\geq 0,
\end{equation}
holding for all $F\in \widehat{\mathcal{F}}$. By claim (a) of
Proposition \ref{3.3pn} and (\ref{WA}), (\ref{WA1}) the integral in
the right-hand side of (\ref{C101}) is absolutely convergent for
each $t\geq 0$. Moreover, by claim (b) of Proposition \ref{3.3pn} we
have that
$$k_{t_2} - k_{t_1} = \int_{t_1}^{t_2} L^\Delta_{\vartheta_T} k_{u} du$$ holding
for all $t_2>t_1 \geq 0$ and $T> t_2$. We multiply both parts of the
latter equality by an arbitrary $G\in \cap_{\vartheta\in \mathds{R}}
\mathcal{G}_\vartheta$ -- also corresponding to $F\in \mathcal{F}$
-- and then integrate with respect to $\lambda$. By claim (b) of
Proposition \ref{3.3pn} this integration and that over $[t_1, t_2]$
can be interchanged, that implies
\begin{equation}
 \label{AUx}
\mu_{t_2}(F) - \mu_{t_1}(F) = \int_{t_1}^{t_2}\langle \! \langle
L^\Delta_{\vartheta_T} k_u, G\rangle \!\rangle d u =
\int_{t_1}^{t_2} \langle \! \langle
 k_u, \widehat{L}G\rangle \!\rangle du  = \int_{t_1}^{t_2} \mu_u (L
 F) du,
 \end{equation}
where we have used (\ref{A10A}), (\ref{A13a}) and the fact that
$G\in \cap_{\vartheta\in \mathds{R}} \mathcal{G}_\vartheta$. This
yields (\ref{A12a}). By (\ref{WA1}) we then get that $\mu_t$
corresponding to $k_t$ is a solution.

Assume now that there exists another solution, say
$\{\tilde{\mu}_t\}_{t\geq 0}\subset \mathcal{P}(\Gamma_*)$, such
that $\tilde{\mu}_0 = \mu_0$. By Lemma \ref{W2lm} we have that
$\tilde{\mu}_t\in \mathcal{P}_{\rm exp}^{\tilde{\vartheta}_t}$ and
$\tilde{\vartheta}_t \in(\vartheta_0, \tilde{\vartheta}_T)$ for some
$\tilde{\vartheta}_T$ and all $t\leq T$. This means that the
corresponding correlation functions, $\tilde{k}_t$, $t\leq T$ belong
to $\mathcal{K}_{\tilde{\vartheta}_t}$. Then the vector $q_u =
L^\Delta_{\tilde{\vartheta}_T}\tilde{k}_u =
L^\Delta_{\tilde{\vartheta}_T \tilde{\vartheta}_u}\tilde{k}_u$, see
(\ref{A20}), lies in $\mathcal{K}_{\tilde{\vartheta}_T}$, and hence
in $\mathcal{K}_{\tilde{\vartheta}_T+\varepsilon}$ for each
$\varepsilon >0$, see (\ref{C21b}). Then, for a fixed $\varepsilon$,
by (\ref{A15}) and (\ref{I3}) we have
\begin{equation}
 \label{A6B}
 \|q_u\|_{\tilde{\vartheta}_T+\varepsilon} \leq C (T, \varepsilon) e^{\tilde{\vartheta}_T}, \qquad u\in [0,t],
\end{equation}
with $C (T, \varepsilon) = 1/e T(\tilde{\vartheta}_T+\varepsilon,
\tilde{\vartheta}_T)$, see (\ref{U1}). Let us prove that the
following holds
\begin{equation}
 \label{C101a}
 \forall G \in \bigcap_{\vartheta\in \mathds{R}} \mathcal{G}_\vartheta \qquad \langle \! \langle \tilde{k}_t - k_0,G \rangle \! \rangle = \int_0^t
\langle \! \langle  q_u, G\rangle \! \rangle du.
\end{equation}
A priori, the equality in (\ref{C101a}) holds for only for $G$
corresponding to $F\in \mathcal{D}(L)$, that includes
$G=\widehat{G}_\tau^{\theta_1, \dots , \theta_m}$, see (\ref{T4AB}).
For $\tau\in (0,1]$, by (\ref{A6A}) and (\ref{A6B}) we then have
\begin{eqnarray*}
 & & \left| \langle \! \langle \tilde{k}_t - k_0,\widehat{G}_\tau^{\theta_1, \dots , \theta_m} \rangle \! \rangle \right|\leq
 \langle \! \langle \left| \tilde{k}_t - k_0\right|,\left|\widehat{G}_\tau^{\theta_1, \dots , \theta_m}\right| \rangle \! \rangle
 \leq 2 e^{\tilde{\vartheta}_T}\delta_m^{\theta_1, \dots , \theta_m} (\tilde{\vartheta}_T + \varepsilon), \\[.2cm] \nonumber & &
 \left| \int_0^t
\langle \! \langle  q_u, \widehat{G}_\tau^{\theta_1, \dots ,
\theta_m}\rangle \! \rangle du \right| \leq \int_0^t \langle \!
\langle \left| q_u\right|,\left| \widehat{G}_\tau^{\theta_1, \dots ,
\theta_m}\right|\rangle \! \rangle du  \leq t C(T,\varepsilon)
e^{\tilde{\vartheta}_T} \delta_m^{\theta_1, \dots , \theta_m} (
\tilde{\vartheta}_T + \varepsilon).
\end{eqnarray*}
Now we write (\ref{C101a}) for $G=\widehat{G}_\tau^{\theta_1, \dots
, \theta_m}$ and pass  to the limit $\tau\to 0^{+}$. By the
dominated convergence theorem and (\ref{T4ABy}) we then obtain
\begin{eqnarray}
 \label{C101c}
& & \int_{(\mathds{R}^d)^m} \left[\tilde{k}_t^{(m)} ( x_1 , \dots , x_m) - k_0^{(m)} ( x_1 , \dots , x_m)    \right] \theta_1 (x_1) \cdots \theta_m(x_m) d x_1 \cdots d x_m \\[.2cm] \nonumber & & \quad  =  \int_0^t\left(\int_{(\mathds{R}^d)^m} q_u^{(m)} ( x_1 , \dots , x_m)  \theta_1 (x_1) \cdots \theta_m(x_m)d x_1 \cdots d x_m \right) d u,
\end{eqnarray}
that holds for all $m\in \mathds{N}$ and $\theta_1, \dots ,
\theta_m\in \varTheta_\psi^+$, see (\ref{T1}). For a fixed $m\in
\mathds{N}$, the set of functions $(x_1, \dots , x_m) \mapsto
\theta_1(x_1) \cdots \theta_m(x_m)$ with $\theta_1 \dots , \theta_m
\in \varTheta_{\psi}^+$ is closed with respect to the pointwise
multiplication and separates points of $(\mathds{R}^{d })^m$. Such
functions vanish at infinity and are everywhere positive. Then by
the corresponding version of the Stone-Weierstrass theorem \cite{dB}
the linear span of this set is dense (in the supremum norm) in the
algebra $C_0((\mathds{R}^{d })^m)$ of all continuous functions that
vanish at infinity. At the same time, $C_0((\mathds{R}^{d })^m)\cap
L^1((\mathds{R}^{d })^m)$ is dense in $L^1((\mathds{R}^{d })^m)$.
For its subset $C_{\rm cs} ((\mathds{R}^{d })^m)$ has this property.
This allows us to extend the equality in (\ref{C101c}) to the
following
\begin{eqnarray*}
& & \int_{(\mathds{R}^d)^m} \left[\tilde{k}_t^{(m)} ( x_1 , \dots , x_m) - k_0^{(m)} ( x_1 , \dots , x_m)    \right] G^{(m)}(x_1, \dots, x_m) d x_1 \cdots d x_m \\[.2cm] \nonumber & & \quad  =  \int_0^t\left(\int_{(\mathds{R}^d)^m} q_u^{(m)} ( x_1 , \dots , x_m) G^{(m)}(x_1, \dots, x_m)d x_1 \cdots d x_m \right) d u,
\end{eqnarray*}
holding for all $G^{(m)} \in L^1((\mathds{R}^{d })^m)$. Then the
passage from this equality to that in (\ref{C101a}) follows by the
fact that $G$ belongs to each $\mathcal{G}_\vartheta$, $\vartheta\in
\mathds{R}$.

By (\ref{A13a}) the equality in (\ref{C101a}) yields
\begin{equation}
  \label{C102}
\langle \! \langle \tilde{k}_t, G \rangle \! \rangle = \langle \!
\langle k_0 , G\rangle \! \rangle + \int_0^t \langle \! \langle
\tilde{k}_u, \widehat{L}_{\tilde{\vartheta}_t \vartheta_0}G\rangle
\! \rangle du,
\end{equation}
in which $\widehat{L}_{\tilde{\vartheta}_t \vartheta_0}G =: G_1 \in
\cap_{\vartheta\in \mathds{R}}\mathcal{G}_{\vartheta}$. In view of
(\ref{C101a}), we can repeat (\ref{C102}) with $G_1$ instead of $G$,
and repeat this procedure again by employing the same arguments.
After repeating $n$ times we arrive at
\begin{gather*}
\langle \! \langle \tilde{k}_t, G \rangle \! \rangle = \langle \!
\langle k_0 , G\rangle \! \rangle + t \langle\! \langle k_0 ,
\widehat{L}_{\tilde{\vartheta}_t \vartheta} G\rangle \! \rangle +
\frac{t^2}{2}\langle\! \langle k_0 ,
(\widehat{L}_{\tilde{\vartheta}_t \vartheta})^2 G\rangle \!
\rangle \\[.2cm] \nonumber + \cdots + \frac{t^{n-1}}{(n-1)!}\langle\! \langle k_0 ,
(\widehat{L}_{\tilde{\vartheta}_t \vartheta})^{n-1} G\rangle \!
\rangle +\int_0^t\int_0^{t_1} \cdots \int_0^{t_{n-1}} \langle \!
\langle \tilde{k}_{t_n}, (\widehat{L}_{\tilde{\vartheta}_t
\vartheta})^{n} G \rangle \! \rangle d t_1 \cdots d t_n.
\end{gather*}
Assume now that $\tilde{\vartheta}_T > \vartheta_0  + T$, see
Proposition \ref{3.3pn}, that is clearly possible by (\ref{C21b}).
Then we write down the same formula -- in the same spaces -- for
$k_t$ considered in (\ref{C101}), i.e., described in Proposition
\ref{3.3pn}. This yields
\begin{gather*}
\langle \! \langle \tilde{k}_t - k_t, G \rangle \! \rangle =
\int_0^t\int_0^{t_1} \cdots \int_0^{t_{n-1}} \langle \! \langle
\tilde{k}_{t_n}- k_{t_n}, (\widehat{L}_{\tilde{\vartheta}_t
\vartheta})^{n} G \rangle \! \rangle d t_1 \cdots d t_n.
\end{gather*}
Now we take $\vartheta = \tilde{\vartheta}_t+
\delta(\tilde{\vartheta}_t)$, see (\ref{U2}). Then by (\ref{A15}),
(\ref{A17}) and (\ref{U2}) we have from the latter
\begin{gather}
  \label{C105}
\left\vert \langle \! \langle \tilde{k}_t - k_t, G \rangle \!
\rangle \right\vert \leq \frac{n^n}{n! e^n} \left(
\frac{t}{\tau(\tilde{\vartheta}_t)}\right)^n |G|_{\vartheta}
\sup_{u\in [0,t]} \left(\|\tilde{k}_u\|_{\tilde{\vartheta}_t} +
\|k_u\|_{\tilde{\vartheta}_t} \right).
\end{gather}
Note that here $\tau (\tilde{\vartheta}_t) \geq \tau
(\tilde{\vartheta}_T)$. Then for $t<\tau (\tilde{\vartheta}_T)$, the
right-hand side of (\ref{C105}) can be made as small as one wants by
taking big enough $n$. Since $G\in \mathcal{G}_\vartheta$ is
arbitrary, this yields $\tilde{k}_t = k_t$ for all such $t$. The
latter implies $\tilde{\mu}_t = \mu_t$, see Proposition \ref{3.3pn}.
The continuation to bigger values of $t$ is made by repeating the
same procedure. The proof that these continuations cover the whole
$\mathds{R}_{+}$ can be done similarly as in the proof of Theorem
3.3 \cite{asia}.\end{proof}
\begin{corollary}
 \label{Walm}
Let $t\mapsto \mu_t$ satisfy the assumptions of Lemma \ref{W2lm}.
Then it solves (\ref{A12a}) with all $F=KG$ with $G\in
\cap_{\vartheta\in \mathds{R}} \mathcal{G}_\vartheta$, also for
unbounded ones.
\end{corollary}
\begin{proof}
By Lemma \ref{W2lm} a solution $\mu_t$ is in $\mathcal{P}_{\rm
exp}^{\vartheta_T}$ for $t\leq T$. Let $k_t$ be its correlation
function, which satisfies the equality in (\ref{C101a}) with $G=
\widehat{G}_\tau^{\theta_1, \dots , \theta_m}$. As we have shown in
the proof of Theorem \ref{2tm} it satisfies this equality for all
$G$ such that $F=KG$ with $G\in
\cap_{\vartheta}\mathcal{G}_\vartheta$, see (\ref{C101a}). This
yields the proof.
\end{proof}

\subsection{Further properties of the solutions}

In this subsection, we prepare proving Lemma \ref{W2lm}. Our
ultimate goal here is to estimate the integrals of the solutions of
(\ref{A12a}) taken with the functions
\begin{equation}
  \label{Ma}
F^\theta_m (\gamma) = \sum_{x_1\in \gamma} \theta (x_1) \sum_{x_2\in
\gamma\setminus x_1} \theta(x_2) \cdots \sum_{x_m\in \gamma\setminus
\{x_1 , \dots , x_{m-1}\}} \theta (x_m), \quad \theta\in
\varTheta_\psi^{+},
\end{equation}
which can be obtained from the functions defined in (\ref{TH1}) by
setting $\theta_1=\cdots =\theta_m=\theta$ and $\tau=0$. Note that
$F^\theta_m$ is unbounded, but integrable for each $\mu \in
\mathcal{P}_{\rm exp}$, as it follows from the formula, see
(\ref{Lenard2}),
\begin{equation}
  \label{Ma1}
  \mu (F_m^\theta) = \int_{\mathds{R}^d} k_\mu^{(m)} (x_1 , \dots ,
  x _m) \theta(x_1) \cdots \theta(x_m) d x_1 \cdots d x_m.
\end{equation}
Then by estimating $\mu_t(F^{\theta}_m)$ we will prove  the
mentioned lemma.

To simplify notations by $\varPhi_\tau^m$ we denote a particular
case of the function defined in (\ref{TH1}), corresponding to the
choice $\theta_1 = \cdots = \theta_m = \theta\in
\varTheta_{\psi}^{+}$ with $\bar{c}_\theta=1$, see (\ref{C801}).
Namely, for $\theta \in \varTheta_\psi^{+}$, we set, cf. also
(\ref{Ma}),
\begin{equation}
  \label{Ma2}
  \varPhi_\tau^m (\gamma) =  \sum_{x_1\in \gamma} \theta (x_1) \sum_{x_2\in
\gamma\setminus x_1} \theta(x_2) \cdots \sum_{x_m\in \gamma\setminus
\{x_1 , \dots , x_{m-1}\}} \theta (x_m) \widetilde{F}^0_\tau
(\gamma\setminus \{x_1, \dots , x_m\}),
\end{equation}
and consider such functions with $\tau\in (0,1]$.
Note that the function defined in (\ref{TH8}) is a particular
case of $ \varPhi_\tau^m (\gamma)$ corresponding to the choice
$\theta = \psi$. Then by (\ref{TH7}) we obtain
\begin{equation}
  \label{Ma3}
 \left|L \varPhi_\tau^m (\gamma)\right| \leq  m
 \varPhi_\tau^{m,\theta^1} (\gamma) + \tau c_a
 \widehat{F}^{m+1}_\tau (\gamma) =:\varPhi_{\tau,1}^{m} (\gamma).
\end{equation}
Here $\theta^1 = a\ast \theta + \theta$, see (\ref{TH4}), and
\begin{equation}
  \label{Ma3a} \varPhi_\tau^{m,\theta'} = \widehat{F}_\tau^{\theta', \theta_2,
\dots, \theta_m}, \qquad \theta_2 = \cdots = \theta_m = \theta.
\end{equation}
Note that $\varPhi_{\tau,1}^{m}$ is a linear combination of the
elements of $\widehat{\mathcal{F}}$, see (\ref{C800}). Hence, any
solution of (\ref{A12a}) should satisfy it also with this function.
Let us then estimate $L \varPhi_{\tau,1}^m$. Proceeding as in
(\ref{TH2}) we obtain
\begin{gather*}
\nabla^{y,x}\varPhi_{\tau,1}^{m} (\gamma) = m[\theta^1(y) -
\theta^1(x)] \varPhi^{m-1}_\tau (\gamma\setminus x) +
m(m-1)[\theta(y) - \theta(x)] \varPhi^{m-1,\theta^1}_\tau
(\gamma\setminus x)\\[.2cm] + m\left[e^{-\tau \psi(y)} - e^{-\tau
\psi(x)}\right]\varPhi^{m,\theta^1}_\tau (\gamma\setminus x) + (m+1)
\tau c_a [\psi(y)-\psi(x)]\widehat{F}^{m}_\tau (\gamma\setminus x)
\\[.2cm] + \tau c_a\left[e^{-\tau \psi(y)} - e^{-\tau
\psi(x)}\right]\widehat{F}^{m+1}_\tau (\gamma\setminus x).
\end{gather*}
Now to estimate $L \varPhi_{\tau,1}^m$ we perform the same
calculations as in passing to the second line in the right-hand side
of (\ref{TH7}), see (\ref{TH5}), (\ref{TH6}). In addition, the third
term in the right-hand side of the latter is estimated by employing
$\theta(x) \leq \psi(x)$, cf. (\ref{C801}), and $\theta^1(x) \leq
c_a \psi(x)$, cf. (\ref{TH4}). This yields
\begin{gather*}
m \left|e^{-\tau \psi(y)} - e^{-\tau
\psi(x)}\right|\varPhi^{m,\theta^1}_\tau (\gamma\setminus x) \leq m
\tau c_a |\psi(y)-\psi(x)|\widehat{F}^{m}_\tau (\gamma\setminus x),
\end{gather*}
see also (\ref{TH8}). Thereafter, we obtain
\begin{eqnarray}
  \label{Ma4}
 \left|L \varPhi_{\tau,1}^m (\gamma)\right|& \leq & m
 \varPhi_\tau^{m,\theta^2} (\gamma) + m(m-1)
 \varPhi_\tau^{m,\theta^1, \theta^1} (\gamma)\\[.2cm]\nonumber & + &(2 m +1) \tau c_a^2 \widehat{F}^{m+1}_\tau
 (\gamma) +  \tau^2 c_a^2 \widehat{F}^{m+2}_\tau
 (\gamma)\\[.2cm] \nonumber & =: & \varPhi_{\tau,2}^m (\gamma).
\end{eqnarray}
Here and below we denote $\theta^0=\theta$ and
\begin{equation}
  \label{Ma5}
\theta^k = a\ast \theta^{k-1} + \theta^{k-1},\qquad k=2,3, \dots,
\end{equation}
$\varPhi_\tau^{m,\theta^2}$ is obtained according to (\ref{Ma3a}),
and
\begin{equation*}
  \varPhi_\tau^{m,\theta',\theta''} = \widehat{F}_\tau^{\theta', \theta'', \theta_3,
\dots, \theta_m}, \qquad \theta_3 = \cdots = \theta_m = \theta.
\end{equation*}
Note that by (\ref{TH4}) we have $\theta^k (x) \leq c_a^k \psi(x)$
(recall that $\bar{c}_\theta =1$).

To proceed further we introduce the following notations. For $m\in
\mathds{N}$ and $n\in \mathds{N}_0$, by $\mathcal{C}_{m,n}$ we
denote the set of all sequences $c=\{c_k\}_{k\in \mathds{N}_0}
\subset \mathds{N}_0$ such that the following holds:
\begin{equation}
  \label{CV}
  c_0 + c_1 + \cdots + c_k +\cdots = m, \qquad c_1 + 2 c_2 +\cdots +
  k c_k +\cdots = n.
\end{equation}
Since all $c_k$ are nonnegative integers, for $c\in \mathcal{C}_{m,n}$ by (\ref{CV}) we have that
$c_{n+j}=0$ for all $j\geq 1$, $c_n \leq 1$, and $c_j=0$ for all
$j=1,2, \dots , n-1$ whenever $c_n=1$. For example,
$\mathcal{C}_{m,0}$ and $\mathcal{C}_{m,1}$ are singletons,
consisting of $c= (m, 0,0 \dots)$ and $c= (m-1, 1,0 \dots)$,
respectively. $\mathcal{C}_{m,2}$ consists of $c=(m-1, 0, 1, 0,
\dots,  )$ and $c=(m-2, 2, 0, \dots)$. For $c\in \mathcal{C}_{m,n}$,
$\tau \in (0,1]$ and $\gamma\in \Gamma_*$, we set
\begin{equation}
  \label{CV1}
  V_\tau (c;\gamma) = \widehat{F}_\tau^{\theta^{q_1}, \dots ,
  \theta^{q_m}}(\gamma),
\end{equation}
where $c_0$ members of the family $\{\theta^{q_1}, \dots ,
  \theta^{q_m}\}$ are equal to $\theta^0=\theta$, $c_1$ of them are
equal to $\theta^1$, etc. In particular, $\varPhi_\tau^{m,\theta^2}$
and $\varPhi_\tau^{m,\theta^1,\theta^1}$ can be written as in
(\ref{CV1}) with $c=(m-1, 0, 1, 0, \dots,  )$ and $c=(m-2, 2, 0,
\dots)$, respectively. In Appendix below, we prove the following
estimates
\begin{equation}
 \label{Bog}
\forall \gamma \in \Gamma_* \qquad  \left|L \varPhi^m_{\tau,n-1} (\gamma)\right| \leq \varPhi^m_{\tau,n} (\gamma) , \quad n\in \mathds{N},
\end{equation}
holding with $\varPhi^m_{\tau,n}$ given by the following formula
\begin{eqnarray}
  \label{Ma6}
\varPhi^{m}_{\tau,n}(\gamma)& = & \sum_{c\in \mathcal{C}_{m,n}}
C_{m,n}
(c) V_\tau (c;\gamma) + c_a^n\sum_{k=1}^m \tau^k w_k (m,n) \widehat{F}^{m+k}_\tau ( \gamma),\\[.2cm] \nonumber
C_{m,n}(c) & = & \frac{m! n!}{c_0! c_1! \cdots c_k! \cdots
(0!)^{c_0} (1!)^{c_1} \cdots (k!)^{c_k} \cdots }.
\end{eqnarray}
We also prove that
\begin{equation}
\label{Ma7a} \sum_{c\in \mathcal{C}_{m,n}} C_{m,n} (c) = m^n.
\end{equation}
The coefficients in the second summand of the first line in
(\ref{Ma6}) are subject to the following recurrence relations
\begin{eqnarray}
  \label{Ma8}
w_1 (m,n+1) & = & m^n + (m+1) w_1 (m,n), \\[.2cm] \nonumber
w_k (m,n+1) & = & w_{k-1} (m,n) + (m+k) w_k (m,n), \qquad k=2, \dots
n, \\[.2cm] \nonumber
w_{n+1}(m, n+1) & = & w_{n}(m,n) = 1,
\end{eqnarray}
that can be deduced in the same way as we obtained the estimate in
(\ref{Ma4}). In the first line of (\ref{Ma8}) we take into account
also (\ref{Ma7a}). The initial condition $w_1 (m,1) = 1$ can easily
be derived from (\ref{Ma3}). Then iterating back to $n=1$ in the
first line of (\ref{Ma8}) yields $w_1 (m,n)= (m+1)^n - m^n$. It
turns out that the complete solution of (\ref{Ma8}) has the
following simple form
\begin{equation}
  \label{Ma9}
 w_k (m,n) = \varDelta^k m^n = \frac{1}{k!} \sum_{s=0}^k {k \choose
 s } (-1)^{k-s} (m+s)^n,
\end{equation}
where $\varDelta$ is the forward difference operator -- a standard
combinatorial object. Note that the right-hand side of (\ref{Ma9})
makes sense for all $k\in \mathds{N}_0$: $w_0(m,n) = m^n$,
$w_k(m,n)=0$ for all $k>n$.

In view of (\ref{CV1}) and Proposition \ref{TH1pn}, all the terms of
the linear combination in the first line in (\ref{Ma6}) are
continuous bounded functions of $\gamma$. Hence, the same is
$\varPhi^m_{\tau,n}$. However, its bound may depend on $n$, and our
aim now is to control this dependence. For $\rho>0$, set
\begin{equation}
  \label{Ma10}
  \varUpsilon^m_{\tau, \rho} (\gamma) = \sum_{n=0}^{+\infty}
  \frac{\rho^n}{n!} \varPhi^m_{\tau, n} (\gamma), \qquad \tau \in (0,1].
\end{equation}
To get an upper bound for $\varUpsilon^m_{\tau, \rho}$ we estimate
each $\theta^q$ in the first line of (\ref{Ma6}) as $\theta^q\leq
c_a^q\psi$, $q\geq 0$, see (\ref{Ma5}), which by (\ref{CV1}) and
(\ref{CV}) yields
\begin{equation*}
V_\tau (c;\gamma) \leq
  c_a^{q_1 + \cdots + q_m} \widehat{F}_\tau^m (\gamma) = c_a^{n} \widehat{F}_\tau^m
  (\gamma),
\end{equation*}
where we have taken into account that $q_1 + \cdots + q_m = c_1 + 2
c_2 +\cdots + k c_k +\cdots = n$. In view of (\ref{Ma7a}), this
leads to the following
\begin{eqnarray}
  \label{Ma12}
\varUpsilon^m_{\tau, \rho} (\gamma) & \leq & \sum_{n=0}^{+\infty}
\frac{(c_a \rho)^n}{n!} \sum_{k=0}^n \tau^k w_k (m,n)
\widehat{F}_\tau^{m+k} (\gamma) \\[.2cm] \nonumber & = &
\sum_{k=0}^{+\infty} \tau^k \left(\sum_{n=k}^{+\infty}\frac{(c_a
\rho)^n}{n!} w_k(m,n)\right) \widehat{F}_\tau^{m+k} (\gamma) \\[.2cm] \nonumber & = &
\sum_{k=0}^{+\infty} \tau^k \left(\sum_{n=0}^{+\infty}\frac{(c_a
\rho)^n}{n!} w_k(m,n)\right) \widehat{F}_\tau^{m+k} (\gamma) \\[.2cm] \nonumber & =
& \sum_{k=0}^{+\infty} \frac{\tau^k}{k!} \sum_{s=0}^k {k \choose s}
(-1)^{k-s} \left( \sum_{n=0}^{+\infty}\frac{(c_a \rho (m+s))^n}{n!}
\right)  \widehat{F}_\tau^{m+k} (\gamma)  \\[.2cm] \nonumber & =
& e^{c_a\rho m}\sum_{k=0}^{+\infty} \frac{\tau^k}{k!} \left(e^{c_a\rho}-1
\right)^k  \widehat{F}_\tau^{m+k} (\gamma).
\end{eqnarray}
Here we used the fact that $\varDelta^k m^{n} = 0$ for $k> n$, see
(\ref{Ma9}). To proceed further we use Proposition \ref{TH1pn} and
(\ref{TH8}) and then obtain
\begin{eqnarray*}
 \widehat{F}_\tau^{m+k} (\gamma) \leq e^{\tau(m+k)} \Psi_0^{m+k} (\gamma) \exp\left( - \tau \Psi_0
 (\gamma) \right) ,
\end{eqnarray*}
which in turn yields in the last line of (\ref{Ma12}) the following
estimate
\begin{eqnarray*}
\varUpsilon^m_{\tau, \rho} (\gamma) & \leq & e^{m(c_a \rho+ \tau)}
\Psi_0^m(\gamma) \exp\bigg{(} - \tau \Psi_0(\gamma) \left[ 1 - e^\tau (e^{c_a\rho} - 1) \right] \bigg{)} \\[.2cm] \nonumber & \leq &
e^{m(c_a \rho+ \tau)}
\Psi_0^m(\gamma) \exp\left( - \tau \varepsilon \Psi_0 (\gamma)
\right)  \\[.2cm] \nonumber & \leq & \left( \frac{m}{e \tau\varepsilon}\right)^m
(e + 1-\varepsilon)^m =: \delta_m(\tau),
\end{eqnarray*}
holding for some fixed $\varepsilon \in(0, 1)$ and all
\begin{equation}
  \label{Ma14}
  \rho \leq \rho_\varepsilon := \frac{1}{c_a} \left[ \log ( 1+e - \varepsilon) - 1\right].
\end{equation}
By (\ref{Ma10}) this yields the estimate in question in the
following form
\begin{equation}
  \label{Ma15}
  \varPhi^m_{\tau,n} (\gamma) \leq \frac{n!}{\rho_\varepsilon^n}
  \delta_m(\tau), \qquad \tau \in (0,1].
\end{equation}

\subsection{Proof of Lemma \ref{W2lm}} According to Definition
\ref{THdf} and (\ref{Ma2}), we have that $\varPhi_\tau^m$ lies in
the linear span of $\widehat{\mathcal{F}}$ for each $\tau>0$ and
$m\in \mathds{N}$. If $\{\mu_t\}_{t\geq 0}\subset
\mathcal{P}(\Gamma_*)$ solves (\ref{A12a}), then
\begin{equation}
  \label{Ma16}
  \mu_t (\varPhi_\tau^m) = \mu_0 (\varPhi_\tau^m) + \int_0^t \mu_u (L
  \varPhi_\tau^m) du  \leq \mu_0 (\varPhi_\tau^m) + \int_0^t \mu_u (
  \varPhi_{\tau,1}^m) du,
\end{equation}
where we have used (\ref{Ma3}). Since $\varPhi_{\tau,1}^m$ is a
linear combination of the elements of $\widehat{\mathcal{F}}$, we can
repeat (\ref{Ma16}) with this function and obtain
\begin{equation*}
  \mu_t (\varPhi_{\tau,1}^m)  \leq \mu_0 (\varPhi_{\tau,1}^m) + \int_0^t \mu_u (
  \varPhi_{\tau,2}^m) du,
\end{equation*}
which then can be used in (\ref{Ma16}). In view of (\ref{Bog}), we can
repeat this procedure due times and thereby get the following
estimate
\begin{eqnarray}
  \label{Ma17}
  \mu_t (\varPhi_\tau^m) & \leq & \sum_{k=0}^{n-1} \frac{t^k}{k!} \mu_0
  (\varPhi_{\tau,k}^m) + \int_0^{t}\int_0^{t_1} \cdots
  \int_0^{t_{n-1}} \mu_{t_n} (\varPhi_{\tau,n}^m) d t_n d t_{n-1}
  \cdots d t_1 \\[.2cm] \nonumber &\leq &  \sum_{k=0}^{n-1} \frac{t^k}{k!} \mu_0
  (\varPhi_{\tau,k}^m) + \left(
  \frac{t}{\rho_\varepsilon}\right)^n\delta_m(\tau),
\end{eqnarray}
where we have used (\ref{Ma15}) and the fact that $\mu_t$ is a
probability measure. For $t< \rho_\varepsilon$, the last summand in
the right-hand side of (\ref{Ma17}) vanishes as $n\to +\infty$.
Hence,
\begin{equation}
  \label{Ma18}
\mu_t (\varPhi_\tau^m) \leq \sum_{n=0}^{+\infty} \frac{t^n}{n!}
\mu_0
  (\varPhi_{\tau,n}^m), \qquad t< \rho_\varepsilon, \ \ \tau \in (0,1].
\end{equation}
By (\ref{TH1}) and (\ref{CV1}) it follows that the element of
$\widehat{\mathcal{F}}$ in the first summand in the first line in
(\ref{Ma6}) satisfies
\begin{eqnarray*}
& & V_\tau (c;\gamma) \leq V_0 (c;\gamma):= \sum_{x_1\in \gamma}
\theta^{q_1}(x_1) \sum_{x_2\in \gamma\setminus x_1}
\theta^{q_2}(x_2) \cdots \sum_{x_m\in \gamma\setminus \{x_1 , \dots
, x_{m-1}\}} \theta^{q_m}(x_m), \qquad \theta \in
\varTheta^{+}_\psi.
\end{eqnarray*}
$V_0 (c;\cdot)$ is an unbounded function, which, however, is
$\mu_0$-integrable. Let $\varkappa_0$ be the type of $\mu_0$. As in
Remark \ref{Iirk}, we then have
\begin{eqnarray}
\label{Ma19}  \mu_0(V_\tau (c;\cdot)) \leq \pi_{\varkappa_0} (V_0
(c;\cdot) )= \varkappa_0^m \langle \theta^{q_1} \rangle \cdots
\langle \theta^{q_m} \rangle = 2^n  \left( \varkappa_0\langle \theta
\rangle\right)^{m},
\end{eqnarray}
where
\[
\langle \theta^{q_j} \rangle := \int_{\mathds{R}^d}\theta^{q_j}  (x)
d x = 2^{q_j} \int_{\mathds{R}^d}\theta  (x) d x =2^{q_j} \langle \theta
\rangle,
\]
see (\ref{Ma5}) and (\ref{C7}). Here we have taken into account that
$q_1 +\cdots + q_{m} = n$.  By (\ref{TH8}) we have
\begin{eqnarray*}
\widehat{F}_\tau^{m+k} (\gamma)\leq  \sum_{x_1\in \gamma} \psi(x_1)
\sum_{x_2\in \gamma\setminus x_1} \psi(x_2) \cdots \sum_{x_{m+k}\in
\gamma\setminus \{x_1 , \dots , x_{m+k-1}\}} \psi(x_{m+k}).
\end{eqnarray*}
Then similarly as in (\ref{Ma19}) we obtain
\begin{equation}
  \label{Ma20}
\mu_0(\widehat{F}_\tau^{m+k}) \leq \left(\varkappa_0 \langle \psi
\rangle \right)^{m+k}.
\end{equation}
We use (\ref{Ma19}) and (\ref{Ma20}) in (\ref{Ma6}) and then in
(\ref{Ma18}) and arrive at the following estimate
\begin{eqnarray*}
\mu_t (\varPhi_\tau^m) & \leq & \left(e^{2t} \varkappa_0 \langle
\theta\rangle \right)^m + \left(\varkappa_0\langle \psi \rangle
\right)^m e^{c_a t m} \sum_{k=1}^\infty \frac{\tau^k}{k!} \left(
\varkappa_0 \langle
\psi \rangle \right)^k \left( e^{c_at} - 1\right)^k \\[.2cm]
\nonumber & \leq & \left( e^{2t} \varkappa_0 \langle \theta\rangle
\right)^m + \tau  \left(\varkappa_0\langle \psi \rangle
\right)^{m+1} e^{c_a t m} \left( e^{c_at} - 1\right) \exp\left(
\varkappa_0\langle \psi \rangle\left( e^{c_at} - 1\right) \right),
\end{eqnarray*}
where we have applied the same approach as in obtaining (\ref{Ma12})
and the fact that $\tau\leq 1$. Since, for each $\gamma\in \Gamma_*$
and an arbitrary sequence $\tau_n\to 0$, $\{\widetilde{F}^0_{\tau_n}
(\gamma)\}_{n\in \mathds{N}}$ is a nondegreasing sequence, by
(\ref{Ma2}) and Beppo Levi's monotone convergence theorem we then
get from the latter that, cf. (\ref{Ma1}),
\begin{eqnarray}
 \label{Ma22}
& &  \lim_{\tau \to 0} \mu_t (\varPhi_\tau^m)  = \mu_t(F^\theta_m) =
\langle k_{\mu_t}^{(m)}, \theta^{\otimes m} \rangle
 \\[.2cm]\nonumber & &   := \int_{\mathds{R}^d}
 k_{\mu_t}^{(m)} (x_1 , \dots  , x_m) \theta(x_1 ) \cdots \theta(x_m) d x_1 \cdots d x_m
  \\[.2cm]\nonumber & & \leq \left( e^{2t}\varkappa_0 \langle
  \theta\rangle
\right)^m,
\end{eqnarray}
holding for all $m\in \mathds{N}$ and $t < \rho_\varepsilon$, see
(\ref{Ma14}).
 Since $\theta\in \varTheta_\psi^{+}$, we have $\langle
\theta \rangle = \| \theta\|_{L^1(\mathds{R}^d)}$, and the latter
estimate can be rewritten in the form, cf. (\ref{Len}),
\begin{equation}
  \label{Ma22a}
\forall m \in \mathds{N} \qquad \langle k_{\mu_t}^{(m)} ,
\theta^{\otimes m} \rangle \leq (2 e^t \varkappa_0)^m
\|\theta\|^m_{L^1(\mathds{R}^d)} , \qquad \theta \in
\varTheta_\psi^{+}.
\end{equation}
The set of functions $\varTheta_{\psi}^+$ defined in (\ref{T1}) is
closed with respect to the pointwise multiplication and separates
points of $\mathds{R}^{d }$. Such functions vanish at infinity and
are everywhere positive. Then by the aforementioned version of the
Stone-Weierstrass theorem \cite{dB} the linear span of this set is
dense (in the supremum norm) in the algebra $C_0(\mathds{R}^{d })$
of all continuous functions that vanish at infinity. At the same
time, $C_0(\mathds{R}^{d })\cap L^1(\mathds{R}^{d })$ is dense in
$L^1(\mathds{R}^{d })$. Therefore, by (\ref{Ma22a}) the maps $\theta
\mapsto \langle k_{\mu_t}^{(m)} , \theta^{\otimes m} \rangle$, $m\in
\mathds{N}$ can be extended to homogeneous continuous monomials on
$L^1(\mathds{R}^d)$.  This yields the proof of the considered
statement for $t< \rho_\varepsilon$, see Remark  \ref{Len1rk}. Since
$\rho_\varepsilon$ is independent of $\varkappa_0$, the continuation
to all $t>0$ can be made by the repetition of the same arguments.

\subsection{Proof of the uniqueness}

By employing Lemma \ref{W2lm} and  Corollary \ref{Walm}, see also
Remark \ref{A2rk}, we prove the following statement.
\begin{lemma}
  \label{A3pn}
Assume that two solutions $\{P^{(i)}_{s,\mu}: s \geq 0, \ \mu \in
\mathcal{P}_{\rm exp}\}$, $i=1,2$, see Definition \ref{A1df},
satisfy $P_{s,\mu}^{1}\circ \varpi_{t}^{-1}=P_{s,\mu}^{2}\circ
\varpi_{t}^{-1}$ for all $t\geq s$, $s\geq 0$ and $\mu\in
\mathcal{P}_{\rm exp}$. Then $P_{s,\mu}^{1}=P_{s,\mu}^{2}$ for all
$s$ and $\mu$.
\end{lemma}
\begin{proof}
By Kolmogorov's extension theorem it is enough to prove that all
finite-dimensional marginals of both path measures coincide. In view
of claim (i) of Proposition \ref{T1pn}, to this end we have to show
that the following holds
\begin{equation}
  \label{AUz}
P_{s,\mu}^{1} \left(\mathsf{F}_{t_1}\cdots \mathsf{F}_{t_n} \right)
= P_{s,\mu}^{2} \left(\mathsf{F}_{t_1}\cdots \mathsf{F}_{t_n}
\right),
\end{equation}
where $\mathsf{F}_{t_i} (\gamma) = \widetilde{F}_{\tau_i}^{\theta_i}
(\varpi_{t_i} (\gamma))$, $i=1, \dots , n$, see (\ref{C800}), ought
to be taken with all possible $\theta_i\in \varTheta_\psi$, $\tau_i
>c_{\theta_i}$ and $t_i$ satisfying $s\leq t_1 \leq \cdots \leq
t_n$. Assume that (\ref{AUz}) holds with a given $n$ and prove its
validity for $n+1$. Since $\mathsf{F}_{t_i} (\gamma)>0$, see
(\ref{T4}), we may set
\[
C^{-1} = P_{s,\mu}^{1} \left(\mathsf{F}_{t_1}\cdots \mathsf{F}_{t_n}
\right),
\]
and then define two path measures on $(\mathfrak{D}_{[t_n,
+\infty)}, \mathfrak{F}_{t_n, +\infty})$
\[
Q^i (\mathbb{B}) = C P_{s,\mu}^{i} \left(\mathsf{F}_{t_1}\cdots
\mathsf{F}_{t_n} \mathds{1}_{\mathbb{B}}\right), \qquad i=1,2.
\]
Since both $P^i$ satisfy (\ref{A12}), we have also
\[
\int_{\mathfrak{D}_{[t_n, +\infty)}}{\sf H}(\gamma) Q^i (d \gamma) =
0, \qquad i=1,2.
\]
Hence, both maps $[t_n , +\infty) \ni t \mapsto Q^i \circ
\varpi^{-1}_{t}=: \mu^i_t \in \mathcal{P}(\Gamma_*)$, $i=1,2$ solve
\begin{equation*}
\mu^i_{u_2} (F) = \mu^i_{u_1} (F) + \int_{u_1}^{u_2} \mu^i_{v} (L F)
 dv ,\qquad F\in \mathcal{D}(L),
\end{equation*}
for all $u_2 > u_1 \geq t_n$, see Remark \ref{A2rk}. By the
inductive assumption and claim (iv) of Proposition \ref{T3pn} it
follows that $\mu_{t_n}^1 = \mu_{t_n}^2 =:\mu \in \mathcal{P}_{\rm
exp}$. By Lemma \ref{W2lm} we then conclude that $\mu_{t}^i \in
\mathcal{P}_{\rm exp}$, $i=1,2$ for all $t>t_n$. That is, both $Q^i$
satisfy all the three conditions of Definition \ref{A1df} and thus
belong to solutions of the restricted initial value martingale
problem. Hence, $\mu_t^1 = \mu_t^2$ by the assumption of the lemma.
In particular,
\[
\mu_{t_{n+1}}^1 (\widetilde{F}_{\tau_{n+1}}^{\theta_{n+1}}) =
\mu_{t_{n+1}}^2 (\widetilde{F}_{\tau_{n+1}}^{\theta_{n+1}}),
\]
which completes the proof.
\end{proof}
\begin{theorem}
  \label{Au1tm}
Let $\{P^{(i)}_{s,\mu}: s \geq 0, \ \mu \in \mathcal{P}_{\rm
exp}\}$, $i=1,2$ be two solution of the restricted initial value
martingale problem in the sense of Definition \ref{A1df}. Then
$P^{(1)}_{s,\mu} = P^{(2)}_{s,\mu}$ for all $s\geq 0$ and $\mu\in
\mathcal{P}_{\rm exp}$.
\end{theorem}
\begin{proof}
By Remark \ref{A2rk} both $P^{(i)}_{s,\mu}\circ \varpi_t$, $t\geq s$
solve (\ref{A12a}), which by Theorem \ref{2tm} yields
$P^{(1)}_{s,\mu}\circ \varpi_t^{-1} = P^{(2)}_{s,\mu}\circ
\varpi_t^{-1}$, holding for all $t\geq s$ and $\mu \in
\mathcal{P}_{\rm exp}$. Then the proof follows by Lemma \ref{A3pn}.
\end{proof}

\section{The Existence: Approximating Models}

The aim of this and the subsequent sections is to prove the
following statement which is the second corner stone in the proof of
Theorem \ref{1tm}.
\begin{theorem}
 \label{3tm}
There exists a family of probability measures which solves the
restricted initial value martingale problem for our model in the
sense of Definition \ref{A1df}.
\end{theorem}
The basic idea is to approximate the model by auxiliary models
described by $L^\alpha$, $\alpha \in [0,1]$ with $L^0$ coinciding
with $L$ defined in (\ref{I5}). For $\alpha\in (0,1]$, the solution
$\{P^\alpha_{s,\mu}: s\geq 0, \mu \in \mathcal{P}_{\rm exp}\}$ of
the corresponding restricted initial value martingale problem for
$L^\alpha$ will be constructed in a direct way. Then the proof of
Theorem \ref{3tm} will be done by showing the weak convergence
 $P^\alpha_{s,\mu}\Rightarrow P_{s,\mu}$ as $\alpha \to 0$, and then
by proving that $\{P_{s,\mu}: s\geq 0, \mu \in \mathcal{P}_{\rm
exp}\}$ is a solution in question. In the current section, we
introduce the auxiliary models and study their relations with the
basic model. The construction of the path measures
$P^\alpha_{s,\mu}$ will be preformed in the subsequent section.

\subsection{The approximating models} Recall that  $\psi$ was introduced in (\ref{C3}), see also (\ref{Psi}).
Along with these functions, we shall use $\Psi_1(\gamma) = 1
+\Psi(\gamma)$ and
\begin{equation}
  \label{Ma30}
  \psi_\alpha (x) = \frac{1}{1+ \alpha |x|^{d+1}}, \qquad \alpha \in
  [0,1].
\end{equation}
Set
\begin{equation}
  \label{N2a}
  a_\alpha (x,y) = a(x-y) \psi_\alpha (x), \quad  \
  \ x,y \in \mathds{R}^d.
\end{equation}
Note that $a_0(x,y) = a (x-y)$ and $a_\alpha(x,y)\neq a_ \alpha
(y,x)$ for $\alpha\in (0,1]$. Now let $L^\alpha$ be defined as in
(\ref{I5}) with $a$ replaced by $a_\alpha$. That is,
\begin{eqnarray}
 \label{Lalpha}
( L^\alpha F)(\gamma)  =  \sum_{x\in \gamma}\int_{\mathds{R}^d} \psi_\alpha (x) a (x-y) \exp\left( -\sum_{z\in \gamma\setminus x} \phi(z-y)\right) \left[ F(\gamma\setminus x\cup y) - F(\gamma)\right] d y. \qquad
\end{eqnarray}
Then keeping in mind  (\ref{A10A}) and (\ref{A13a}) we define
$L^{\Delta,\alpha}$ by the following expression
\begin{equation*}
\mu(L^\alpha F^\theta) = \langle \! \langle L^{\Delta,\alpha} k_\mu
, e(\theta; \cdot) \rangle\! \rangle, \qquad \alpha \in [0,1].
\end{equation*}
One observes that $L^{\Delta,0}$ coincides with the operator
introduced in (\ref{K3}). For $\alpha \in (0,1]$,
$L^{\Delta,\alpha}$ is then obtained by replacing in (\ref{K3})
$a(x-y)$ by $a_\alpha(x,y)\leq a(x-y)$. Hence, $L^{\Delta,\alpha}$
clearly satisfies (\ref{A15}) and similar estimates. Then by
repeating the construction realized in subsection \ref{SS4.2} we
obtain the family of bounded operators
$\{Q^\alpha_{\vartheta'\vartheta} (t): t\in
[0,T(\vartheta',\vartheta))\}$ (resp.
$\{H^\alpha_{\vartheta\vartheta'} (t): t\in
[0,T(\vartheta',\vartheta))\}$), $\vartheta'>\vartheta$ acting from
$\mathcal{K}_{\vartheta}$ to $\mathcal{K}_{\vartheta'}$ (resp. from
$\mathcal{G}_{\vartheta'}$ to $\mathcal{G}_{\vartheta}$). By
employing these families we then set
\begin{equation}
  \label{Ma31}
  k^\alpha_t = Q^\alpha_{\vartheta'\vartheta}(t)k_0, \qquad G^\alpha_t
  = H^\alpha_{\vartheta \vartheta'} (t) G_0,
\end{equation}
with $k_0 \in \mathcal{K}_{\vartheta}$ and $G_0 \in
\mathcal{G}_{\vartheta'}$. Note that, for $\alpha=0$, these vectors
coincide with those introduced in (\ref{A26}) and (\ref{A28}),
respectively, and thus they satisfy (\ref{A29}) for all $\alpha \in
[0,1]$. Moreover, as in Proposition \ref{3.3pn}, for each
$\vartheta_0\in \mathds{R}$ and $\mu\in \mathcal{P}_{\rm
exp}^{\vartheta_0}$, by (\ref{Ma31}) with $k_0 = k_{\mu}$ we obtain
a family, $\{\mu_t^\alpha: t\geq 0, \mu_0=\mu\}\subset
\mathcal{P}_{\rm exp}$, $\mu_t^\alpha \in \mathcal{P}_{\rm
exp}^{\vartheta_t}$ such that
\begin{equation}
  \label{Ma32}
  \mu_t^\alpha (F^\theta) = \langle \! \langle k_t^\alpha, e(\theta,
  \cdot)\rangle \! \rangle , \qquad \theta \in L^1(\mathds{R}^d).
\end{equation}
Next, by repeating the construction used in the proof of Theorem
\ref{2tm} one obtains that the map $t\mapsto \mu_t^\alpha$ is a
unique solution of the equation
\begin{equation*}
 \mu^\alpha_{t_2} (F) = \mu^\alpha_{t_1} (F) + \int_{t_1}^{t_2} \mu^\alpha_{u} (L^\alpha
 F) du, \qquad t_2 > t_1 \geq 0,
\end{equation*}
holding for all $F:\Gamma_* \to \mathds{R}$ which can be written as
$F=KG$ with $G\in \cap_{\vartheta\in \mathds{R}}
\mathcal{G}_\vartheta$, see Corollary \ref{Walm}. Here and below we
set
\[
\mathcal{D}(L^\alpha) = \mathcal{D}(L), \qquad \alpha \in (0,1],
\]
with $\mathcal{D}(L)$ as in Definition \ref{THdf}.

\subsection{The weak convergence}
\label{SS6.2} Our aim now is to prove that the families
$\{\mu_t^\alpha: t\geq 0, \mu_0=\mu\}\subset \mathcal{P}_{\rm exp}$,
$\alpha \in [0,1]$ constructed above have the following property.
\begin{lemma}
  \label{U3lm}
For each $t>0$, it follows that $\mu^\alpha_t \Rightarrow \mu_t$ as
$\alpha \to 0$, where we mean the weak convergence of measures on
the Polish space $\Gamma_*$.
\end{lemma}
We begin by proving the convergence of the corresponding correlation
functions.
\begin{lemma}
  \label{U2lm}
For each $t>0$, one finds $\tilde{\vartheta}_t > \vartheta_t$ such
that the following holds
\begin{equation}
  \label{U4a}
\forall G\in \mathcal{G}_{\tilde{\vartheta}_t} \qquad \langle \!
\langle k^\alpha_t , G \rangle \! \rangle \to \langle \! \langle k_t
, G \rangle \! \rangle,  \qquad {\rm as} \ \ \alpha
  \to 0.
\end{equation}
\end{lemma}
\begin{proof}
We recall that $k_t$ satisfies (\ref{C31}) with
$L^\Delta_{\vartheta_T}$ corresponding to $\alpha =0$. Note
that the domains of $L^{\Delta,\alpha}_{\vartheta}$ are the same for
all $\alpha \in [0,1]$.

Assume now that the convergence stated in (\ref{U4a}) holds for a
given $t\geq 0$. Note that $k_0=k^\alpha_0 = k_{\mu_0}$; hence, this
assumption is valid for  at least $t=0$. Let us prove that there
exists $s_0>0$ -- possibly dependent on $t$ -- such that this
convergence holds for all $t+s$, $s\leq s_0$. Keeping in mind that
$Q^\alpha$ and $k_t^\alpha$ satisfy the corresponding analogs of
(\ref{A24}) and (\ref{C31}), respectively, we write
\begin{equation}
 \label{Ua}
  k_{t+s} - k^\alpha_{t+s} = Q_{\bar{\vartheta}_t \vartheta_t} (s)
k_t - Q^\alpha_{\bar{\vartheta}_t \vartheta_t}(s) k^\alpha_t,
\end{equation}
where $\bar{\vartheta}_t = \vartheta_t + \delta(\vartheta_t)$ and
$\vartheta_t= \vartheta_0 + t$. Note that the left-hand side of
(\ref{Ua}) is considered as a vector in
$\mathcal{K}_{\bar{\vartheta}_t}$. Both $Q_{\bar{\vartheta}_t
\vartheta_t}(s)$ and $Q^\alpha_{\bar{\vartheta}_t \vartheta_t}(s)$
are defined only for $s< \tau(\vartheta_t)$, see (\ref{U2}). At the
same time, for each $\vartheta'>\vartheta$, $Q_{\vartheta'
\vartheta}(0)=Q^\alpha_{\vartheta' \vartheta}(0) =
I_{\vartheta'\vartheta}$, where the latter is the embedding
operator, see (\ref{C21b}). Keeping this and (\ref{A24}) in mind we
rewrite (\ref{Ua}) as follows
\begin{eqnarray}
  \label{U5}
k_{t+s} - k^\alpha_{t+s} & = & Q_{\bar{\vartheta}_t \vartheta_t} (s)
(k_t - k_t^\alpha) - \left(\int_0^s \frac{d}{du}
[Q_{\bar{\vartheta}_t \vartheta_1}(s-u)
Q^\alpha_{\vartheta_1\vartheta_t }(u)] du \right)k^\alpha_t \qquad
\\[.2cm] \nonumber & =  & Q_{\bar{\vartheta}_t \vartheta_t} (s)
(k_t - k_t^\alpha) + \int_0^s Q_{\bar{\vartheta}_t
\vartheta_2}(s-u)L^\Delta_{\vartheta_2\vartheta_1}
Q^\alpha_{\vartheta_1 \vartheta_t}(u)k_t^\alpha d u \\[.2cm] \nonumber & -
& \int_0^s Q_{\bar{\vartheta}_t
\vartheta_2}(s-u)L^{\Delta,\alpha}_{\vartheta_2\vartheta_1}
Q^\alpha_{\vartheta_1 \vartheta_t}(u)k_t^\alpha d u
\\[.2cm] \nonumber & =  & Q_{\bar{\vartheta}_t \vartheta_t} (s)
(k_t - k_t^\alpha) + \int_0^s Q_{\bar{\vartheta}_t \vartheta_2}(s-u)
\widetilde{L}^{\Delta,\alpha}_{\vartheta_2\vartheta_1}
k^\alpha_{t+u} du,
\end{eqnarray}
where $\widetilde{L}^{\Delta,\alpha}$ is given in (\ref{K3}) with
$a(x-y)$ replaced by $\tilde{a}_\alpha(x,y) = a(x-y)(1- \psi_\alpha
 (x))$. The  choice of $s$ and  $\vartheta_1$,
$\vartheta_2$ should be made in such a way that the series as in
(\ref{A22}) converge for the corresponding operators. Set
$\vartheta_1 = \vartheta_t + \delta(\vartheta_t)/2$. We use this in
(\ref{U1}) and obtain that
\begin{equation}
  \label{tau1}
T(\bar{\vartheta}_t ,\vartheta_1) = \frac{\tau(\vartheta_t)}{2} <
T(\vartheta_1,\vartheta_t).
\end{equation}
Then for some $\epsilon\in (0,1)$, we set
\begin{equation}
  \label{U6}
 s_0 = \epsilon \tau (\vartheta_t )/2 = \epsilon T(\bar{\vartheta}_t ,\vartheta_1).
\end{equation}
Since the map $\vartheta \mapsto T(\bar{\vartheta}_t , \vartheta)$
is continuous, one can find $\vartheta_2 \in
(\vartheta_1,{\vartheta}_t)$ such that $s_0 < T(\bar{\vartheta}_t
,\vartheta_2)$, cf. (\ref{U6}), which together with (\ref{tau1})
yields that all the three $Q_{\bar{\vartheta}_t \vartheta_1}(s-u)$,
$Q_{\bar{\vartheta}_t \vartheta_2}(s-u)$ and $Q_{\vartheta_1
\vartheta_t}^\alpha(u)$ in (\ref{U5}) are defined for all $s\leq
s_0$ and $u\in [0,s]$. Now we take
$G\in\mathcal{G}_{\bar{\vartheta}_t}$ and set $G_s =
H_{\vartheta_2\bar{\vartheta}_t}(s)G$, $s\leq s_0$. Then $G_s \in
\mathcal{G}_{\vartheta_2} \subset \mathcal{G}_{\vartheta_t}$, which
yields by (\ref{U5}) the following
\begin{gather}
  \label{U7}
\langle\!\langle k_{t+s} - k_{t+s}^\alpha , G\rangle \! \rangle  =
\langle\!\langle k_{t} - k_{t}^\alpha , G_s\rangle \! \rangle +
Y_\alpha (s), \\[.2cm] Y_\alpha (s) := \int_0^s \langle\!\langle \widetilde{L}^{\Delta,\alpha}_{\vartheta_2 \vartheta_1}
k_{t+u}^\alpha, G_{s-u}\rangle \! \rangle d u. \nonumber
\end{gather}
Thus, we have to prove that $Y_\alpha (s) \to 0$ as $\alpha \to 0$.
Since $L^{\Delta,\alpha}$ consists of two terms, see (\ref{K3}), it
is convenient for us to write $Y_\alpha (s)= Y_\alpha^{(1)} (s) +
Y_\alpha^{(2)} (s)$, where
\begin{eqnarray}
  \label{U8}
& & Y_\alpha^{(1)} (s)  =  \int_0^s \int_{\Gamma_0} \bigg{(}
\sum_{y\in \eta} \int_{\mathds{R}^d} \tilde{a}_\alpha (x,y) e(\tau_y
; \eta\setminus y) (W_y k^\alpha_{t+u})(\eta\setminus y\cup x) d x
\bigg{)}\\[.2cm] \nonumber & & \quad \times   G_{s-u} (\eta) \lambda ( d
\eta) d u \\[.2cm] \nonumber & & =  \int_0^s \int_{\Gamma_0} \bigg{(}
\int_{(\mathds{R}^d)^2} \tilde{a}_\alpha (x,y) e(\tau_y ; \eta) (W_y
k^\alpha_{t+u})(\eta\cup x) G_{s-u}(\eta \cup y) d x dy \bigg{)}
\lambda ( d \eta) d u,
\end{eqnarray}
and
\begin{eqnarray}
  \label{U9}
& & Y_\alpha^{(2)} (s)  =  - \int_0^s \int_{\Gamma_0} \bigg{(}
\sum_{x\in \eta} \int_{\mathds{R}^d} \tilde{a}_\alpha (x,y) e(\tau_y
;
\eta\setminus x) (W_y k^\alpha_{t+u})(\eta) d y \bigg{)}\\[.2cm] \nonumber & & \quad \times   G_{s-u} (\eta)  \lambda ( d
\eta)d u \\[.2cm] \nonumber & & =  - \int_0^s \int_{\Gamma_0} \bigg{(}
\int_{(\mathds{R}^d)^2} \tilde{a}_\alpha (x,y) e(\tau_y ; \eta) (W_y
k^\alpha_{t+u})(\eta\cup x) G_{s-u}(\eta \cup x) d x dy \bigg{)}
\lambda ( d \eta)d u.
\end{eqnarray}
To estimate both terms we take into account that $e(\tau_y;\eta)
\leq 1$ and
\[
|(W_y k^\alpha_{t+u})(\eta\cup x)| \leq
 \exp\left( \vartheta_1 +
\vartheta_1 |\eta| + \langle \phi \rangle e^{\vartheta_1}\right),
\]
where the latter estimate follows by the fact that
$k^\alpha_{t+u}(\eta) \leq \exp(\vartheta_{t+u}|\eta|)\leq
\exp(\vartheta_{1}|\eta|)$, see claim (a) of Proposition
\ref{3.3pn}. By these estimates we obtain from (\ref{U8}) and
(\ref{U9}) the following
\begin{equation}
  \label{u10}
\left|Y_\alpha^{(i)} (s)\right| \leq \int_{\mathds{R}^d}
h^{(i)}_\alpha (y) g^{(i)}_s(y) d y, \qquad i=1,2,
\end{equation}
where
\begin{equation}
  \label{U11}
  h^{(1)}_\alpha (y) = \int_{\mathds{R}^d} \tilde{a}_\alpha (x,y) d x =
  \int_{\mathds{R}^d} (1- \bar{\psi}_\alpha  (|x|)) a (x -y) d x,
\end{equation}
$\bar{\psi}_\alpha (r) := (1+\alpha r^{d+1})^{-1}$, cf.
(\ref{Ma30}), and
\begin{gather*}
g^{(1)}_s (y) = c(\vartheta_1) \int_0^s \int_{\Gamma_0} \left|
G_{s-u}(\eta \cup y)\right| e^{\vartheta_1 |\eta|} \lambda ( d \eta)
d u, \\[.2cm] \nonumber c(\vartheta_1):= \exp\left(\vartheta_1 + \langle \phi
\rangle e^{\vartheta_1} \right).
\end{gather*}
Let us show that $g^{(1)}_s$ is integrable for all $s\leq s_0$. To
this end we use the fact that $G_{s-u}\in \mathcal{G}_{\vartheta_2}$
for all $s\leq s_0$ and $u\leq s$. Then its norm can be estimated
\begin{equation*}
 |G_{s-u}|_{\vartheta_2} \leq \frac{T(\bar{\vartheta}_t ,
 \vartheta_2) }{T(\bar{\vartheta}_t ,
 \vartheta_2) - s_0} |G|_{\bar{\vartheta}_t} =: C_G
\end{equation*}
which is finite by our choice of $\vartheta_2$ and $s_0$. Then
\begin{eqnarray}
  \label{U14}
  \int_{\mathds{R}^d} g^{(1)}_s (y) d y & = & c(\vartheta_1) \int_0^s
  \int_{\Gamma_0} \int_{\mathds{R}^d} |G_{s-u}(\eta\cup y)| e^{\vartheta_1
  |\eta|} d y \lambda (d\eta) d u \\[.2cm] \nonumber &= & c(\vartheta_1)e^{-\vartheta_1}\int_0^s
  \int_{\Gamma_0} |G_{s-u}(\eta)||\eta| e^{\vartheta_1
  |\eta|} d y \lambda (d\eta)  d u
   \\[.2cm] \nonumber &= & c(\vartheta_1)e^{-\vartheta_1}\int_0^s
  \int_{\Gamma_0}\left(|\eta| e^{-|\eta| (\vartheta_2 -
  \vartheta_1)} \right) |G_{s-u} (\eta)| e^{\vartheta_2 |\eta|}
  \lambda ( d\eta) d u
   \\[.2cm] \nonumber &\leq &
      \frac{c(\vartheta_1)s}{e^{1+\vartheta_1}(\vartheta_2-\vartheta_1)} C_G.
\end{eqnarray}
Now let us turn to (\ref{U11}). First of all, we note that
$h^{(1)}_\alpha (y) \leq 1$, see (\ref{C7}). The function $r\mapsto
1 - \bar{\psi}_\alpha (r)$ is increasing. Then, for a certain $r>0$,
we have
\begin{eqnarray}
  \label{U15}
h^{(1)}_\alpha (y) & = & \int_{\mathds{R}^d}(1- \bar{\psi}_\alpha
(|x+y|))
a(x) d x \\[.2cm] \nonumber & \leq & \int_{B_r}(1- \bar{\psi}_\alpha (r+|y|)) a(x) d x +
\int_{B_r^c}  a(x) d x \\[.2cm] \nonumber & \leq
& (1-\bar{\psi}_\alpha( r+|y|)) + \frac{m^a_{d+1}}{r^{d+1}},
\end{eqnarray}
where the second term of the last line was obtained by Markov's
inequality and (\ref{C8}) together with the estimate
$1-\bar{\psi}_\alpha (r) \leq 1$. Now we set in (\ref{U15}) $r =
\alpha^{-1/(d+2)}$ and obtain
\begin{equation*}
h^{(1)}_\alpha (y) \leq \frac{\alpha^{1/(d+2)}(1+|y|)^{d+1}} {1+
\alpha^{1/(d+2)}(1+|y|)^{d+1}} + m^a_{d+1} \alpha^{\frac{d+1}{d+2}}.
\end{equation*}
Hence, for each $y$, $h^{(1)}_\alpha (y)\to 0$ as $\alpha \to 0$.
Then by Lebesgue's dominated convergence theorem and (\ref{U14}) and
(\ref{u10}) we conclude that $Y^{(1)}_\alpha (s) \to 0$ as $\alpha
\to 0$, holding for all $s\leq s_0$.

Now we turn to (\ref{U9}) by which we get
\begin{equation*}
 h^{(2)}(y) = \int_{\mathds{R}^d} \tilde{\psi}_\alpha(y) a(x-y) d x =
 \tilde{\psi}_\alpha(y)=: 1 - \psi_\alpha (y),
\end{equation*}
and $g_s^{(2)}(y) = g_s^{(1)}(y)$. Hence, also $Y^{(2)}_\alpha (s)
\to 0$ as $\alpha \to 0$, holding for all $s\leq s_0$, which by
(\ref{U7}) yields the proof of (\ref{U4a}) for $t+s$ with $s\leq
s_0$ whenever it holds for $t$. To complete the proof let us
consider the following sequences, cf. (\ref{U6}),
\begin{gather}
  \label{U18}
t_l = t_{l-1}+s_{0l}, \qquad t_0=0, \ \ l\in \mathds{N}, \\[.2cm]
s_{0l} = \epsilon \tau(\vartheta_{t_{l-1}})/2. \nonumber
\end{gather}
Since $k_0^\alpha = k_0 = k_{\mu}$, the proof made above yields the
stated convergence for $t\leq \sup_l t_l = \lim_l t_l$. Thus, our
aim is to show that $t_l \to +\infty$ as $l\to +\infty$. Assume that
$\sup_l t_l = t_*<\infty$. By the first line in (\ref{U18}) we have
that $t_l = s_{01} +\cdots s_{0l}$ and hence $s_{0l} \to 0$ in this
case. Now we pass in the second line of (\ref{U18}) to the limit
$l\to +\infty$ ($\tau$ is continuous) and get that $t_*$ should
satisfy $\tau(\vartheta_{t_*})=\tau(\vartheta_0 + t_{*})=0$, which
is impossible as $\tau(\vartheta)>0$ for all $\vartheta\in
\mathds{R}$. This completes the proof with $\tilde{\vartheta}_t =
\bar{\vartheta}_t$.
\end{proof}
\vskip.1cm \noindent {\it Proof of Lemma \ref{U3lm}.} By Lemma
\ref{U2lm} and (\ref{Lenard2}) it follows that $\mu^\alpha_t (F) \to
\mu_t(F)$ as $\alpha\to 0$, holding for all $F\in \mathcal{F}$, see
(\ref{WA}). Then the proof follows by the fact that
$\widetilde{\mathcal{F}}\subset \mathcal{F}$, see (\ref{WA1}), and
claim (ii) of Proposition \ref{T1pn}. \hfill $\square$ \vskip.1cm
\noindent Below we us the following fact, that can be considered as
a complement to Lemma \ref{U3lm}.
\begin{lemma}
  \label{U40lm}
Assume that a sequence $\{\nu_n\}_{n\in \mathds{N}}\subset
\mathcal{P}_{\rm exp}^\vartheta$, $\vartheta \in \mathds{R}$, cf.
(\ref{C30A}), satisfy $\nu_n \Rightarrow \nu$ as $n\to +\infty$ for
some $\nu \in \mathcal{P}(\Gamma_*)$. Then $\nu \in \mathcal{P}_{\rm
exp}^\vartheta$. Furthermore, for each $G\in
\cap_{\vartheta}\mathcal{G}_\vartheta$, it follows that
\begin{equation}
  \label{QQa}
 \langle \! \langle k_{\nu_n}, G \rangle \! \rangle \to \langle \!
\langle k_{\nu}, G \rangle \! \rangle, \qquad n\to +\infty.
\end{equation}
\end{lemma}
\begin{proof}
By assumption $\nu_n (F) \to \nu(F)$ for each $F\in
\widehat{\mathcal{F}}$, see (\ref{C800}) and Proposition
\ref{TH1pn}. By (\ref{TH1}), (\ref{Ma2}) and (\ref{Ma1}), for given
$m\in \mathds{N}$, $\theta\in \varTheta_\psi^+$ and $\tau\in (0,1]$,
we then get
\begin{gather*}
\nu (\varPhi^m_\tau ) \leq \sup_{n\in \mathds{N}} \nu_n \leq
(\varPhi^m_\tau ) e^{m \vartheta} \|\theta\|^m.
\end{gather*}
Then the proof of $\nu \in \mathcal{P}_{\rm exp}^\vartheta$ follows
by the monotone convergence theorem and  (\ref{Len}). The validity
of (\ref{QQa}) for $G$ such that $KG \in \widehat{\mathcal{F}}$
follows by the fact just mentioned, i.e., just because $\nu$ has a
correlation function. The extension of (\ref{QQa}) to all
$G\in\cap_{\vartheta}\mathcal{G}_\vartheta$ is then made by the same
arguments as the proof of (\ref{C101a}).
\end{proof}

\section{The Existence: Approximating Processes}
In this section, we prove Theorem \ref{3tm} by constructing path
measures for the models described by $L^\alpha$, $\alpha \in (0.,1]$
introduced in the preceding section. This will be done in a direct
way by means of the corresponding Markov transition functions.

\subsection{The Markov transition functions}

The transition functions in question will be obtained in the form
\begin{equation}
  \label{Markov}
p^{\alpha}_t(\gamma, \cdot) = S^\alpha (t) \delta_\gamma, \quad
t\geq 0, \ \ \alpha\in (0,1],
\end{equation}
where $\delta_\gamma$ is the Dirac measure with atom at $\gamma\in
\Gamma_*$ and $S^\alpha = \{S^\alpha (t)\}_{t\geq 0}$ is a
stochastic semigroup of linear operators, related to the Kolmogorov
operator $L^\alpha$. Hence, we begin by constructing $S^\alpha$.

\subsubsection{Stochastic semigroups} A more detailed presentation
of the notions and facts which we introduce here can be found in
\cite{hon,Banasiak,TV}.

Let $\mathcal{E}$ be an ordered real Banach space, and
$\mathcal{E}^{+}$ be a generating cone of its positive elements. Set
$\mathcal{E}^{+,1} = \{ x\in \mathcal{E}^{+}:
\|x\|_{\mathcal{E}}=1\}$ and assume that the norm is additive on
$\mathcal{E}^{+}$, i.e., $\|x+y\|_{\mathcal{E}} =
\|x\|_{\mathcal{E}} + \|y\|_{\mathcal{E}}$ whenever $x,y \in
\mathcal{E}^{+}$. In such spaces, there exists a positive linear
functional, $\varphi_{\mathcal{E}}$, such that
\begin{equation}
  \label{V5}
  \varphi_{\mathcal{E}} (x) = \|x\|_{\mathcal{E}}, \quad x \in
  \mathcal{E}^{+}.
\end{equation}
A $C_0$-semigroup, $S=\{S(t)\}_{t\geq 0}$, of bounded linear
operators on $\mathcal{E}$ is said to be stochastic (resp.
substochastic) if the following holds $\|S(t) x \|_{\mathcal{E}} =
1$ (resp. $\|S(t) x \|_{\mathcal{E}} \leq 1$) for all $t>0$ and
$x\in \mathcal{E}^{+,1}$. Let $\mathcal{D}\subset \mathcal{E}$ be a
dense linear subspace, $\mathcal{D}^{+} =\mathcal{D}\cap
\mathcal{E}^{+}$ and $(A,\mathcal{D})$, $(B,\mathcal{D})$ be linear
operators in $\mathcal{E}$. A paramount question of the theory of
stochastic semigroups is under which conditions the closure (resp.
an extension) of $(A+B,\mathcal{D})$ is the generator of a
stochastic semigroup. Classical works on this subject trace back to
Feller, Kato, Miyadera, etc, see \cite{hon,TV}. In the present work,
we will use a result of \cite{TV}, which we present now in the form
adapted to the context.

To proceed we need to further specify the properties of the space
$\mathcal{E}$.
\begin{assumption}
  \label{1ass}
There exists a linear subspace, $\widetilde{\mathcal{E}}\subset
\mathcal{E}$, which has the following properties:
\begin{itemize}
  \item[(i)] $\widetilde{\mathcal{E}}$ is dense in $\mathcal{E}$ in
  the norm
  $\|\cdot\|_{\mathcal{E}}$.
  \item[(ii)] There exists a norm,
  $\|\cdot\|_{\widetilde{\mathcal{E}}}$, on $\widetilde{\mathcal{E}}$
  that makes it a Banach space.
  \item[(iii)] $\widetilde{\mathcal{E}}^{+}:=\widetilde{\mathcal{E}}
  \cap \mathcal{E}^{+}$ is a generating cone in
  $\widetilde{\mathcal{E}}$; $\|\cdot\|_{\widetilde{\mathcal{E}}}$
  is additive on $\widetilde{\mathcal{E}}^{+}$ and hence there
  exists a linear functional, $\varphi_{\widetilde{\mathcal{E}}}$, on
$\widetilde{\mathcal{E}}$, such that
$\|x\|_{\widetilde{\mathcal{E}}} =
\varphi_{\widetilde{\mathcal{E}}}(x)$ whenever
$x\in\widetilde{\mathcal{E}}^{+}$, cf. (\ref{V5}).
  \item[(iv)] The cone $\widetilde{\mathcal{E}}^{+}$ is dense in
  $\mathcal{E}^{+}$.
\end{itemize}
\end{assumption}
For $\mathcal{D}$ as above, set $\widetilde{\mathcal{D}} = \{x\in
\mathcal{D}\cap \widetilde{\mathcal{E}}: A x \in
\widetilde{\mathcal{E}}\}$. Then $(A,\widetilde{\mathcal{D}})$ is
the \emph{trace} of $A$ in $\widetilde{\mathcal{E}}$. The next
statement is an adaptation of \cite[Theorem 2.7]{TV}.
\begin{proposition}[Thieme-Voigt]
  \label{TVpn}
Assume that:
\begin{itemize}
  \item[(i)] $-A:\mathcal{D}^{+} \to \mathcal{E}^{+}$ and $B:\mathcal{D}^{+}
\to \mathcal{E}^{+}$;
\item[(ii)] $(A,\mathcal{D})$ is the generator of a
substochastic semigroup, $S=\{S(t)\}_{t\geq 0}$, on $\mathcal{E}$
such that $S(t):\widetilde{\mathcal{E}} \to \widetilde{\mathcal{E}}$
for all $t\geq 0$ and the restrictions $S
(t)|_{\widetilde{\mathcal{E}}}$ constitute a $C_0$-semigroup on
$\widetilde{\mathcal{E}}$ generated by $(A,
\widetilde{\mathcal{D}})$;
\item[(iii)] $B:\widetilde{\mathcal{D}} \to \widetilde{\mathcal{E}}$ and
$\varphi_{\mathcal{E}} \left( (A+B) x\right) = 0$,  for $x\in
\mathcal{D}^{+}$;
\item[(iv)] there exist $c>0$ and $\varepsilon >0$ such that
\[
\varphi_{\widetilde{\mathcal{E}}} \left( (A+B) x\right) \leq c
\varphi_{\widetilde{\mathcal{E}}} (x) - \varepsilon \|A
x\|_{\mathcal{E}}, \qquad {\rm for}  \ \ x\in
\widetilde{\mathcal{D}}\cap \mathcal{E}^{+}.
\]
\end{itemize}
Then the closure of $(A+B,\mathcal{D})$ in $\mathcal{E}$ is the
generator of a stochastic semigroup, $S_{\mathcal{E}}=
\{S_{\mathcal{E}}(t)\}_{t\geq 0}$, on $\mathcal{E}$ which leaves
$\widetilde{\mathcal{E}}$ invariant. The restrictions
$S_{\widetilde{\mathcal{E}}}(t):=S_{\mathcal{E}}(t)|_{\widetilde{\mathcal{E}}}$,
$t\geq 0$ constitute a $C_0$-semigroup,
$S_{\widetilde{\mathcal{E}}}$, on $\widetilde{\mathcal{E}}$
generated by the trace of the generator of $S_{\mathcal{E}}$ in
$\widetilde{\mathcal{E}}$.
\end{proposition}
\begin{remark}
  \label{1rk}
Without assuming item (iv) above one can only guarantee that an
extension of $(A+B,\mathcal{D})$ is the generator of a substochastic
semigroup on $\mathcal{E}$, which corresponds to a dishonesty of the
evolution described by this semigroup. More on this item can be
found in \cite{hon}.
\end{remark}
Now we turn to constructing the semigroups $S^\alpha$.

\subsubsection{The Banach spaces of measures}

Let $\mathcal{M}$ be the linear space of finite signed measures on
$(\Gamma, \mathcal{B}(\Gamma))$, see \cite[Chapter 4]{Cohn}. That
is, $\mu\in \mathcal{M}$ is a $\sigma$-additive map
$\mu:\mathcal{B}(\Gamma) \to \mathds{R}$ which takes only finite
values. By $\mathcal{M}^{+}$ we denote the set of all such $\mu$
that take only nonnegative values. Then the Jordan decomposition of
$\mu$ is the unique representation $\mu = \mu^{+} - \mu^{-}$ with
$\mu^{\pm} \in \mathcal{M}^{+}$. Thus, $\mathcal{M}^{+}$ is a
generating cone. Set $|\mu| = \mu^{+} + \mu^{-}$. Then
\begin{equation}
  \label{Ma40}
\|\mu\|:=|\mu|(\Gamma)
\end{equation}
is a norm, that is clearly additive on $\mathcal{M}^{+}$. By
\cite[Proposition 4.1.8, page 119]{Cohn} with this norm
$\mathcal{M}$ is a Banach space. Let $\Psi_1$ be the function
defined in (\ref{C3}). For $n\in \mathds{N}$, let $\mathcal{M}_n$ be
the subset of $\mathcal{M}$ consisting of all those $\mu$ for which
$\Psi_1^n \mu$ are finite signed measures. Recall that $\Psi_1 = 1
+\Psi$, see (\ref{Psi}). We equip $\mathcal{M}_n$ with the norm
\begin{equation}
  \label{V6}
  \|\mu\|_n = \int_{\Gamma} \Psi_1^n (\gamma) |\mu|(d \gamma)=: \varphi_n (|\mu|).
\end{equation}
By the same \cite[Proposition 4.1.8, page 119]{Cohn} with this norm
$\mathcal{M}_n$ is a Banach space. Now for $\beta>0$, let
$\mathcal{M}_\beta$ be the subset of $\mathcal{M}$ the elements of
which remain finite measures being multiplied by $\exp(\beta \Psi_0
(\gamma))$. We equip it with the norm
\begin{equation*}
 \|\mu\|_\beta = \int_{\Gamma} \exp(\beta \Psi_0 (\gamma)) |\mu|
 (d\gamma) =: \varphi_\beta (|\mu|).
\end{equation*}
Then also $(\mathcal{M}_\beta, \|\cdot \|_\beta)$ is a Banach space.
By (\ref{N1}) and (\ref{C4}) it follows that
\begin{equation*}
\forall \mu \in \mathcal{M}_1 \qquad |\mu|(\Gamma_*) =
|\mu|(\Gamma).
\end{equation*}
That is, for each $\mu \in \mathcal{M}_1$, it follows that
$|\mu|(\Gamma_*^c) =0$. Define
\begin{equation*}
 \mathcal{M}_* = \{ \mu \in \mathcal{M}: |\mu|(\Gamma_*^c) =0\}.
\end{equation*}
Thus, $\mathcal{M}_1 \subset \mathcal{M}_*$. Obviously, also all
$\mathcal{M}_n$ and $\mathcal{M}_\beta$ have the same property. For
a subset, $\mathcal{M}'\subset \mathcal{M}$, let
$\overline{\mathcal{M}'}$ denote its closure in $\|\cdot\|$ defined
in (\ref{Ma40}).
\begin{lemma}
  \label{V1lm}
For each $n\in \mathds{N}$ and $\beta >0$, it follows that
\begin{equation}
  \label{V7b}
 \overline{\mathcal{M}_n} =  \overline{\mathcal{M}_\beta} =
 \mathcal{M}_*.
\end{equation}
\end{lemma}
\begin{proof}
Obviously, for each $n\in \mathds{N}$ and $\beta >0$, the following
holds $\mathcal{M}_\beta \subset \mathcal{M}_n$. Then it is enough
to prove the validity of (\ref{V7b}) for $\mathcal{M}_\beta$. Let us prove
the inclusion $\overline{\mathcal{M}_\beta} \subset
\mathcal{M}_{*}$. For a given $\mu\in \overline{\mathcal{M}_\beta}$,
let $\{\mu_n\}_{n\in \mathds{N}} \subset \mathcal{M}_\beta$ be a
sequence such that $\|\mu - \mu_n\| \to 0$. Fix $n$ and let
then $\Gamma = \mathbb{P}\cup \mathbb{N}$ be the Hahn decomposition
for $\mu - \mu_n$, i.e.,  $\mu(\mathbb{A}) \geq \mu_n(\mathbb{A})$
for each $\mathbb{A}\subset \mathbb{P}$, and $\mu(\mathbb{A}) \leq
\mu_n(\mathbb{A})$ for each $\mathbb{A}\subset \mathbb{N}$. Then
\begin{gather*}
\|\mu -\mu_n \| = (\mu - \mu_n)(\mathbb{P}) + (\mu_n -
\mu)(\mathbb{N}) \geq (\mu - \mu_n)(\mathbb{P}\cap\Gamma_*^c) +
(\mu_n - \mu)(\mathbb{N}\cap\Gamma_*^c) \\[.2cm]
= \mu (\mathbb{P}\cap\Gamma_*^c) - \mu(\mathbb{N}\cap\Gamma_*^c) =
\mu^{+} (\Gamma^c_*) + \mu^{-} (\Gamma^c_*) =  |\mu|(\Gamma_*^c),
\end{gather*}
where we have taken into account that $|\mu_n|(\Gamma_*^c)=0$. Then the assumed convergence $\mu_n\to \mu$ yields that
$\mu\in \mathcal{M}_*$. To prove the opposite inclusion we take an
arbitrary $\mu\in \mathcal{M}_*$ and write its Jordan decomposition
$\mu = \mu^{+} - \mu^{-}$. For a given $n\in \mathds{N}$, let $I_n$
be the indicator of the set $\Gamma_{*,n}$ defined in (\ref{N1}).
Then both $\mu^{\pm}_n := I_{n} \mu^{\pm}$ are in
$\mathcal{M}_\beta$. At the same time, by (\ref{N1}) the sequence of
function $J_{n} (\gamma) := 1 - I_{n}(\gamma)$ converges to zero
pointwise on $\Gamma_*$. Since $\mu\in \mathcal{M}_*$, we have
\begin{gather}
\label{V7d} \|\mu^{\pm}-\mu_n^{\pm}\| = \int_{\Gamma} J_{n} (\gamma)
\mu^{\pm} ( d\gamma) = \int_{\Gamma_*}J_{n} (\gamma) \mu^{\pm} (
d\gamma) \to 0, \quad {\rm as}  \ \ n\to +\infty.
\end{gather}
By the triangle inequality we then obtain that $\|\mu -\mu_n\| \to
0$, where $\mu_n := \mu^{+}_n - \mu^{-}_n \in \mathcal{M}_\beta$.
\end{proof}
By the very definition of the spaces $\mathcal{M}_n$,
$\mathcal{M}_\beta$ and $\mathcal{M}_*$, we conclude that they have
generating cones of positive elements consisting of those $\mu$ that take
nonnegative values only.
\begin{corollary}
  \label{V1co}
The set $\mathcal{M}_*$ equipped with the norm $\|\cdot\|$ defined
in (\ref{Ma40}) is a Banach space. Let $\mathcal{M}_*^{+}$ be its
cone of positive elements. Then for each $n\in \mathds{N}$ and
$\beta>0$, it follows that
\[
 \overline{\mathcal{M}_n^{+}} =  \overline{\mathcal{M}_\beta^{+}} =
 \mathcal{M}_*^{+},
\]
where we mean the closure in the norm of $\mathcal{M}_*$.
\end{corollary}
\begin{proof}
 The first part of the statement follows directly by (\ref{V7b}).
 The  second part is obtained by the construction used in
 (\ref{V7d}).
\end{proof}
Let $\mathcal{M}_*^{+,1}$ be the subset of $\mathcal{M}_*$
consisting of probability measures, i.e., for which $\|\mu\| =
\mu(\Gamma) = \mu(\Gamma_*) = 1$. Then by (\ref{C4}) it follows that
\begin{equation*}
\mathcal{P}_{\rm exp} \subset \mathcal{M}_1^{+,1} \subset
\mathcal{M}_*^{+,1}.
\end{equation*}
By (\ref{N701}), for each $\beta >0$, we also have
\begin{equation*}
\mathcal{P}_{\rm exp} \subset \mathcal{M}_\beta^{+,1} :=
\mathcal{M}_\beta \cap \mathcal{M}^{+,1}_* \subset \mathcal{M}_n
\cap \mathcal{M}^{+,1}_*, \quad {\rm for} \ {\rm all} \ n \in
\mathds{N} .
\end{equation*}

\subsubsection{The stochastic semigroup}

For a given $\alpha\in (0,1]$, set
\begin{equation}
  \label{NN3}
\Phi_\alpha (\gamma) = \sum_{x\in \gamma}\int_{\mathds{R}^d}
a_\alpha(x,y) \exp\left( - \sum_{z\in \gamma\setminus x} \phi(y-z)
\right) dy , \qquad \gamma \in \Gamma_*.
\end{equation}
Since $\psi(x) < \psi_\alpha(x) \leq \psi(x)/\alpha$, see
(\ref{C3}), for all $\alpha\in (0,1]$ we have
\begin{equation}
  \label{NN4}
\Psi(\gamma) < \Phi_\alpha (\gamma) \leq \Psi (\gamma)/\alpha,
\end{equation}
and hence $\Phi_\alpha (\gamma) < \infty$ for $\gamma \in \Gamma_*$.
Now let $L^\alpha$ be the corresponding Kolmogorov operator
(\ref{Lalpha}). Our aim is to define its `predual',
$L^{\dag,\alpha}$, acting according to the rule
\begin{equation}
  \label{V11}
 \mu (L^\alpha F) = (L^{\dag,\alpha} \mu) (F),
\end{equation}
for appropriate $\mu \in \mathcal{P}(\Gamma_*)$ and $F:\Gamma_* \to
\mathds{R}$, and then to use it to define the corresponding
operators acting in the spaces of measures just introduced.
Obviously,  we can restrict ourselves to  the elements of
$\mathcal{M}_*$. By (\ref{Lalpha}) and (\ref{N2a}) we thus obtain it
in the form
\begin{equation}
  \label{V12}
L^{\dag,\alpha} = A + B
\end{equation}
where $A$ is the multiplication operator by the function $-
\Phi_\alpha$ defined in (\ref{NN3}). In view of (\ref{NN4}) the
domain of $A$ is to be
\begin{equation}
  \label{V13}
 \mathcal{D} = \{ \mu \in \mathcal{M}_*: \Phi_\alpha \mu \in
 \mathcal{M}_*\} = \mathcal{M}_1.
\end{equation}
To define $B$ we introduce the following measure kernel
\begin{equation}
  \label{N10}
  \varOmega^\gamma_\alpha ({ \mathbb{A}}) = \sum_{x\in \gamma} \int_{\mathds{R}^d} a_\alpha
  (x ,y) \exp\left(- \sum_{z\in \gamma\setminus x} \phi(y-z) \right)
  \mathds{1}_{\mathbb{A}} (\gamma\setminus x \cup y) d y,
\end{equation}
with $\gamma \in \Gamma_*$ and $\mathbb{A}\in
\mathcal{B}(\Gamma_*)$. By (\ref{NN3}) we then have
\begin{equation}
  \label{N11}
\varOmega^\gamma_\alpha (\Gamma_*) = \Phi_\alpha (\gamma).
\end{equation}
Next, define
\begin{equation}
  \label{N12}
 (B \mu) (\mathbb{A}) = \int_{\Gamma_*} \varOmega^\gamma_\alpha (\mathbb{A}) \mu(d
 \gamma).
\end{equation}
Note that
\begin{equation}
  \label{N12a}
B: \mathcal{M}_1^{+}\to \mathcal{M}_*^{+}.
\end{equation}
Moreover, for $\mu\in \mathcal{M}^{+}_1$, by (\ref{N11}) and
(\ref{N12}) we have
\begin{equation}
  \label{N13}
\| B\mu\|=(B \mu) (\Gamma_*) = \int_{\Gamma_*} \Phi_\alpha (\gamma)
\mu( d\gamma) = - (A \mu)(\Gamma_*).
\end{equation}
Hence, we can take $\mathcal{M}_1$ as the domain of $B$ and then
define $L^{\dag,\alpha}$ by (\ref{V12}) with domain
$\mathcal{D}=\mathcal{M}_1$, see (\ref{V13}).

In the sequel, we will use one more property of $B$. By (\ref{N10}),
(\ref{N12}) and (\ref{V6}) we get
\begin{eqnarray}
  \label{K2a}
\varphi_n(B\mu) & = & \int_{\Gamma_*} \Psi^n_1 (\gamma) (B\mu)(d
\gamma)
\\[.2cm] \nonumber
& = & \int_{\Gamma_*}\left( \sum_{x\in \gamma} \int_{\mathds{R}^d}
a_\alpha (x,y) \exp\left(- \sum_{z\in \gamma\setminus x}
\phi(y-z)\right) \Psi^n_1(\gamma\setminus x \cup y) d y\right)
\mu(d\gamma).
\end{eqnarray}
By (\ref{C3}) it follows that
\begin{equation}
  \label{N16a}
 \Psi^n_1(\gamma\setminus x\cup y) = \left(\Psi_1(\gamma) + \psi(y) -
\psi(x) \right)^n \leq 2^n \Psi^n_1(\gamma).
\end{equation}
We apply this and (\ref{NN4}) in (\ref{K2a}) and obtain
\begin{equation*}
 \forall \mu\in \mathcal{M}_{n+1}^{+} \qquad \|B\mu\|_n =
 \varphi_n(B\mu) \leq (2^n/\alpha) \|\mu\|_{n+1}.
\end{equation*}
This yields the following extension of (\ref{N12a})
\begin{equation}
  \label{N12b}
B: \mathcal{M}_{n+1}^{+} \to \mathcal{M}_n^{+},
\end{equation}
holding for all $n\in \mathds{N}$. Since $\|A\mu\|_{n} \leq
\alpha^{-1} \|\mu\|_{n+1}$, by (\ref{N12b}) we also get
\begin{equation*}
\forall n\in \mathds{N}_0 \qquad  L^{\dag,\alpha}: \mathcal{M}_{n+1}
\to \mathcal{M}_n,
\end{equation*}
that can be used to define the powers of $L^{\dag,\alpha}$
\begin{equation}
  \label{N12bb}
(L^{\dag,\alpha})^m :\mathcal{M}_{n+m} \to \mathcal{M}_n, \qquad
n\in \mathds{N}_0, \ m\in \mathds{N}.
\end{equation}
Here -- and in the sequel in similar expressions -- $\mathcal{M}_0$
(corresponding to $\mathcal{M}_n$ with $n=0$) is understood as
$\mathcal{M}_*$ Let us now define a bounded linear operator
$L^{\dag,\alpha}_{\beta'\beta}:\mathcal{M}_\beta  \to
\mathcal{M}_{\beta'} $, $\beta'<\beta$, the action of which is the
same as that of the unbounded operator $L^{\dag,\alpha}=A+B$ defined
in (\ref{V12}) and (\ref{N12}). For a given $\mu\in \mathcal{M}_*$,
let $\mu=\mu^{+}-\mu^{-}$ be its Jordan decomposition. Then
\[
L^{\dag,\alpha} \mu = \left( B\mu^{+} - A \mu^{-}\right) - \left(
B\mu^{-} - A \mu^{+}\right) =: \mu_1^{+} - \mu_1^{-}, \qquad
\mu^{\pm}_1 \in \mathcal{M}_*^{+}.
\]
This yields that
\begin{equation}
  \label{N12z}
\|L^{\dag,\alpha} \mu\|_{\beta'} \leq \|\mu_1^{+}\|_{\beta'} +
\|\mu^{-}_1\|_{\beta'} = \|B\mu^{+}\|_{\beta'} +
\|B\mu^{-}\|_{\beta'} + \|A\mu^{+}\|_{\beta'} +
\|A\mu^{-}\|_{\beta'},
\end{equation}
holding for all $\mu\in \mathcal{M}_\beta$. Here we have used the
additivity of the norms on the positive cone as well as the
positivity of $B$ and $-A$. By (\ref{NN4}) and the following evident
inequality $xe^{-\kappa x}\leq 1/ e \kappa$ holding for all positive
$x$ and $\kappa$, we obtain
\begin{equation}
  \label{N12w}
\Phi_\alpha (\gamma) \exp\left( \beta' \Psi_{0}(\gamma) \right) \leq
\frac{1}{\alpha e(\beta - \beta')}\exp\left( \beta \Psi(\gamma)
\right).
\end{equation}
By (\ref{N12w}) for $\mu\in \mathcal{M}^{+}_\beta$, we then get
\begin{equation}
  \label{N12v}
\|A \mu\|_{\beta'} \leq \frac{\|\mu\|_\beta}{\alpha e(\beta -
\beta')}.
\end{equation}
Next, similarly as in (\ref{K2a}) it follows that
\begin{eqnarray*}
& & \int_{\Gamma_*}\exp\left( \beta' \Psi (\gamma) \right) (B\mu)(d
\gamma)
\\[.2cm] \nonumber
& & \quad =  \int_{\Gamma_*}\bigg{(} \sum_{x\in \gamma}
\int_{\mathds{R}^d} a_\alpha (x,y) \exp\left(- \sum_{z\in
\gamma\setminus x}
\phi(y-z)\right) \\[.2cm] &  & \quad\times  \exp\left( \beta' \Psi (\gamma)  \right)
\exp\left( \beta'[\psi(y) - \psi(x)]\right) d y\bigg{)}
\nonumber
\mu(d\gamma) \\[.2cm] \nonumber & & \leq e^\beta  \int_{\Gamma_*}
\Phi_\alpha(\gamma)  \exp\left( \beta' \Psi (\gamma)  \right) \mu( d
\gamma) \leq \frac{e^\beta \|\mu\|_\beta}{\alpha e(\beta - \beta')}.
\end{eqnarray*}
We combine this estimate with (\ref{N12v}) and (\ref{N12z}) to
obtain
\begin{equation*}
\|L^{\dag,\alpha}_{\beta'\beta}\| \leq \frac{e^\beta + 1}{\alpha
e(\beta - \beta')}.
\end{equation*}
In a similar way, for each $n\in \mathds{N}$, we also obtain,  cf.
(\ref{A15}),
\begin{equation}
  \label{N12l}
\|(L^{\dag,\alpha})^n_{\beta'\beta}\| \leq \left(\frac{n}{e
T_\alpha(\beta,\beta')}\right)^n, \qquad T_\alpha(\beta , \beta') :=
\frac{\alpha(\beta - \beta')}{e^\beta +1}.
\end{equation}
By (\ref{N12bb}), for each $n\in \mathds{N}$ and $\mu \in
\mathcal{M}_\beta $, we have that $(L^{\dag,\alpha})^n\mu \in
\mathcal{M}_{\beta'}$, $\beta'<\beta$, and the following holds
\begin{equation}
  \label{N12m}
(L^{\dag,\alpha})_{\beta'\beta}^n \mu = (L^{\dag,\alpha})^n \mu.
\end{equation}
\begin{lemma}
  \label{C1lm}
For each $\alpha\in (0,1]$, the closure of $(L^{\dag,\alpha},
\mathcal{M}_1 )$ in $\mathcal{M}_*$ is the generator of a stochastic
semigroup, $S^\alpha =\{S^\alpha(t)\}_{t\geq 0}$, in $\mathcal{M}_*$
such that $S^\alpha(t):\mathcal{M}_n\to \mathcal{M}_n$ for each
$n\in \mathds{N}$. The restrictions $S^\alpha(t)|_{\mathcal{M}_n}$
constitute a $C_0$-semigroup on $\mathcal{M}_n$. Moreover, for each
$\beta >0$ and $\beta' \in (0,\beta)$, $S^\alpha (t) :
\mathcal{M}_\beta^{+}  \to \mathcal{M}_{\beta'}^{+} $ for $t <
T_\alpha(\beta, \beta')$, see (\ref{N12l}).
\end{lemma}
\begin{proof}

The construction of the semigroup in question will be made, in
particular, by showing that all the conditions of Proposition
\ref{TVpn} are met. We thus begin by checking whether each of the
spaces $\mathcal{M}_n$ and $\mathcal{M}_\beta$ enjoys the properties
listed in Assumption \ref{1ass}. By Lemma \ref{V1lm} the density
assumed in (i) is guaranteed. Each of these spaces is a Banach space
with the corresponding norm, that was already mentioned in the
course of their introduction. The properties assumed in (iii) are
evident, whereas (iv) follows by Corollary \ref{V1co}. Thus, we can
start checking the validity of the conditions imposed in Proposition
\ref{TVpn}. Recall that both $A$ and $B$ are (densely) defined on
the domain $\mathcal{D}=\mathcal{M}_1$, see (\ref{V13}) and Lemma
\ref{V1lm}, and $A$ is the multiplication operator by the function
$(-\Phi_\alpha)$. Hence, condition (i) of Proposition \ref{TVpn} is
satisfied. Moreover, $A$ generates the semigroup $S$ consisting of
the following operators
\begin{equation}
\label{C18} (S(t) \mu)(d \gamma) = \exp\left( - t \Phi_\alpha
(\gamma) \right) \mu ( d \gamma).
\end{equation}
Then
\begin{equation}
  \label{C19}
  \|S(t)\mu\| \leq \|\mu\|,
\end{equation}
which obviously holds for all $\mu\in \mathcal{M}_*$. To check
whether $S$ is strongly continuous in $\mathcal{M}_*$, for a given
$\mu\in\mathcal{M}_* $ and $\varepsilon>0$, we have to find
$\delta>0$ such that $\|\mu - S(t) \mu\|<\varepsilon$ for all
$t<\delta$. Since $\mathcal{M}_* $ is the $\|\cdot\|$-closure of
$\mathcal{M}_1$ (by Lemma \ref{V1lm}), for the chosen $\mu$, one
finds $\mu' \in \mathcal{M}_1$ such that $\|\mu - \mu'\| <
\varepsilon/3$. Then by (\ref{C18}) and (\ref{C19}) we get
\begin{eqnarray}
\label{C20} \|\mu - S(t) \mu\| & \leq & \|\mu-\mu'\| +
\|S(t)(\mu-\mu')\| + \|\mu' - S(t) \mu'\| \\[.2cm] \nonumber
&\leq &  t\|A \mu'\| + 2\varepsilon /3 \leq ( t/\alpha)\|\mu'\|_1 +
2\varepsilon /3,
\end{eqnarray}
which completes the proof for $\mathcal{M}_*$. Clearly, $S(t) :
\mathcal{M}^{+}_n \to \mathcal{M}^{+}_n$, and the domain of the
trace of $A$ in $\mathcal{M}_n$ is
$\widetilde{\mathcal{D}}_n=\mathcal{M}_{n+1}$. Then the proof that
$S(t)|_{\mathcal{M}_n}$ is strongly continuous in $\mathcal{M}_{n}$
can be performed similarly as in (\ref{C20}). Thus, condition (ii)
of Proposition \ref{TVpn} is met. In view of (\ref{N12b}) to
complete the proof of (iii) we have to show $\varphi((A+B)\mu)=0$
whenever $\mu \in \mathcal{M}_1^{+}$, which is obviously the case by
(\ref{N13}). Then it remains to show that, for a fixed $n\in
\mathds{N}$,
\begin{equation}
  \label{N15}
  \int_{\Gamma_*} \Psi_1^n (\gamma) (L^{\dag,\alpha} \mu)(d\gamma) \leq c \int_{\Gamma_*} \Psi_1^n (\gamma)
  \mu(d\gamma) - \varepsilon \int_{\Gamma_*} \Phi_\alpha (\gamma)
  \mu(d\gamma),
\end{equation}
holding for each $\mu \in\mathcal{M}_{n+1}^{+}$ and some positive
$c$ and $\varepsilon$, possibly dependent on $n$. In view of the
following estimate, cf. (\ref{NN4}),
\begin{equation*}
  \alpha \Phi_\alpha (\gamma) \leq 1 + n \sum_{x\in \gamma} \psi(x)
  \leq \Psi^n_1 (\gamma), \qquad n\in \mathds{N}, \ \  \gamma \in
  \Gamma_*,
\end{equation*}
it is enough to show (\ref{N15}) with $\varepsilon =0$ and
sufficiently big $c$. By (\ref{V11}) this amounts to showing
\begin{equation}
  \label{N16}
  L^\alpha \Psi^n_1 (\gamma) \leq c \Psi^n_1 (\gamma), \qquad \gamma\in
  \Gamma_*.
\end{equation}
By (\ref{N16a}) it follows that
\begin{eqnarray}
  \label{N17}
|\Psi^n_1 (\gamma\setminus x \cup y) - \Psi^n_1 (\gamma)| \leq 2^n
|\psi(y) - \psi(x) | \Psi^{n-1}_1(\gamma).
\end{eqnarray}
Assume that $|x|>|y|$. By (\ref{C3}) we have
\begin{eqnarray}
  \label{N18}
|\psi(y) - \psi(x) | & = & \left( |x|^{d+1} - |y|^{d+1}\right)
\psi(x)
\psi(y) \\[.2cm] \nonumber &\leq & \left[ (|x-y|+|y|)^{d+1} -
|y|^{d+1}\right] \psi(x)
\psi(y) \\[.2cm] \nonumber & = & \sum_{l=1}^{d+1} {d+1 \choose
l}|x-y|^l |y|^{d+1-l} \psi(x) \psi(y)  \\[.2cm] \nonumber & \leq & \psi(x)\sum_{l=1}^{d+1} {d+1 \choose
l}|x-y|^l \\[.2cm] \nonumber & := & \psi(x) v(|x-y|).
\end{eqnarray}
For $|y|>|x|$, in a similar way we get
\begin{eqnarray}
  \label{N19}
|\psi(y) - \psi(x) | & \leq & \sum_{l=1}^{d+1} {d+1 \choose
l}|x-y|^l |x|^{d+1-l} \psi(x) \psi(y)  \\[.2cm] \nonumber & \leq & \sum_{l=1}^{d+1} {d+1 \choose
l}|x-y|^l |y|^{d+1-l} \psi(x) \psi(y)  \\[.2cm] \nonumber & \leq & \psi(x)\sum_{l=1}^{d+1} {d+1 \choose
l}|x-y|^l = \psi(x) v(|x-y|).
\end{eqnarray}
Now we apply (\ref{N17}), (\ref{N18}) and (\ref{N19}) to obtain
\begin{eqnarray*}
{\rm LHS}(\ref{N16}) &\leq & 2^n \sum_{x\in \gamma}
\int_{\mathds{R}^d} \psi_\alpha (x) \psi(x) a(x-y ) v(|x-y|)
\Psi^{n-1}_1(\gamma) d y
\\[.2cm] \nonumber &\leq & 2^n \Psi^{n-1}_1(\gamma) \left( \sum_{x\in
\gamma}\psi(x)\right) \int_{\mathds{R}^d} a (y) v(|y|) d y
\\[.2cm] \nonumber &\leq & \left(2^n \int_{\mathds{R}^d} a (y) v(|y|) d y
\right) \Psi^n_1(\gamma),
\end{eqnarray*}
that by (\ref{C8}) proves (\ref{N16}). Thus, all the conditions of
Proposition \ref{TVpn} are met, which proves the part of the lemma
related to the stochastic semigroup $S^\alpha$ acting in
$\mathcal{M}_*$ and its restrictions to $\mathcal{M}_n$, $n\in
\mathds{N}$. To prove the second part of the lemma we use the
estimate in (\ref{N12l}) and define bounded operators
$S^\alpha_{\beta'\beta}(t):\mathcal{M}_\beta \to
\mathcal{M}_{\beta'}$, $t < T_\alpha(\beta,\beta')$ by the series
\[
 S^\alpha_{\beta'\beta}(t) = \mathds{I}_{\beta' \beta}+
\sum_{n=1}^\infty \frac{t^n}{n!} (L^{\dag,\alpha})^n_{\beta'\beta},
\]
where $\mathds{I}_{\beta' \beta}$ is the embedding operator. By the
latter formula and (\ref{N12m}) we conclude that
\begin{equation}
  \label{N12n}
\forall \mu \in \mathcal{M}_\beta \qquad S^\alpha (t) \mu =
S^\alpha_{\beta'\beta} (t) \mu, \quad t<T_\alpha(\beta,\beta').
\end{equation}
By (\ref{N12l}) $S^\alpha_{\beta'\beta} (t):\mathcal{M}_\beta  \to
\mathcal{M}_{\beta'}$ is a bounded operator, the norm of which
satisfies
\begin{equation*}
\|S^\alpha_{\beta'\beta} (t)\| \leq
\frac{T_\alpha(\beta,\beta')}{T_\alpha(\beta,\beta')-t}.
\end{equation*}
The positivity of $S^\alpha_{\beta'\beta} (t)$ follows by
(\ref{N12n}) and the positivity of $S^\alpha (t)$. This completes
the proof.
\end{proof}
\vskip.1cm \noindent In view of (\ref{Markov}), we conclude that Lemma \ref{C1lm} establishes
the existence of the transition function corresponding to
$L^\alpha$, with the properties arising from the corresponding
properties of the semigroup. Note that $\delta_\gamma\in
\mathcal{M}_*$ (and hence in all $\mathcal{M}_n$ and
$\mathcal{M}_\beta$) if and only if $\gamma\in \Gamma_*$, that will
be assumed below. It is straightforward
that $p^{\alpha}_t(\gamma, \cdot)$ satisfies the corresponding
standard conditions and thus determines finite-dimensional
distributions of a Markov process, see \cite[pages 156, 157]{EK}.
Our next step is to show that it has cadlag versions.

\subsection{Constructing path measures}

The construction of the families of path measures $P^\alpha$ which
solve the martingale problem in the sense of Definition \ref{A1df}
can be done by defining their finite dimensional marginals with the
help of the transition function (\ref{Markov}). In this case,
however, the one dimensional marginals
\begin{equation}
  \label{Ma60}
 \Pi^\alpha_t = S^\alpha(t) \mu = \int_{\Gamma_*} p^\alpha_t (\gamma,
 \cdot) \mu(d\gamma),
\end{equation}
need not be in $\mathcal{P}_{\rm exp}$, even for $\mu\in
\mathcal{P}_{\rm exp}$. The only fact guaranteed by Lemma \ref{C1lm}
is that $\Pi^\alpha_t \in \mathcal{M}_n$, for all $n\in \mathds{N}$,
and that $\Pi^\alpha_t \in \mathcal{M}_\beta$ with $t$ belonging to
a bounded interval. This obstacle is removed by the following
statement.
\begin{lemma}
  \label{C3lm}
For a given $\mu\in \mathcal{P}_{\rm exp}$, let $\{\mu^\alpha_t:
\mu_0 = \mu, \ t\geq 0\}\subset \mathcal{P}_{\rm exp}$ be the family
of measures defined by their correlation functions $k^\alpha_t$
according to (\ref{Ma32}). For the same $\mu$, let $\Pi^\alpha_t$,
$t\geq 0$ be as in (\ref{Ma60}). Then $\Pi_t^\alpha=\mu_t^\alpha$
for all $t>0$ and $\alpha \in (0,1]$.
\end{lemma}
\begin{proof} Exactly as in Theorem \ref{2tm} one proves that the
Fokker-Planck equation (\ref{A12a}) with $L$ replaced by $L^\alpha$
has a unique solution, which is $\mu_t^\alpha$. At the same time, by
construction $\Pi^\alpha_t$ also solves this equation.
\end{proof}
\vskip.1cm \noindent Now we can start constructing the path measures
in question. To this end we use Chentsov's theorem in the following
version, see \cite[Theorems 8.6 -- 8.8, pages 137--139]{EK}. Recall
that the complete metric $\upsilon_*$ of $\Gamma_*$ was introduced
in (\ref{Psi4}). For $\alpha\in (0,1]$, $\gamma\in \Gamma_*$ and
$u,v\geq 0$, set
\begin{gather}
  \label{Ma50}
w^\alpha_u(\gamma) = \int_{\Gamma_*} \upsilon_* (\gamma,\gamma')
p^\alpha_u ( \gamma, d \gamma'), \\[.2cm] \nonumber
W^\alpha_{u,v} (\gamma) =  \int_{\Gamma_*} \upsilon_*
(\gamma,\gamma') w^\alpha_u(\gamma') p^\alpha_v ( \gamma, d
\gamma').
\end{gather}
Thereafter, for $t_3\geq t_2\geq t_1 \geq 0$ (such sets are called
{\em triples}), let us consider
\begin{equation}
  \label{Ma51}
\mathcal{W}^\alpha(t_1, t_2, t_3) = \int_{\Gamma_*}
W^\alpha_{t_3-t_2,t_2-t_1} (\gamma') \Pi^\alpha_{t_1}(d \gamma') =
\int_{\Gamma_*} W^\alpha_{t_3-t_2,t_2-t_1} (\gamma')
\mu^\alpha_{t_1}(d \gamma'),
\end{equation}
where $\mu$, $\mu_t^\alpha$ and $\Pi_t^\alpha$ are as in
(\ref{Ma60}).
\begin{proposition} (Chentsov)
  \label{MApn}
Assume that there exists $C_\alpha>0$ and $\delta>0$ such that, for
each triple $t_1,t_2,t_3$, the following holds
\begin{equation}
  \label{Ma52}
\mathcal{W}^\alpha(t_1, t_2, t_3) \leq C_\alpha |t_3-t_1|^2, \qquad
t_3 - t_1 < \delta.
\end{equation}
Then the following is true:
\begin{itemize}
  \item[(i)] The transition function (\ref{Markov}) and $\mu\in \mathcal{P}_{\rm exp}$ determine a probability
measure $P^\alpha$ on $\mathfrak{D}_{\mathds{R}_{+}}(\Gamma_*)$.
\item[(ii)] If the estimate in (\ref{Ma52}) holds uniformly in
$\alpha$, i.e., with some $C>0$ independent of $\alpha\in (0,1]$,
and if the family $\{\Pi_t^\alpha: \alpha \in (0,1]\} \subset
\mathcal{P}(\Gamma_*)$
 is tight for each $t>0$, then the family $\{P^\alpha:\alpha \in
 (0,1]\}$ of measures as in (i)
 is also tight, and hence possesses accumulation point in the weak
 topology.
\end{itemize}
\end{proposition}
Note that the tightness of the family $\{\Pi_t^\alpha: \alpha \in
(0,1]\}$ follows by Lemmas \ref{C3lm} and \ref{U3lm}.
\begin{lemma}
  \label{MAlm}
For each $\mu\in \mathcal{P}_{\rm exp}$, the estimate in
(\ref{Ma52}) holds true for all $\alpha\in (0,1]$ with one and the
same $C>0$.
\end{lemma}
\begin{proof}
By (\ref{Markov}) and standard semigroup formulas, e.g., \cite[page
9]{EK}, we have
\begin{equation}
  \label{Starr}
p^{\alpha}_u (\gamma, \cdot) = \delta_\gamma + \int_{0}^u
L^{\dag,\alpha}
 p^{\alpha}_s (\gamma, \cdot)d s,
\end{equation}
since $\delta_\gamma\in \mathcal{D}=\mathcal{M}_1$. Then by this
formula and (\ref{Ma50}) we obtain
\begin{gather}
  \label{D6}
 w_u^\alpha (\gamma) = w_0^\alpha (\gamma) + \int_0^u
 \left(
 \int_ {\Gamma_*} \upsilon_* (\gamma, \gamma') (L^{\dag,\alpha}
 p^{\alpha}_s  (\gamma, d \gamma') \right) ds\\[.2cm] \nonumber =
 \int_0^u \left(
 \int_ {\Gamma_*} (L^\alpha  \upsilon_* (\gamma, \gamma') p^{\alpha}_s  (\gamma, d
 \gamma') \right) ds ,
\end{gather}
where we have taken into account that $w_0^\alpha (\gamma) =
\upsilon_*(\gamma, \gamma) =0$ as $\upsilon_*$ is a metric. We apply
now $L^\alpha$  to $\upsilon_*(\gamma, \cdot)$ -- which is a bounded
continuous function of $\gamma'$, and obtain
\begin{eqnarray*}
& & J^{\gamma}(\gamma') := (L^\alpha \upsilon_*)(\gamma,\gamma') \\[.2cm]
\nonumber& & \quad = \sum_{x\in \gamma'}\int_{\mathds{R}^d} a_\alpha
(x,y) \exp\left(-\sum_{z\in \gamma'\setminus x}
\phi(y-z)\right)\left[\upsilon_*(\gamma, \gamma'\setminus x \cup y)
- \upsilon_*(\gamma, \gamma')\right] dy.
\end{eqnarray*}
By the triangle inequality for $\upsilon_*$ we then get from the
latter
\begin{equation}
  \label{D8}
|J^{\gamma}(\gamma')| \leq \sum_{x\in \gamma'}\int_{\mathds{R}^d}
a_\alpha (x,y) \upsilon_*(\gamma'\setminus x\cup y, \gamma') d y.
\end{equation}
In view of (\ref{Psi4}), to estimate $\upsilon_*(\gamma'\setminus
x\cup y, \gamma')$ we consider $|\theta(y) - \theta(x)|$ with
$\theta(x) = g(x) \psi(x)$, $g\in C^{L}_{\rm b} (\mathds{R}^d)$,
$\|g\|_{BL}\leq 1$, for which we obtain, cf. (\ref{C81}),
\begin{gather}
  \label{N5Z}
|\theta(y) - \theta(x)| = \psi(x) \psi(y) \left|
\frac{g(y)}{\psi(x)} - \frac{g(x)}{\psi(y)}\right|\\[.2cm] \nonumber = \psi(x)
\psi(y) \left| \frac{g(y)-g(x)}{\psi(y)} + g(y)
\left[\frac{1}{\psi(x)}-\frac{1}{\psi(y)} \right]\right|\\[.2cm]
\nonumber \leq \psi(x) |x-y| + \psi(x)\psi(y) \left||x|^{ d+1}-
|x|^{ d+1} \right|\\[.2cm] \nonumber \leq \psi(x)\left[ |x-y| + \sum_{l=1}^{d+1} {d+1 \choose l} |x-y|^l\right].
\end{gather}
Now we use this in (\ref{D8}) and arrive at
\begin{eqnarray}
  \label{D9}
\left|J^\gamma(\gamma') \right| \leq C_a \Psi(\gamma'), \qquad C_a:=
m^a_1 + \sum_{l=1}^{d+1} {d+1 \choose l}m_l^a.
\end{eqnarray}
Then we use (\ref{D9}) in (\ref{D6}) and obtain
\begin{gather}
  \label{D13}
w^\alpha_u(\gamma) \leq C_a \int_0^u \chi_s^\alpha (\gamma) d s,
\qquad \chi_s^\alpha (\gamma):= \int_{\Gamma_*} \Psi (\gamma')
p^{\alpha}_s  ( \gamma,d \gamma').
\end{gather}
Note that $\chi^\alpha_0(\gamma) = \Psi (\gamma)$. Similarly as in
(\ref{Starr}) we write
\begin{equation}
  \label{D14}
\chi_s^\alpha (\gamma) = \Psi(\gamma) + \int_0^s \left(
\int_{\Gamma_*} (L^\alpha \Psi) (\gamma') p^{\alpha}_v
(\gamma,d\gamma')\right) d v
\end{equation}
Like in (\ref{N16}) one gets
\[
(L^\alpha \Psi) (\gamma)\leq \left|(L^\alpha \Psi)(\gamma) \right|
\leq 2 c_a \Psi (\gamma),
\]
where $c_a$ is as in (\ref{ca}). We use this in (\ref{D14}), take
also into account the definition of $\chi_s^\alpha$ in (\ref{D13})
and obtain
\begin{equation*}
  \chi_s^a (\gamma) \leq \Psi (\gamma) + 2 c_a\int_0^s \chi^a_v(\gamma)
  d v,
\end{equation*}
which by the Gr\"onwall inequality and (\ref{D13}) yields the
following estimate
\begin{equation*}
w^\alpha_u (\gamma) \leq C_a u e^{2 c_a u} \Psi (\gamma).
\end{equation*}
We employ this in the second line of (\ref{Ma50}) and obtain
\begin{equation}
  \label{d16}
 W^\alpha_{u,v} (\gamma) \leq C_a u e^{2 c_a u} q_v^\alpha (\gamma),
 \qquad  q_v^\alpha(\gamma) := \int_{\Gamma_*} \Psi(\gamma')
 \upsilon_* (\gamma,\gamma') p^{\alpha}_v (\gamma, d \gamma').
\end{equation}
Note that $q_0^\alpha(\gamma) = 0$ as $\upsilon_*$ is a metric.
Similarly as in (\ref{D6}) we then get
\begin{equation*}
q_v^\alpha(\gamma) = \int_0^v \left( \int_{\Gamma_*}\left( L^\alpha
\Psi \upsilon_* (\gamma, \cdot)\right) (\gamma') p^{\alpha}_s
(\gamma, d \gamma') \right) d s.
\end{equation*}
Thus, we have to estimate
\begin{eqnarray}
  \label{D18}
  \left|  \left(L^\alpha \Psi \upsilon_* (\gamma,\cdot)\right)(\gamma')\right| & \leq & \sum_{x\in \gamma'}
 \int_{\mathds{R}^d} a(x-y) \bigg{|}\Psi (\gamma'\setminus x \cup y) \upsilon_* (\gamma,\gamma'\setminus x \cup
 y) \\[.2cm] \nonumber & - & \Psi (\gamma') \upsilon_*
 (\gamma,\gamma') \bigg{|} dy  \leq  \sum_{x\in \gamma'}
 \int_{\mathds{R}^d} a(x-y) |\psi(y) - \psi(x) |d y \\[.2cm]
 \nonumber & + & \Psi (\gamma')  \sum_{x\in \gamma'}
 \int_{\mathds{R}^d} a(x-y) \upsilon_* (\gamma',\gamma'\setminus
 x\cup  y) dy \\[.2cm] &\leq & c_a \Psi(\gamma')  +
 C_a \Psi^2(\gamma'), \nonumber
\end{eqnarray}
where we used the same estimate as in as in (\ref{N5Z}). Now we use
(\ref{D18}) in (\ref{d16}) and then plug this into (\ref{Ma51}). In
doing so, we will deal with
\[
\int_{\Gamma_*} p^{\alpha}_s (\gamma, d\gamma') \Pi^\alpha_{t_1} ( d
\gamma) = \Pi^\alpha_{t_1+s} ( d \gamma') = \mu^\alpha_{t_1+s} (d
\gamma'),
\]
that follows by the Chapman-Kolmogorov property of the transition
function (\ref{Markov}), see \cite[page 156]{EK}, and then by Lemma
\ref{C3lm}. Thereafter, we obtain
\begin{gather}
  \label{D19}
\mathcal{W}^\alpha (t_1, t_1 + v , t_1 + u + v) \leq C_a u e^{2
c_au} \int_0^v \left( c_a \mu^\alpha_{t_1+s} (\Psi) + C_a
\mu^\alpha_{t_1+s} (\Psi^2) \right) d  s.
\end{gather}
We recall that $\mu\in \mathcal{P}_{\rm exp}^{\vartheta_0}$ and thus
the correlation functions of $\mu_t^\alpha$ satisfy the estimate in
claim (a) of Proposition \ref{3.3pn} with $\vartheta_t = \vartheta_0
+t$. Then by (\ref{N7}) we get
\begin{gather*}
\mu^\alpha_t (\Psi) \leq \langle \psi \rangle e^{\vartheta_0 +t},
\quad \mu^\alpha_t (\Psi^2) \leq \langle \psi \rangle e^{\vartheta_0
+t} + \langle \psi \rangle^2 e^{2\vartheta_0 +2t}.
\end{gather*}
We use this in (\ref{D19}) and obtain
\[
\mathcal{W}^\alpha (t_1, t_1 + v , t_1 + u + v)\leq C (u+v)^2,
\]
where, for a fixed $\delta>0$, the independent of $\alpha$ constant
can be calculated explicitly for $u+v<\delta$. This yields
(\ref{Ma52}) with $C$ independent of $\alpha$, and hence completes
the whole proof.
\end{proof}

\subsection{Proof of Theorem \ref{3tm}}

For each $\alpha\in (0,1]$, $s\geq 0$ and $\mu\in \mathcal{P}_{\rm
exp}$, by Proposition \ref{MApn} the measure $P^\alpha_{s,\mu}$ on
$\mathfrak{D}_{[s,+\infty)}(\Gamma_*)$ is defined by its finite
dimensional marginals constructed with the use of the transition
function (\ref{Markov}). Namely, for $s\leq t_1< t_2 < \cdots < t_m$
and $\mathbb{A}_1, \dots , \mathbb{A}_m \in \mathcal{B}(\Gamma_*)$,
we have, cf. \cite[eq. (1.10), page 157]{EK},
\begin{gather}
  \label{Novemm}
P^{\alpha}_{s,\mu} \left( (\mathds{1}_{\mathbb{A}_1} \circ
\varpi_{t_1}) \cdots (\mathds{1}_{\mathbb{A}_m} \circ \varpi_{t_m})
\right) = \int_{\Gamma_*^{m+1}} \mathds{1}_{\mathbb{A}_m} (\gamma_m)
p^\alpha_{t_m-t_{m-1}} ( \gamma_{m-1}, d \gamma_{m}) \\[.2cm] \nonumber\times
\mathds{1}_{\mathbb{A}_{m-1}} (\gamma_{m-1})
p^\alpha_{t_{m-1}-t_{m-2}} ( \gamma_{m-2}, d \gamma_{m-1})\cdots
\mathds{1}_{\mathbb{A}_1} (\gamma_1) p^\alpha_{t_1-s} (\gamma_0 , d
\gamma_1)  \mu ( d\gamma_0).
\end{gather}
In particular, for $t\geq s$, this yields
\begin{equation}
  \label{Novem}
P^\alpha_{s,\mu} \circ \varpi^{-1}_t = S^\alpha (t-s) \mu.
\end{equation}
Then the validity of conditions (a) and (b) of Definition \ref{A1df}
follow by (\ref{Novem}) and Lemma \ref{C3lm}. Now we turn to proving
the validity of (c). Let $\sf G$ be as in (\ref{A1200}) with a given
$m\in \mathds{N}$ and $s\leq s_1 < s_2 < \cdots < s_m< t_2$. For a
given $F\in \mathcal{D}(L)$ and $u\in [s_m,t_2]$, we set ${\sf F}_u=
F\circ \varpi_u$,  ${\sf K}_u= (L F)\circ \varpi_u$ and ${\sf
K}^\alpha_u= (L^\alpha F)\circ \varpi_u$. Next, we define
\[
\chi^\alpha_{s_1} (d \gamma) = C_1 F_1(\gamma) \mu^\alpha_{s_1} (d
\gamma) = C_1 F_1(\gamma)\int_{\Gamma_*} p^\alpha_{s_1-s} (\gamma_0,
d \gamma) \mu(d \gamma_0), \quad C^{-1}_1 :=
\int_{\Gamma_*}F_1(\gamma) \mu^\alpha_{s_1} (d \gamma).
\]
By (\ref{Ma60}) and Lemma \ref{C3lm}, and then by claim (iv) of
Proposition \ref{T3pn}, we have that $\chi^\alpha_{s_1} \in
\mathcal{P}_{\rm exp}^{\vartheta_1}$ with $\vartheta_1$ dependent on
$s_1-s$ and the type of $\mu \in \mathcal{P}_{\rm  exp}$, and
independent of $\alpha$ since the norms of $L^{\Delta,\alpha}$ can
be estimated uniformly in $\alpha\in [0,1]$. Then we define
recursively
\begin{equation*}
  \chi^\alpha_{s_l} (d \gamma) = C_l F_l (\gamma) \int_{\Gamma_*} p^\alpha_{s_l-s_{l-1}} (\gamma_0,
d \gamma)\chi^\alpha_{s_{l-1}} (d \gamma_0), \qquad l=2, \dots , m,
\end{equation*}
and obtain $\chi^\alpha_{s_m} \in \mathcal{P}_{\rm
exp}^{\vartheta_m}$ with $\vartheta_m$ independent of $\alpha$.
Thereafter, by (\ref{Novemm}) we conclude that
\begin{equation}
  \label{Novv}
P^\alpha_{s, \mu} ({\sf F}_u {\sf G}) = C P^\alpha_{s_m, \chi_{s_m}}
({\sf F}_u ) = CP^\alpha_{s_m, \chi_{s_m}} ( F\circ \varpi_u ),
\qquad u \geq s_m,
\end{equation}
where $C$ is a normalizing constant, i.e. $C=
P^{\alpha}_{s,\mu}({\sf G})$. Then $P^\alpha_{s, \mu} ({\sf H}) = 0$
follows by the fact that the map $u\mapsto P^\alpha_{s_m,
\chi_{s_m}}\circ\varpi^{-1}_u$ solves the Fokker-Planck equation
(\ref{A12a}) with $L^\alpha$, see (\ref{Novem}). This proves (c),
and hence the family $\{P^\alpha_{s,\mu}: s\geq 0, \mu \in
\mathcal{P}_{\rm exp}\}$ is a unique solution of the corresponding
restricted martingale problem, see Theorem \ref{Au1tm}.

By Lemma \ref{MAlm} and claim (ii) of Proposition \ref{MApn}, for
each $s$ and $\mu$, the family $\{P^\alpha_{s, \mu}:\alpha \in
(0,1]\}$ is relatively weakly compact, and each of its accumulation
points has the same one dimensional marginals, that coincide with
the measures $\mu_t$, see Lemmas \ref{U3lm} and \ref{C3lm}. Let us
show that these accumulation points solve the restricted initial
value martingale problem for $L$. By Lemmas \ref{U3lm} and
\ref{C3lm} one concludes that conditions (a) and (b) of Definition
\ref{A1df} are met, and we thus turn to proving (\ref{A12}). Given
sequence $\{\alpha_n\}_{n\in \mathds{N}}\subset (0,1]$, $\alpha_n
\to 0$ and $s\geq 0$, $\mu \in \mathcal{P}_{\rm exp}$, let
$P^{\alpha_n}_{s,\mu} \Rightarrow P_{s,\mu}$. Let also $\sf G$ in
(\ref{A120}) be as in (\ref{A1200}) with a given $m$, $s_1, \dots ,
s_m$ and $F_j \in \widetilde{\mathcal{F}}$, $j=1, \dots , m$. Set
$C_n = P^{\alpha_n}_{s,\mu} ({\sf G})$. Then the measures
$\nu_{n,u}\in \mathcal{P}(\Gamma_*)$ defined by
\[
\nu_{n,u}(\mathbb{A}) = C^{-1}_n P^{\alpha_n}_{s,\mu} ({\sf G} \cdot
(\mathds{1}_{\mathbb{A}}\circ \varpi_u))= P^{\alpha_n}_{s_m
\chi_{s_m}}( \varpi^{-1}_u(\mathbb{A})) , \qquad u\in [s_m,t_2], \ \
\mathbb{A}\in \mathcal{B}(\Gamma_*).
\]
are in $\mathcal{P}_{\rm exp}^\vartheta$ with $\vartheta$
independent of $n$ and $u\in [s_m,t_2]$ (see (\ref{Novv})). We also
let
\[
\nu_{u}(\mathbb{A}) = C^{-1} P_{s,\mu}( {\sf G}\cdot
(\mathds{1}_{\mathbb{A}}\circ \varpi_u)), \qquad u\in [s_m,t_2], \ \
\mathbb{A}\in \mathcal{B}(\Gamma_*),
\]
with $C = P_{s,\mu} ({\sf G})$. Then $\nu_{n,u}\Rightarrow \nu_u$
for all $u\in [s_m,t_2]$. By Lemma \ref{U40lm} this yields $\nu_u
\in\mathcal{P}_{\rm exp}^\vartheta$, and hence the corresponding
correlation functions satisfy $k_u^{\alpha_n}, k_u\in
\mathcal{K}_{\vartheta}$ for all $u\in [s_m,t_2]$ and $n\in
\mathds{N}$. To prove  $P_{s,\mu}({\sf H})=0$ we rewrite it, cf.
(\ref{A120}),
\begin{equation}
\label{A1201} P_{s,\mu}({\sf F}_{t_2} {\sf G}) - P_{s,\mu}({\sf
F}_{t_1} {\sf G}) - \int_{t_1}^{t_2} P_{s,\mu}({\sf K}_{u} {\sf G})
d u = 0.
\end{equation}
For $u\in [s_m,t_2]$ and $n\in \mathds{N}$, we then set
\begin{gather*}
 a_n (u) = P_{s,\mu}({\sf F}_{u} {\sf G}) - P^{\alpha_n}_{s,\mu}({\sf F}_{u} {\sf
 G}), \\[.2cm] \nonumber  b_n (u) = P_{s,\mu}({\sf K}_{u} {\sf G}) - P^{\alpha_n}_{s,\mu}({\sf K}_{u} {\sf
 G}), \\[.2cm] \nonumber c_n(u) = P^{\alpha_n}_{s,\mu}(({\sf K}_{u} - {\sf K}^{\alpha_n}_{u}){\sf
 G}).
\end{gather*}
Since $P^{\alpha_n}_{s,\mu}({\sf H})=0$, it follows that
\begin{equation}
  \label{A1203}
{\rm LHS}(\ref{A1201}) = \left[a_n (t_2) - a_n(t_1)\right] -
\int_{t_1}^{t_2} b_n (u) du - \int_{t_1}^{t_2} c_n(u) d u=:
I_n^{(1)}+ I_n^{(2)} + I_n^{(3)}.
\end{equation}
By the assumed weak convergence of $P^{\alpha_n}_{s,\mu}$ one
readily gets $a_n (u) \to 0$, which yields $I_n^{(1)} \to 0$ as
$n\to +\infty$. At the same time,
\begin{equation}
\label{A1204}
  b_n(u) = C^{-1}_n\left[\nu_u(LF) - \nu_{n,u}(LF) \right] + (C^{-1}
  - C_n^{-1}) \nu_u (LF).
\end{equation}
Since $C_n \to C >0$, to prove $b_n(u) \to 0$ as $n\to +\infty$ by
(\ref{A1204}) it is enough to show that $\nu_u(LF) -
\nu_{n,u}(LF)\to 0$ for $F\in \mathcal{D}(L)$. To this end we recall
that ${\sf G}= {\sf G}_{m-1} (F_m \circ \varpi_{s_m})$, see
(\ref{A1200}). Set
\begin{gather*}
\tilde{\nu}_{n,u} (\mathbb{A}) = \tilde{C}_n^{-1}
P^{\alpha_n}_{s,\mu} ({\sf G}_{m-1}( \mathds{1}_{\mathbb{A} }\circ
\varpi_u)) , \quad u \in [s_{m-1}, s_m],\\[.2cm]
\tilde{\nu}_{u} (\mathbb{A}) = \tilde{C}^{-1} P_{s,\mu} ({\sf
G}_{m-1}( \mathds{1}_{\mathbb{A} }\circ \varpi_u)).
\end{gather*}
As above, we have that $\tilde{\nu}_{n,u}, \tilde{\nu}_{u} \in
\mathcal{P}_{\rm exp}^{\tilde{\vartheta}}$ for all $n$ and $u$ as
above. Clearly, we may assume that $\vartheta> \tilde{\vartheta}$,
and hence their correlation functions, $\tilde{k}_u$ and
$\tilde{k}^{\alpha_n}_u$, lie in the corresponding
$\mathcal{K}_{\tilde{\vartheta}}$. As in (\ref{Ua}) we then can
write
\begin{equation}
\label{A1208} k_u - k^{\alpha_n}_u = Q_{\vartheta\tilde{\vartheta}}
(u-s_m)\tilde{k}_{s_m} - Q^{\alpha_n}_{\vartheta\tilde{\vartheta}}
(u-s_m)\tilde{k}^{\alpha_n}_{s_m}.
\end{equation}
For $m=1$, $\tilde{k}_{u}$ and $\tilde{k}^{\alpha_n}_{u}$ are the
correlation functions of $\mu_u$ and $\mu^{\alpha_n}_u$, and hence
one may apply Lemma \ref{U2lm}, which yields
\[
\mu_u (L F) - \mu_u^{\alpha_n}(LF) = \langle \! \langle
\tilde{k}_{u} - \tilde{k}^{\alpha_n}_{u},  \widehat{L} \tilde{G}
\rangle \! \rangle  = \langle \! \langle \tilde{k}_{u} -
\tilde{k}^{\alpha_n}_{u}, G\rangle \! \rangle \to 0, \qquad n\to
+\infty,
\]
where $\tilde{G}\in \cap_{\vartheta} \mathcal{G}_\vartheta$ is such
that $F=K\tilde{G}$, see (\ref{WA}) and (\ref{WA1}). Therefore, we
may inductively assume in (\ref{A1208})  that
\[
\langle \! \langle \tilde{k}_{s_m} - \tilde{k}^{\alpha_n}_{s_m},
\widehat{L} \tilde{G} \rangle \! \rangle \to 0,
\]
and obtain
\[
\nu_u(LF) - \nu_{n,u}(LF) = \langle \! \langle k_u - k^{\alpha_n}_u
, \widehat{L} \tilde{G} \rangle \! \rangle \to 0,
\]
by repeating the steps made in the proof of Lemma \ref{U2lm}. This
yields $b_n(u) \to 0$. As already mentioned above, both terms of
$b_n(u)$ are bounded uniformly in $n$ and $u$, that yields in
(\ref{A1203}) $I_n^{(2)} \to 0$.

Let us now turn to $I_n^{(3)}$. As above, we have here
\begin{equation*}
|c_n(u)| = \left|\langle \! \langle k^{\alpha_n}_u, \widehat{L}_n G
\rangle \! \rangle\right| \leq e^\vartheta | \widehat{L}_n
G|_{\vartheta},
\end{equation*}
where $G\in \cap_{\vartheta}\mathcal{G}_\vartheta$ is such that $F =
KG$, see (\ref{WA}) and (\ref{WA1}), and $\widehat{L}_n$ is obtained
by replacing $a(x-y)$ in (\ref{A9a}) by
\[
a_n (x,y) = a(x-y) (1- \psi_{\alpha_n})(x) = \alpha_n
\frac{a(x-y)|x|^{d+1}}{1 + \alpha_n |x|^{d+1}} \leq \alpha_n
a(x-y)|x|^{d+1} =: \alpha_n \hat{a}(x,y).
\]
Proceeding as in obtaining (\ref{A11}) we then get, see (\ref{C8}),
\[
| \widehat{L}_n G|_{\vartheta} \leq \frac{2 \alpha_n
m^a_{d+1}}{e(\vartheta' - \vartheta)} \exp\left(e^\vartheta \langle
\phi \rangle \right) |G|_{\vartheta'}.
\]
Here $\vartheta'$ can be an arbitrary number since $G\in
\cap_{\vartheta'} \mathcal{G}_{\vartheta'}$, see (\ref{WA1}). This
yields $I_n^{(3)}\to 0$ as $n\to +\infty$ (and hence $\alpha_n \to
0$), which by (\ref{A1203}) implies (\ref{A1201}). Therefore,
 the proof of Theorem \ref{3tm} is completed.

\subsection{Proof of Theorem \ref{1tm}} Claim (a) follows by Theorem
\ref{3tm} (existence) and Theorem \ref{2tm} (uniqueness). The
validity of (b) is then a standard fact, cf. \cite[Theorem 5.1.2,
claim (iv), page 80]{Dawson}. To prove (c), we proceed as follows.
By construction, the law of $X(t)$ is $\mu_t \in \mathcal{P}_{\rm
exp}$; hence, $X(t)\in \breve{\Gamma}_*$ with probability one, see
Lemma \ref{Lenalm}. Let $\{D_k\}_{k\in \mathds{N}}$ and
$\{\Gamma_{*,k} \}_{k\in \mathds{N}}$ be the collections of balls
and sets, respectively, used in the proof of Lemma \ref{Lenardlm},
see (\ref{Psi6}). As we show there, each $\Gamma_{*,k}$ is an open
subset of  $\Gamma_*$, and $\breve{\Gamma}_* = \cap_{k}
\Gamma_{*,k}$. For $H$ as in (\ref{Hk}) and $k\in \mathds{N}$, we
set $H_k (\gamma) = H (\gamma_k) = H(\gamma\cap D_k)$. Then
$H_k(\gamma)< \infty$ for $\gamma \in \Gamma_{*,k}$. For $N\in
\mathds{N}$ and $s\geq 0$, let us consider the following stopping
time
\begin{equation*}
  T^k_N = \inf\{t\geq s: H_k (X(t) )> N \},
\end{equation*}
cf. \cite[page 180]{EK}, and then set $T^k_N \wedge t = \min\{T^k_N;
t\}$ and $Z(t) = \lim_{N\to +\infty} X(T^k_N \wedge t)$, which
exists as $T^k_N \leq T^k_{N+1}$. Let $\varPhi^m_\tau \in
\mathcal{D}({L})$ be the same as in (\ref{Ma16}). Then
\begin{equation*}
  \varPhi_\tau^m (X(t)) - \int_s^t ({L} \varPhi^m_\tau
  )(X(u)) d u
\end{equation*}
is a right-continuous martingale. Let $\tilde{\mu}_t$ be the law of
$Z(t)$ and $T^k = \lim_{N\to +\infty} T^k_N$. Similarly as in
\cite[page 180]{EK}, for each $t>s$ by the optional sampling
theorem, we can write
\begin{equation*}
  E\left[\varPhi_\tau^m (X(T^k_N \wedge t)) \right] = E\left[\varPhi_\tau^m (X(s))
  \right] + E \left[\int_s^{T^k_N\wedge t} ({L} \varPhi^m_\tau
  )(X(u)) d u \right],
\end{equation*}
which after passing to the limit $N\to +\infty$ yields
\begin{gather*}
\tilde{\mu}_t (\varPhi_\tau^m) = \mu (\varPhi_\tau^m) + E
\left[\int_s^{T^k \wedge t} ({L} \varPhi^m_\tau
  )(X(u)) d u \right] \\[.2cm] \nonumber \leq  \mu (\varPhi_\tau^m)
  + E
\left[\int_s^{T^k\wedge t} \left| ({L} \varPhi^m_\tau
  )(X(u))\right| d u \right]  \\[.2cm] \nonumber \leq \mu (\varPhi_\tau^m) +
E \left[\int_s^{ t} \left| ({L} \varPhi^m_\tau
  )(X(u))\right| d u \right] \leq \mu (\varPhi_\tau^m)
  + \int_s^t \mu_u (\varPhi^m_{\tau,1}) d u,
\end{gather*}
where $\mu_u = {P}_{s,\mu}\circ \varpi^{-1}_u$ is the law of $X(u)$.
Note that in the last line we used (\ref{Bog}). By this estimate and
(\ref{Ma18}) (with $\mu_0 = \mu$) we then get the following
\begin{equation*}
  \tilde{\mu}_t (\varPhi_\tau^m) \leq \sum_{n=0}^{+\infty}
  \frac{(t-s)^n}{n!} \mu (\varPhi_{\tau,n}^m).
\end{equation*}
Now we proceed as in (\ref{Ma22}) and arrive at
\begin{equation*}
  \lim_{\tau \to 0} \tilde{\mu}_t (\varPhi_\tau^m) \leq(\varkappa
  e^{2(t-s)})^m
  \|\theta\|^m_{L^1(X)},
\end{equation*}
holding for all $m\in \mathds{N}$ and $t-s < \rho_\varepsilon$. Here
$\varkappa$ is the type of $\mu$. This yields that $\tilde{\mu}_t
\in \mathcal{P}_{\rm exp}$ for such $t$, and hence $Z(t) \in
\breve{\Gamma}_*$ almost surely, implying $T^k
>t$. Now we fix $v< s + \rho_\varepsilon$, repeat this procedure
with the martingale
\[
  \varPhi_\tau^m (X(t+v)) - \int_{s+v}^{t+v} ({L} \varPhi^m_\tau
  )(X(u)) d u
\]
and eventually conclude that $T^k >t$ for all $t$, and hence almost
all sample paths of $X$ remain in
$\mathfrak{D}_{[s,+\infty)}(\Gamma_{*,k})$, holding for every $k$.
Since $\breve{\Gamma}_* = \cap_{k} \Gamma_{*,k}$, this yields that
these paths remain in
$\mathfrak{D}_{[s,+\infty)}(\breve{\Gamma}_*)$, cf. \cite[Proof of
Proposition 3.10, page 180, 181]{EK}, which complete the proof.

\vskip.6cm
\paragraph{\bf Acknowledgement} This work was supported by the Deutsche
Forschungsgemeinschaft through SFB 1283 ``Taming uncertainty and
profiting from randomness and low regularity in analysis,
stochastics and their applications" that is acknowledged by the
authors. Yuri Kozitsky thanks Lucian Beznea, Oleh Lopushansky and
Yuri Tomilov for discussing some of its aspects. He thanks also
Lucian Beznea and BIT DEFENDER for warm hospitality and financial
support during his stay in Bucharest in April 2019, where a part of
this work was done. Last but not least, the authors are cordially
grateful to the referee for valuable and favorable suggestion that
helped to improve the quality of the final version of this work.

\section*{Appendix}

Here we prove (\ref{Bog}), (\ref{Ma6}) and (\ref{Ma7a}). For $n=2$,
(\ref{Bog}) is just (\ref{Ma4}) with $\varPhi^{m}_{\tau,2}$ given by
(\ref{Ma6}) with  $w_1(m,2) = (m+1)^2 - m^2 = 2m +1$, $w_2(m,2)=1$,
see (\ref{Ma9}) and (\ref{Ma8}). Assume then that
$\varPhi^{m}_{\tau,n}$ is as in (\ref{Ma6}). By (\ref{CV1}),
similarly as in (\ref{TH7}), we get
\begin{eqnarray*}
\left|L V_\tau(c;\gamma) \right| & \leq & \sum_{j=1}^m
\widehat{F}_\tau^{m, \theta^{q_1}, \dots , \theta^{q_{j-1}} ,
\theta^{q_j} + 1, \theta^{q_{j+1}}, \dots \theta^{q_m}}(\gamma) ) +
 \tau c_a^{n + 1} \widehat{F}_\tau^{m+1}(\gamma).
\end{eqnarray*}
Here we have taken into account that $\bar{c}_\theta=1$, and also
$\theta^q (x) \leq c^q_a \psi(x)$, see (\ref{TH3}) and (\ref{C81}).
In a similar way, by (\ref{TH8}) we obtain
\[
\left| L \widehat{F}_\tau^m (\gamma) \right| \leq m c_a
\widehat{F}_\tau^m (\gamma) + \tau c_a \widehat{F}_\tau^{m
+1}(\gamma).
\]
Now we use both this estimates in (\ref{Ma6}) and obtain
\begin{eqnarray}
  \label{Zu}
 \left| L \varPhi^{m}_{\tau,n} (\gamma)\right|  & \leq &  \sum_{c\in \mathcal{C}_{m,n}}
C_{m,n}(c) \bigg{(} c_0 V_\tau(c_0-1, c_1+1, c_2, \dots, c_k,
\dots;\gamma)
\\[.2cm]\nonumber & +& c_1 V_\tau(c_0, c_1-1, c_2+1, \dots, c_k,
\dots;\gamma)+\cdots + \\[.2cm] \nonumber & + & c_n V_\tau(c_0, c_1, \dots, c_n
-1, c_{n+1}+1, \dots;\gamma)\bigg{)} \\[.2cm] \nonumber & + & \tau
c_a^{n+1} m^n \widehat{F}^{m+1}_\tau (\gamma) +  c_a^{n+1} \bigg{(}
\sum_{k=1}^n \tau^k (m+k) w_{k}(m,n)\widehat{F}^{m+k}_\tau (\gamma)
\\[.2cm] \nonumber &+ & \sum_{k=2}^{n+1} \tau^k (m+k) w_{k-1}(m,n)\widehat{F}^{m+k}_\tau
(\gamma) \bigg{)}.
\end{eqnarray}
If one takes into account the recurrence formulas in (\ref{Ma8}),
the latter two lines of the right-hand side of (\ref{Zu}) convert
into the second term of (\ref{Ma6}) written for
$\varPhi_{\tau,n+1}^m$. Thus, it remains to prove that the first
three lines of (\ref{Zu}) yield the first term of (\ref{Ma6})
written for $\varPhi_{\tau,n+1}^m$. Note that therein the summands
corresponding to $c_j=0$ vanish automatically since we multiply them
by zero in this case. Assuming that a given $c_j \neq 0$ we can
write the corresponding summand in (\ref{Zu}), denoted $S_j^{n+1}$,
as follows, see the second line in (\ref{Ma6}),
\begin{eqnarray}
\label{Zux}
 S_j^{n+1} & = & \frac{m! n! (j+1)! (c_{j+1}+1)}{ c_0!
\cdots (c_j-1)! (c_{j+1} + 1)! \cdots (0!)^{c_0} \cdots (j!)^{c_j-1}
j! ((j+1)!)^{c_{j+1}+1}\cdots
} \\[.2cm] \nonumber & \times & V_\tau (c_0, c_1, \dots , c_j-1, c_{j+1}+ 1,
\cdots ;\gamma) \\[.2cm] & = & \frac{j+1}{n+1} c'_{j+1} C_{m,n+1}
( c') V_\tau (c'; \gamma), \qquad c'\in \mathcal{C}_{m,n+1},
\nonumber
\end{eqnarray}
where $c' = (c_0, \dots , c_{j}-1, c_{j+1}+1, \dots)$. To get
convinced that $c'$ is indeed in $\mathcal{C}_{m,n+1}$ one computes
the corresponding sums, cf. (\ref{CV}), that yields $c_0 + \cdots +
c_j -1 + c_{j+1}+1 +\cdots  = c_0 + \cdots + c_j + c_{j+1} + \cdots
=m$, and $c_1 + \cdots + j (c_j-1) + (j+1) (c_{j+1}+1) + \cdots =
c_1 +\cdots + j c_j + (j+1)c_{j+1} +\cdots - j + j+1 = n+1$. Then we
rewrite each summand in the first three lines of (\ref{Zu}) as in
(\ref{Zux}) and observe that the corresponding $c'$ runs over the
whole $\mathcal{C}_{m,n+1}$ when $c$ runs through
$\mathcal{C}_{m,n}$. Then these three lines, denoted $S^{n+1}$, take
the following form
\begin{eqnarray}
  \label{Zu1}
  S^{n+1} & = & \sum_{c'\in \mathcal{C}_{m,n+1}} \bigg{(} \frac{1}{n+1} \sum_{j=1}^{n+1} j
  c'_j\bigg{)} C_{m,n+1} (c') V_\tau (c';\gamma) \\[.2cm] \nonumber
  & = & \sum_{c'\in \mathcal{C}_{m,n+1}} C_{m,n+1} (c') V_\tau
  (c';\gamma),
\end{eqnarray}
where we have taken into account that $\sum_{j} j c_j = n+1$, see
(\ref{CV}). This completes the proof of (\ref{Bog}) and (\ref{Ma6}).
It then remains to prove (\ref{Ma7a}). For $n=1$,
$\mathcal{C}_{m,1}$ is a singleton consisting of $c= (m-1,1,
0,\dots)$, which yields
\[
\sum_{c\in \mathcal{C}_{m,1}} C_{m,1}(c) = \frac{m!}{ (m-1)! 1!} =
m.
\]
Now we set in the second line of (\ref{Zu1}) $V_\tau (c';\gamma)
\equiv 1$ and calculate $S^{n+1}$ with this $V_\tau$, which is equal
to the first three lines of (\ref{Zu}). That is,
\begin{eqnarray*}
\sum_{c'\in \mathcal{C}_{m,n+1}} C_{m,n+1}(c') = \sum_{c\in
\mathcal{C}_{m,n}} C_{m,n}(c) \bigg{(} c_0 + c_1 +\cdots +
c_n\bigg{)} = m \sum_{c\in \mathcal{C}_{m,n}} C_{m,n}(c),
\end{eqnarray*}
where we once again have used the first equality in (\ref{CV}). Now
(\ref{Ma7a}) is obtained from the latter by the induction in $n$.

\end{document}